\documentclass[a4paper,11pt,twoside]{article}
\usepackage{geometry}                
\usepackage{graphicx}

\usepackage{lscape}
\usepackage{amsfonts}
\usepackage{amsmath}
\usepackage{amssymb}
\usepackage{mathdots}

\usepackage{epstopdf}
\usepackage{url}

\newtheorem{theorem}{Theorem}[section]

\newtheorem{corollary}[theorem]{Corollary}

\newtheorem{lemma}[theorem]{Lemma}

\newtheorem{proposition}[theorem]{Proposition}
\newtheorem{remark}[theorem]{Remark}

\newenvironment{proof}[1][Proof]
  { \begin{description}
      \item\textbf{Proof.} }
   {\hfill\rule{2.1mm}{2.1mm}
     \end{description} }

\title{On connection coefficients,
zeros and interception points 
of some perturbed of arbitrary order of the Chebyshev polynomials of second kind}

\author{Z\'elia da {\sc Rocha}\\
Departamento de Matem{\'a}tica - CMUP \\
Faculdade de Ci{\^e}ncias da Universidade do Porto\\
Rua do Campo Alegre n.687, 4169 - 007 Porto, Portugal\\
Email: mrdioh@fc.up.pt
}

\date{October 2017}                                           

\begin{document}

\maketitle

\begin{abstract}
 
Orthogonal polynomials satisfy a recurrence relation of order two, where appear two coefficients. If we modify one of these coefficients at a certain order, we obtain a perturbed orthogonal sequence. In this work we consider in this way some perturbed of Chebyshev polynomials of second kind and we deal with the problem of finding the connection coefficients that allow to write the perturbed sequence in terms of the original one and in terms of the canonical basis. From the connection relations obtained and from two other relations, we deduce some results about zeros and interception points of these perturbed polynomials.  All the work is valid for arbitrary order of perturbation. 

\end{abstract}

 


\tableofcontents

\section{Introduction}

Perturbed orthogonal polynomials are obtained by modifying, in a certain way, a finite number of elements of the two sequences of coefficients presented in the linear recurrence relation of order two satisfied by these polynomials. In this work, we focus our attention on elementary modifications of only one of these two coefficients at a unique specified order $r$, the so-called, in literature, generalized co-recursive or co-dilated polynomials. 
During the last decades, this subject have interested several authors. We would like to cite,
with respect to the co-recursive case ($r=0$), the references \cite{Chihara_0,Chihara,Ifantis_1995,Letessier_1994,Ronveaux_1995,Slim_1988}; concerning the co-dilated situation, we consider \cite{Erb_2015,MAROMEJRI}; for the co-modified sequences refer to \cite{Dini_1988,Dini_1989,Ronveaux_1990}; for the generalized co-polynomials see \cite{Castillo_2015,WKoepf_2004,Marcellan_1990}. 
Also, we call the attention to the general articles \cite{Peherstorfer_1992,Zhedanov_1997}.
We notice that perturbed orthogonal polynomials have some applications, which motivate further their study \cite{Erb_2015,Ismail_1990,Ismail_2005,Letessier_1994,Marcellan_1990,Slim_1988}. 
This type of perturbation is a transformation that can promote a significative change of properties of the original sequence, but the second degree character is preserved by it \cite{MARO 91, MARO 95}. The {\it PSDF\/} algorithm (see \cite{ZeliadaRocha2015_1}) allows to explicit some semi-classical properties of perturbed second degree forms leading, in special, to the second order linear differential equation.
It is well known that the four Chebyshev sequences \cite{Gautschi_1992,Gautschi_2004,Mason_2003} are the most important cases of second degree forms \cite{MARO 95_1}.  In particular, for the purposes of perturbation, the family of second kind is the most simple among them, because it is self-associated \cite{MaroniTounsi2004}, therefore it is often taken as study case in the mentioned literature. There are some specific works about perturbed Chebyshev families, namely \cite{MAROMEJRI} about the co-dilated case of the second kind form, and \cite{MARO 95_1,GSansigre} concerning all the four forms.  Also, in \cite{ZeliadaRocha2016_1},
we present some semi-classical properties obtained with {\it PSDF\/} \cite{ZeliadaRocha2015_1} for the second kind sequence corresponding to the complete perturbation of order 1 and the perturbation of 2 by dilatation, generalizing most known results. In \cite{ZeliadaRocha2015_1}, we give similar properties for the co-recursive and co-dilated sequences of order 3. 

In the present work, we focus our attention on connection coefficients and some consequences related to zeros and interception points valid for any order $r$ of perturbation of the second kind Chebyshev sequence. 

About connection coefficients for orthogonal polynomials the reader can find an extensive  bibliography in \cite{Tcheutia_2014}.
In fact, the literature on this subject is vast and a wide variety of methods have been developed using several techniques. Here, we refer mainly to the following references
\cite{Abd-Elhameed_2015_1,Askey2,Ronveaux_2015_1,WKoepf2,Comptet, Ronveaux4,Lewanowicz1,WKoepf_1,Riordan1,Riordan2,Tcheutia_2013b}.
Zeros of orthogonal polynomials is another widely discussed subject due to its applications in several problems of applied sciences \cite{Szego_1939} and their crucial role in quadrature formulas \cite{Gautschi_2004}. In particular, several authors have dedicated attention to the relationship between the zeros of perturbed polynomials and the zeros of the original sequences finding results about interlacing and monotonicity behaviors and distribution functions for co-recursive and co-dilated cases, see \cite{Chihara_0,Ifantis_1995,Leopold_1998,Marcellan_1990,Ronveaux_1995, Szego_1939,Slim_1988}, and more recently \cite{Castillo_2015, Castillo_2015_1,Erb_2015}. 

Let us summarize the contents of this article.
In Section 2, we present the theoretical background concerning  orthogonal polynomials \cite{Chihara,MARO 91}, perturbation \cite{Chihara,Marcellan_1990,MARO 91}, connection coefficients \cite{ZeliaPascal3,ZeliaPascal2013_0} and properties of the four families of Chebyshev polynomials \cite{Mason_2003}.  It is natural to consider the first, third and four kinds  Chebyshev families as perturbed of the sequence of second kind \cite{ZeliadaRocha2016_1}. From this point on, we shall deal only with the second kind Chebyshev family and its elementary perturbations of order $r$ by translation and by dilatation \cite{ZeliadaRocha2015_1}.
The contribution of this article is focused on finding explicit results about connection coefficients and some of their consequences valid for any order of perturbation. We think that this point is important, because thus one can choose the parameter of perturbation and its order in such a way a prescribed behavior occurs, for example, the positiveness of connection coefficients or the location of certain zeros or interception points.
Section 3 is dedicated to the connection coefficients that allow to write the perturbed family in terms of the original one (Section 3.1) and in terms of the canonical basis (Section 3.2). 
We started this study from some tables of the first connection coefficients, for fixed values of the order $r$ of perturbation, recursively computed by the symbolic software 
{\it CCOP - Connection Coefficients for Orthogonal Polynomials} \cite{ZeliaPascal2013_0,ZeliaPascal2013_1,ZeliaPascal2013_2}.\footnote{{\it CCOP} is written in the {\it Mathematica$^{\circledR}$} language and is available in the library {\sc Numeralgo} of {\sc Netlib} (\url{http://www.netlib.org/numeralgo/}) as na34 package.} 
We realize that the connection coefficients are constant by diagonal and have very simple expressions, therefore it was easy to infer closed formulas valid for any $r$. Demonstrations of these formulas will be done by induction.  
From the connection relations deduced in Section 3.1 and from the connection coefficients of the second kind Chebyshev family in terms of the canonical basis, we deduce in Section 3.2, the connection coefficients of perturbed polynomials in terms of the canonical basis. 
Section 4 concerns some results about zeros and interception points of perturbed polynomials.
We begin by presenting the Hadamard--Gershg\"orin location of zeros. We determine the values of the parameters of perturbation for which there are zeros outside the interval $[-1,1]$. After that, results are deduced from the connection relations previously obtained and from two other relations. 
We point out the fact that the behavior of perturbed polynomials at the origin are related to the parity of $r$. 
Depending on the values of the parameters of perturbation, the smallest and the greatest zeros of perturbed polynomials have a special location with respect to the extremal zeros of Chebyshev polynomials of the same degree.  
Perturbed polynomials with fixed degree
and different values of parameters intercept each other at the zeros of two other Chebyshev polynomials with prescribed degrees: these points are stable, they do not depend on perturbation.  Interceptions points can be simple or double depending on the degrees of polynomials and their relationship with $r$. We distinguish the interceptions points that are common zeros. In the last section, we present some graphical representations in order to illustrate results given herein about zeros and interception points. In fact, these graphs were the source of inspiration to the development of this study. The reasonings and arguments employed in the proofs are similar for the translation and the dilatation cases, but the dilatation one is more simple, because perturbed sequences are symmetric. 
We remark that zeros of perturbed Chebyshev polynomials satisfy some interlacing and monotonicity properties, not studied here, that can be the subject for a forthcoming article.

\section{Theoretical background} \label{section2}%

\subsection{Connection coefficients for perturbed orthogonal polynomials} \label{subsection2_1}%

Let $\mathcal{P}$ be the vector space of polynomials with coefficients in $\mathbb{C}$ and let
$\mathcal{P}^{\prime}$ be its topological dual space. We denote by $\langle u,p\rangle$ the effect of $u\in\mathcal{P}^{\prime }$ on $p\in\mathcal{P}$. In particular,
$\langle u,x^{n}\rangle:=\nolinebreak\left( u\right) _{n},n\geq 0$, represent the {\it moments} of $u$.
A form $u$ is {\it normalized} if its first moment is unit, e.g., $(u)_0=1$.
Let $\{P_{n}\}_{n\geq 0}$ be a {\it monic polynomial sequence} (MPS) with $\deg P_n=n,\ n\geq 0$, e.g., $P_n(x)=x^n+\ldots$. 
Then there exists a unique sequence $\{u_n\}_{n\geq0}$, $u_n\in\mathcal{P}^{\prime}$, called the {\it dual sequence} of $\{P_{n}\}_{n\geq 0}$, such that  $<u_{n},P_{m}>=\delta_{n,m}$, $n,m\ge 0$. The form $u_0$ is called the {\it canonical form} of  $\{P_{n}\}_{n\geq 0}$, it is normalized, e.g., $(u_0)_0=1$. 

A form $u$ is said {\it regular} \cite{MARO 93, MARO 94} if and only if there exists a polynomial sequence $\{P_{n}\}_{n\geq 0}$, such that:
\begin{eqnarray}
\left\langle u,P_{n}P_{m}\right\rangle &=& 0\ ,\ n\neq m\ ,\ n,m\geq 0\ ,
\label{cond.Ort.} \\
\left\langle u,P^2_{n}\right\rangle=k_n &\neq &0\ ,\ n\geq 0\ .
\label{cond.Ort.Reg.}
\end{eqnarray}

Consequently $\deg P_n=n$, $n\geq 0$, and any $P_n$ can be taken monic, then $\{P_{n}\}_{n\geq 0}$ is called a {\it monic orthogonal polynomial sequence} (MOPS) with respect to $u$.
Necessarily $u=(u)_0 u_{0}$, $(u)_0 \not=0$. If we take $u$ normalized, then $u=u_0$.  In this work, we will always consider MOPS and normalized forms.
The identities (\ref{cond.Ort.}) are called the {\it orthogonality conditions} and (\ref{cond.Ort.Reg.}) are the {\it regularity conditions}.

The sequence $\{P_{n}\}_{n\geq 0}$ is regularly orthogonal with respect to $u$
if and only if  \cite{ MARO 93, MARO 94} there are two sequences of coefficients
$\{\beta_{n}\}_{n\geq 0}$ and $\{\gamma_{n+1}\}_{n\geq 0}$, with $\gamma_{n+1}\neq 0,\;n\geq0$, such that, $\{P_{n}\}_{n\geq 0}$ satisfies
the following initial conditions and linear recurrence relation of order 2
\begin{eqnarray}
&&P_{0}(x)=1,\quad P_{1}(x)=x-\beta _{0},  \label{icrecOrto}\\
&&P_{n+2}(x) =(x-\beta _{n+1})P_{n+1}(x)-\gamma
_{n+1}P_{n}(x),\quad n\geq 2.\label{recOrto}
\end{eqnarray}

Furthermore, the recurrence coefficients $\{\beta_{n}\}_{n\geq 0}$ and $\{\gamma_{n+1}\}_{n\geq 0}$ satisfy:
\begin{eqnarray}
\beta _{n} &=&\frac{\left\langle u,xP_{n}^{2}(x)\right\rangle }{k_n},\quad n\geq 0, \notag \\
 \notag\\
\gamma _{n+1} &=&\frac{k_{n+1}}{k_n},\quad n\geq 0.  \label{gamas}
\end{eqnarray}
We remark that, from (\ref{icrecOrto}) and (\ref{gamas}), the regularity conditions (\ref{cond.Ort.Reg.}) are equivalent to the conditions $\gamma_{n+1}\neq 0,\ n\geq0$.
As usual, we suppose that, $\beta_{n}=0$, $\gamma_{n+1}=0$ and $P_{n}(x)=0$,  for all$n<0$.

The MOPS $\{P_n\}_{n\geq 0}$ is real if and only if $\beta_{n}\in \mathbb{R}$ and $\gamma_{n+1}\in \mathbb{R}-\{0\}$. These conditions are equivalent to $(u)_n\in \mathbb{R},\ n\geq 0$ and $u$ is real. If, in addition, we suppose that $\gamma_{n+1}>0$, $n\geq 0$, then $u$ is {\it positive definite}, because this condition is equivalent to $<u,p>>0$, $\forall p\in {\cal P}$:  $p\not\equiv 0$, $p(x)\geq0$, $x\in \mathbb{R}$ \cite{Chihara,MARO 91}.

A form $u$ is {\it symmetric} if and only if
$(u)_{2n+1}=0,\ n\geq 0.$
A polynomial sequence, $\{P_{n}\}_{n\geq 0}$, is symmetric if and only if
$P_n(-x)=(-1)^nP_n(x), \ n\geq 0.$
If   $\{P_{n}\}_{n\geq 0}$ is a MOPS with respect to $u$, then these conditions are equivalent to $\beta_n=0$, $n\geq 0$ \cite{Chihara,MARO 91}.


The {\em $r$th-associated sequence of a MOPS $\{P_{n}\}_{n\ge 0}$} satisfying (\ref{icrecOrto})-(\ref{recOrto}) is a MOPS $\{P^{(r)}_{n}\}_{n\ge 0}$ whose recurrence coefficients are given by \cite{Chihara,MARO 91}
\begin{equation}\label{rcoefass}
\beta_{n}^{(r)}=\beta_{n+r}\quad,\quad\gamma_{n+1}^{(r)}=\gamma_{n+1+r}\quad,\quad n\;,\;r\ge 0\;.
\end{equation}

Let us consider two {\em perturbed sequences} of a MOPS  $\{P_n\}_{n\geq0}$ satisfying (\ref{icrecOrto})-(\ref{recOrto}), obtained by modifying only one of its recurrence coefficients at order $r$. We shall modify $\beta_{r}$ {\em by translation} and $\gamma_{r}$ {\em by dilatation} by means of two parameters $\mu_r$ and $\lambda_r$. Thus, the {\em $r$th-perturbed sequence by translation}, for $r\geq 0$, noted by 
$\{P^{t}_{n}(\mu_{r};r)(x)\}_{n\geq 0}$,
is a MOPS with the following recurrence coefficients
\begin{equation} \label{rcoefftra}
\beta^{t}_{r}= \beta_{r}+\mu_{r}\ ,\ \beta^{t}_{n}= \beta_{n}\ ,\ n\neq r\ ,\ n\geq 0\quad ;\quad \gamma^{t}_{n+1}=\gamma_{n+1}\ ,\  n\geq 0\ .
\end{equation}
When $r=0$, we recover the so-called {\em co--recursive} sequence \cite{Chihara_0,Chihara}, for which only the initial polynomial $P_1(x)$ is perturbed becoming $P^{t}_1(\mu_{0};0)(x)=x-\beta_0-\mu_0$. 
The {\em $r$th-perturbed sequence by dilatation}, for $r\geq 1$, noted by
$\{P^{d}_{n}(\lambda_{r};r)(x)\}_{n\geq 0}$,
is a MOPS with the following recurrence coefficients
\begin{equation} \label{rcoeffdil}
\beta^{d}_{n}= \beta_{n}\ ,\ n\geq 0\quad ;\quad  \gamma^{d}_{r}= \lambda_{r}\gamma_{r}\ ,\ \lambda_{r}\neq 0,1\ , \ \gamma^{d}_{n+1}=\gamma_{n+1}\ ,\ n\neq r-1,\ n\geq 0\ .
\end{equation}
We note by $\tilde u:=\displaystyle u^t\left(\mu_r;r\right)$ and $\tilde u:=\displaystyle u^d\left(\lambda_r;r\right)$, the forms with respect to which these families are orthogonal.
In literature, these sequences are often designated as {\em $r$th-generalized co-recursive} and {\em $r$th-generalized co-dilated} polynomials and both situations as {\it $r$th-generalized co-polynomials} \cite{Marcellan_1990}. Also, they are particular cases of a general perturbation of order $r$ defined in \cite{MARO 91}.
Perturbation by translation destroys the symmetry, but does not change the positive definiteness character of the original sequence. Perturbation by dilatation does not change the symmetry; moreover if $\lambda_r>0$, it preserves also the positive definiteness character.

The canonical sequence $\{X_{n}\}_{n\geq 0}$,  $X_{n}(x)= x^{n}$, is orthogonal with respect to the Dirac measure $\delta_0$,
$\left\langle \delta_0, f\right\rangle=f(0),\ {\rm defined\ by\ the\ moments} \left(\delta_0\right)_n={\delta}_{0,n},\ n\geq 0$. 
This sequence is not regularly orthogonal, since it satisfies (\ref{icrecOrto})-(\ref{recOrto}) with recurrence coefficients 
\begin{equation}\label{canonicalrr}
\beta_{n}=0\ ,\ \gamma_{n+1}=0\ ,\ n\geq 0\ .
\end{equation}

Given any two MPS $\{P_{n}\}_{n\geq 0}$ and $\{{\tilde P}_{n}\}_{n\geq 0}$, not necessarily orthogonal, the coefficients that satisfy the {\em connection relation} (CR)
\begin{equation}\label{PntildePnu}
{\tilde P}_n (x)=\sum_{m=0}^{n}\lambda_{n,m}^{{\tilde P} P}{P}_m(x),
\ n\geq 0,
\end{equation}
are called the {\em connection coefficients} (CC) $\lambda_{n,m}:=\lambda_{n,m}^{{\tilde P} P}:=\lambda_{n,m}({\tilde P}\leftarrow P)$.
It is obvious that these coefficients exist and are unique, because the polynomial sequences are linearly independent. 
In the case $P_m(x)= X_m(x)=x^m$, $m\geq 0$, we have 
\begin{eqnarray}
{\tilde P}_n(x)=\sum_{m=0}^{n} \lambda^{{\tilde P}X}_{n,m}x^m, \label{PntildeXn}
\end{eqnarray}
with $\lambda^{{\tilde P}X}_{n,m}:=\lambda_{n,m}({\tilde P}\leftarrow X)$.
We shall consider also the Vi\`ete's formulas \cite{Weisstein_2015} that establish a relationship between $\lambda^{{\tilde P}X}_{n,m}$  and the zeros 
$\{\xi_{m}^{(n)}\}_{m=1(1)n}$ of ${\tilde P}_n$, in particular
\begin{eqnarray}
\lambda^{{\tilde P}X}_{n,n-1}=-\sum_{m=1}^{n}\xi_{m}^{(n)}\quad, \quad 
\lambda^{{\tilde P}X}_{n,0}=(-1)^n\prod_{m=1}^{n}\xi_{m}^{(n)}\ , \ n\geq 1\ . \label{CC_Can_Zeros_Pn}
\end{eqnarray}

Let us suppose that the MPS $\{P_{n}\}_{n\geq 0}$ is orthogonal with respect to $u$, then multiplying both members of (\ref{PntildePnu}) by $P_k(x)$, applying $u$ and taking into account 
(\ref{cond.Ort.}) and (\ref{cond.Ort.Reg.}), we obtain \cite[p.295]{ZeliaPascal3} 
\begin{eqnarray}
\lambda_{n,m}^{{\tilde P} P} & = & \frac{\left\langle u,{\tilde P}_n P_m\right\rangle}{k_n}\  ,
\ 0\leq m\leq n\ , \ n\geq 0\ . \label{linkcoeffnm}
\end{eqnarray}
In addition, let us suppose that the MPS $\{{\tilde P}_{n}\}_{n\geq 0}$ is orthogonal with respect to $\tilde u$, and that $\{P_{n}\}_{n\geq 0}$ and $\{{\tilde P}_{n}\}_{n\geq 0}$ are given by  their recurrence coefficients
$\{\beta_{n}\}_{n\geq 0}$, $\{\gamma_{n+1}\}_{n\geq 0}$ and $\{{\tilde \beta}_{n}\}_{n\geq 0}$, $\{{\tilde\gamma}_{n+1}\}_{n\geq 0}$, respectively. In (\ref{linkcoeffnm}), using
(\ref{recOrto}) and (\ref{gamas}) for both sequences, it is possible 
to demonstrate that the CC fulfill the following {\em boundary} and {\em initial conditions} and {\em general recurrence relation} \cite[pp.295-296]{ZeliaPascal3} (see, also, \cite{ZeliaPascal2013_0}) 
\begin{eqnarray}
&& \lambda_{n,m}^{{\tilde P} P} = 0\ , \ n< 0 \text{\ or\ } m<0 \text{\ or\ } m >  n\ ; \
 \lambda_{n,n}^{{\tilde P} P} =1\ , \ n\geq 0\ ; \
 \lambda^{{\tilde P} P}_{1,0}=\beta_0-{\tilde\beta}_0\ , \label{linkcoeffnms} \\
&& \lambda^{{\tilde P} P}_{n,m} = \left({\beta}_m- {\tilde\beta}_{n-1}\right) \lambda^{{\tilde P} P}_{n-1,m}-
{\tilde\gamma}_{n-1}\lambda^{{\tilde P} P}_{n-2,m} +{\gamma}_{m+1}\lambda^{{\tilde P} P}_{n-1,m+1}+\lambda^{{\tilde P} P}_{n-1,m-1}\ ,\label{GenRecRel}\\
&&\hspace{1,5cm}\ 0\leq m \leq n-1\ ,\ n \geq 2\ . \notag
\end{eqnarray}

If $\{ P_n\} _{n\geq 0}$ and $\{ \tilde P_n\} _{n\geq 0}$ are symmetric, then from (\ref{linkcoeffnm}) and the symmetry of $u$, we have
$$
\lambda_{2n-1,2m}^{{\tilde P} P} =0\ ,\ \lambda_{2n,2m+1}^{{\tilde P} P}=0 \ , \ 0\leq m \leq n-1\ ,\ n\geq 1\ .
$$
Moreover, as $\beta_n=0=\tilde \beta_n$, $n\geq0$, the relation (\ref{GenRecRel}) is equivalent to
\begin{eqnarray}
\lambda_{2n,2m}^{{\tilde P} P} & = &
{\gamma}_{2m+1}\lambda_{2n-1,2m+1}^{{\tilde P} P}+\lambda_{2n-1,2m-1}^{{\tilde P} P}-{\tilde\gamma}_{2n-1}\lambda_{2n-2,2m}^{{\tilde P} P}\ , \notag\\
\lambda_{2n+1,2m+1}^{{\tilde P} P}  & = &
{\gamma}_{2m+2}\lambda_{2n,2m+2}^{{\tilde P} P}+\lambda_{2n,2m}^{{\tilde P} P}-{\tilde\gamma}_{2n}\lambda_{2n-1,2m+1}^{{\tilde P} P}\ ,
 \notag
\end{eqnarray}
for $0\leq m \leq n-1\ ,\ n\geq 1$.
Furthermore, in the case $P_n\equiv X_n$, due to (\ref{canonicalrr}), these relations become
\begin{eqnarray}
\lambda^{{\tilde P}X}_{2n,2m} & = &
\lambda^{{\tilde P}X}_{2n-1,2m-1}-{\tilde\gamma}_{2n-1}\lambda^{{\tilde P}X}_{2n-2,2m}\ , \notag\\
\lambda^{{\tilde P}X}_{2n+1,2m+1}  & = &
\lambda^{{\tilde P}X}_{2n,2m}-{\tilde\gamma}_{2n}\lambda^{{\tilde P}X}_{2n-1,2m+1}\ .
 \notag
\end{eqnarray}


\subsection{The four families of Chebyshev polynomials}\label{4ChePol}

There are four sequences of Chebyshev polynomials, they are called Chebyshev polynomials of first ($T_n$), second ($P_n$), third ($V_n$) and fourth ($W_n$) kinds. W. Gaustchi \cite{Gautschi_1992} named these last two sequences in this way, before they had been designated as airfoil polynomials (see, e.g., \cite{FrommeGolberg_1981}).
Their trigonometric definitions are
\begin{eqnarray}
T_n(x)=\frac{1}{2^{n-1}}\cos nt\quad & , & \quad
P_n(x)=\frac{1}{2^{n}}\frac{\sin (n+1)t}{\sin t}\ , \notag\\ 
V_n(x)=\frac{1}{2^{n}}\frac{\cos (n+\frac{1}{2})t}{\cos\frac{1}{2}t}\quad & , &\quad
W_n(x)=\frac{1}{2^{n}}\frac{\sin (n+\frac{1}{2})t}{\sin\frac{1}{2}t}\; , \notag 
\end{eqnarray}
where $x=\cos t$, $t\in[0,\pi]$, $n\geq0$.
Notice that, as in this work we always consider monic polynomials, 
thus some normalization constants must appear in the preceding definitions.
From them, it is trivial  to obtain explicit formulas for the zeros
\cite[pp.20-21]{Mason_2003} 
\begin{eqnarray}
T_n: &  \rho^{(n)}_k=\cos\left(\frac{(k-\frac{1}{2})\pi}{n}\right), \ k=1(1)n\  ;\  
P_n: &  \xi^{(n)}_k=\cos\left(\frac{k\pi}{n+1}\right), \ k=1(1)n, \label{zerosTnPn}\\
V_n: &  \rho^{(n)}_k=\cos\left(\frac{(k-\frac{1}{2})\pi}{n+\frac{1}{2}}\right), \ k=1(1)n\  ;\    
W_n: &  \rho^{(n)}_k=\cos\left(\frac{k\pi}{n+\frac{1}{2}}\right), \ k=1(1)n. \label{zerosVnWn} 
\end{eqnarray}
Using some trigonometric identities, it can be shown that these families satisfy (\ref{icrecOrto})-(\ref{recOrto}) with the following recurrence coefficients  \cite{Mason_2003}
\begin{eqnarray}
T_n: &  \beta_{n}=0\;, \, n\geq 0\; &,\quad \gamma_{1}=\frac{1}{2}\;,\;\gamma_{n+1}={\frac{1}{4}}\; ,\; n\ge 1\;,  \notag\\
P_n: &  \beta_{n}=0\;, \, n\geq 0\; &,\quad \gamma_{n+1}={\frac{1}{4}}\; ,\; n\ge 1\;,  \label{rcoefTT2}\\
V_n: & \beta_0=\frac{1}{2}\;,\;\beta_{n}=0\;, n\geq 1\;  &,\quad  \gamma_{n+1}={\frac{1}{4}}\; ,\; n\ge 0\;,  \notag\\
W_n: &  \beta_0=-\frac{1}{2}\;,\;\beta_{n}=0\;, n\geq 1\; &,\quad \gamma_{n+1}={\frac{1}{4}}\; ,\; n\ge 0\;.  \notag 
\end{eqnarray}
Therefore, $\{T_n\}_{n\geq 0}$ and $\{P_n\}_{n\geq 0}$ are symmetric, $\{V_n\}_{n\geq 0}$ and $\{W_n\}_{n\geq 0}$ are not; they are all positive definite. As $\{P_n\}_{n\geq 0}$ has the most simple recurrence coefficients, then it is natural to consider $\{T_n\}_{n\geq 0}$, $\{V_n\}_{n\geq 0}$ and $\{W_n\}_{n\geq 0}$ as perturbed of $\{P_n\}_{n\geq 0}$ as follows \cite{ZeliadaRocha2016_1}
\begin{equation}\label{TnVnWn}
T_n(x)=P^{d}_n\left(2;1\right)(x)\  , \ 
 V_n(x)=P^{t}_n\left(\frac{1}{2};0\right)(x)\ , \ W_n(x)=P^{t}_n\left(-\frac{1}{2};0\right)(x)\ .
 \end{equation}
With respect to the association, we have that 
$T^{(1)}_n\equiv V^{(1)}_n\equiv W^{(1)}_n\equiv P_n;$ moreover $P^{(r)}_n\equiv P_n$, 
$\forall r\geq 0$, e.g., $P_n$ is a {\em self-associated} sequence \cite{MaroniTounsi2004}.

The following CR that allow to express $\{ T_n\} _{n\geq 0}$, $\{ V_n\} _{n\geq 0}$ and $\{ W_n\} _{n\geq 0}$ in terms of $\{ P_n\} _{n\geq 0}$ are well known \cite[pp.4,8]{Mason_2003}
\begin{eqnarray}
&& T_{n+2}(x)=P_{n+2}(x)-\frac{1}{4}P_n(x)\ ,\ n\geq 0\ ,\label{CR_TP}\\
&& V_{n+1}(x)=P_{n+1}(x)-\frac{1}{2}P_n(x)\ , \ n\geq 0\ ,\label{CR_VP}\\
&& W_{n+1}(x)=P_{n+1}(x)+\frac{1}{2}P_n(x)\ , \ n\geq 0\ . \label{CR_WP}
\end{eqnarray}
One goal of this article is to generalize these CR putting in the first member of them any perturbed polynomial $P^{t}_n(\mu_r;r)(x)$ and $P^{d}_n(\lambda_r;r)(x)$ of the sequence of second kind. 

The CC $C_{n,m}=\lambda_{n,m}^{{P} X}$ and the CR of $\{P_n\}_{n\geq 0}$ in the canonical basis are \cite[p.223]{Riordan_1}
\begin{eqnarray}
&& P_{2n}(x)=\sum_{\nu=0}^nC_{2n,2\nu} x^{2\nu}\ ,  \label{P2n}\\
&& C_{2n,2\nu}=\frac{(-1)^{n-\nu}}{2^{2(n-\nu)}}\binom{n+\nu}{n-\nu},\ \nu=0(1)n, \ C_{2n,2\nu+1}=0,\  \nu=0(1)n-1\ ;\label{C2n}\\
&&P_{2n+1}(x)=\sum_{\nu=0}^nC_{2n+1,2\nu+1} x^{2\nu+1}\  , \label{P2n1}\\
&& \   C_{2n+1,2\nu+1}=\frac{(-1)^{n-\nu}}{2^{2(n-\nu)}}\binom{n+\nu+1}{n-\nu},\ \nu=0(1)n ,\
C_{2n+1,2\nu}=0,\ \nu=0(1)n\ . \label{C2n1}
\end{eqnarray}
Another goal of this article is to generalize these well known CC and CR putting in the first member of them any perturbed polynomial $P^{t}_n(\mu_r;r)(x)$ and $P^{d}_n(\lambda_r;r)(x)$ of the sequence of second kind. 
From the above identities, we obtain, in particular,  that
\begin{eqnarray}
&&  C_{n,n-1}=0\ , \ C_{2n,0}=\frac{(-1)^{n}}{2^{2n}}\ , \  C_{2n+1,0}=0\ ,\ n\geq 0\ ,\label{CCTT2_0n1}
\end{eqnarray}
which implies, by (\ref{CC_Can_Zeros_Pn}), that
\begin{eqnarray}
&& \sum_{k=1}^{n} {\xi}_k^{(n)}=0\ ,\label{S_ZerosTT2}\\
&&\prod_{k=1}^{2n} {\xi}_k^{(2n)}=\frac{(-1)^{n}}{2^{2n}}\Longrightarrow P_{2n}(0)\neq0\ ,
\prod_{k=1}^{2n+1} {\xi}_k^{(2n+1)}=0\Longrightarrow P_{2n+1}(0)=0\ . \label{P_ZerosTT2}
\end{eqnarray}
We remark that 
$C_{2n,2\nu+1}=0$, $\nu=0(1)n-1$, $C_{2n+1,2\nu}=0$ and $P_{2n+1}(0)=0$, $n\geq 0$ are  assured by symmetry.

The four kinds of Chebyshev polynomials \cite{Gautschi_1992,Mason_2003}  belong to the Jacobi class of classical polynomials $\{P_{n}^{(\alpha,\beta)}(x)\}_{n\geq 0}$ \cite{MARO 93,MARO 94} orthogonal with respect to the form $\mathcal{J}(\alpha,\beta)$. A Jacobi form is regular if and only if $\alpha\neq -n$, $\beta\neq -n$, $\alpha+\beta\neq -(n+1)$, $n\geq1$; 
 it is symmetric for $\alpha=\beta$; it is positive definite if and only if $\alpha+1>0$ and $\beta+1>0$ and has the following integral representation for $Re(1+\alpha)>0$ and $Re(1+\beta)>0$ \cite{Chihara,MARO 93, MARO 94}
$$
<\mathcal{J}(\alpha,\beta),f>=\frac{1}{2^{\alpha+\beta+1}}\frac{\Gamma(\alpha+\beta+2)}{\Gamma(\alpha+1)\Gamma(\beta+1)}\int_{-1}^{1}(1+x)^\alpha(1-x)^\beta f(x)dx\;.
$$
Let us note the Chebyshev normalized forms by ${\cal T}_1$ $(\alpha=\beta=-\frac{1}{2})$, ${\cal U}$ $(\alpha=\beta=\frac{1}{2})$, ${\cal T}_3$ $(\alpha=\frac{1}{2},\ \beta=-\frac{1}{2})$ and ${\cal T}_4$ $(\alpha=-\frac{1}{2},\ \beta=\frac{1}{2})$ ; 
their integral representations are given by \cite{Mason_2003}
\begin{eqnarray}
<{\cal T}_1,f>=\frac{1}{\pi}\int_{-1}^{+1}\frac{f(x)}{\sqrt{1-x^2}}dx 
&, &  <{\cal U},f>=\frac{2}{\pi}\int_{-1}^{1}f(x)\sqrt{1-x^2}dx\ , \label{intrepTT12}\\
 <{\cal T}_3,f>=\frac{1}{\pi}\int_{-1}^{1}f(x)\sqrt{\frac{1+x}{1-x}}dx 
& , & <{\cal T}_4,f>=\frac{1}{\pi}\int_{-1}^{1}f(x)\sqrt{\frac{1-x}{1+x}}dx\ . \label{intrepTT34}
\end{eqnarray}

Chebyshev forms are of second degree \cite{DBegMARO 97}. Perturbed Chebyshev polynomials are also of second degree and consequently they are semi-classical \cite{MARO 95,MARO 95_1}; the {\it PSDF - Perturbed Second Degree Forms\/} symbolic algorithm \cite{ZeliadaRocha2015_1} allowed to explicit some of their main properties, namely the second order linear differential equation \cite{ZeliadaRocha2016_1,ZeliadaRocha2015_1}. 
Integral representations of perturbed Chebyshev forms are given in the co-recursive case in \cite{Chihara_0} and in the co-dilated situation of order 1 in \cite{MAROMEJRI}. 
To the best of my knowledge, these integral representations remain an open problem for orders of perturbation greater than or equal to two.

From now on, $\{P_n\}_{n\geq 0}$ will note the sequence of monic Chebyshev polynomials of second kind.

\section{Connection coefficients and connection relations in terms of the Chebyshev polynomials of second kind}\label{CCChe}

In this section, we shall give the CC and the CR that allow to express $\{P^{t}_n(\mu_r;r)(x)\}_{n\geq 0}$ and $\{P^{d}_n(\mu_r;r)(x)\}_{n\geq 0}$ in terms of  $\{P_n\}_{n\geq 0}$, for any order $r$ of perturbation. That is, we shall deal with 
$$\lambda^{t}_{n,m}:=\lambda^{P^tP}_{n,m}=\lambda_{n,m}(P^{t}(\mu_r;r)\leftarrow P)\  {\rm and} \ \lambda^{d}_{n,m}:=\lambda^{P^dP}_{n,m}=\lambda_{n,m}(P^{d}(\lambda_r;r)\leftarrow P)\ .$$
We begin by two lemmas, where we rewrite the general recurrence relation (\ref{linkcoeffnms})-(\ref{GenRecRel}) in the two particular cases considered, replacing $\{\beta_n\}_{n\geq 0}$, $\{\gamma_{n+1}\}_{n\geq 0}$ and $\{{\tilde\beta}_n\}_{n\geq 0}$, $\{\tilde\gamma_{n+1}\}_{n\geq 0}$  by their particular values obtained from (\ref{rcoefftra}), (\ref{rcoeffdil}) and (\ref{rcoefTT2}) and given by 
\begin{eqnarray}
\beta_r^{t}=\mu_r\ ,\ \beta_n^{t}=0,\ n\neq r\ ;\ \gamma_{n+1}^t=\frac{1}{4}\ ,\ n\geq 0\ ,\ r\geq 0\ ,\ {\rm or}\label{rrc_tra}\\
\beta_n^{d}=0\ ;\ \gamma_{r}^d=\lambda_r\frac{1}{4}\ ,\ \gamma_{n+1}^d=\frac{1}{4},\ n\neq r-1\ ,\ n\geq 0\label{rrc_dil}\ ,\ r>0\ .
\end{eqnarray}
\begin{lemma}\label{lemma1}
For the {\it co-recursive} case $(r=0)$, 
\begin{eqnarray}
&&\lambda^{t}_{n,n}=1,\ n\geq 0\ ;\ \lambda^{t}_{1,0}=-\mu_0\ ;\label{GRRCoRecCondi}\\
&&\lambda^{t}_{n,m}=\frac{1}{4}\left(-\lambda^{t}_{n-2,m}+\lambda^{t}_{n-1,m+1}\right)+\lambda^{t}_{n-1,m-1}\ ,\ 0\leq m<n\ ,\ n\geq2\ . \label{GRRCoRec}
\end{eqnarray}
For the {\em $r$th-perturbed  by translation} case $(r\geq 1)$, 
\begin{eqnarray}
&&\lambda^{t}_{n,n}=1,\ n\geq 0\ ;\ \lambda^{t}_{1,0}=0\ ;\label{GRRGenCoRecCondi}\\
&&\lambda^{t}_{r+1,m}=-\mu_r\lambda^{t}_{r,m}+\frac{1}{4}\left(-\lambda^{t}_{r-1,m}+\lambda^{t}_{r,m+1}\right)+\lambda^{t}_{r,m-1}\ ,\ 0\leq m<r+1\ ;\label{GRRGenCoRecR}\\
&&\lambda^{t}_{n,m}=\frac{1}{4}\left(-\lambda^{t}_{n-2,m}+\lambda^{t}_{n-1,m+1}\right)+\lambda^{t}_{n-1,m-1},\ 0\leq m<n, n\neq r+1, n\geq2. \label{GRRGenCoRecN}
\end{eqnarray}
\end{lemma}
\begin{figure}[htbp]
\caption{Representation of recurrence relations (\ref{GRRGenCoRecR}) and  (\ref{GRRGenCoRecN}) for the {\em $r$th-perturbed by translation case for $r\geq1$.}}
\label{GRR_Tra_r_n}
\begin{center}
\begin{tabular}{|c|c|}\hline
\includegraphics[width=6.5cm]{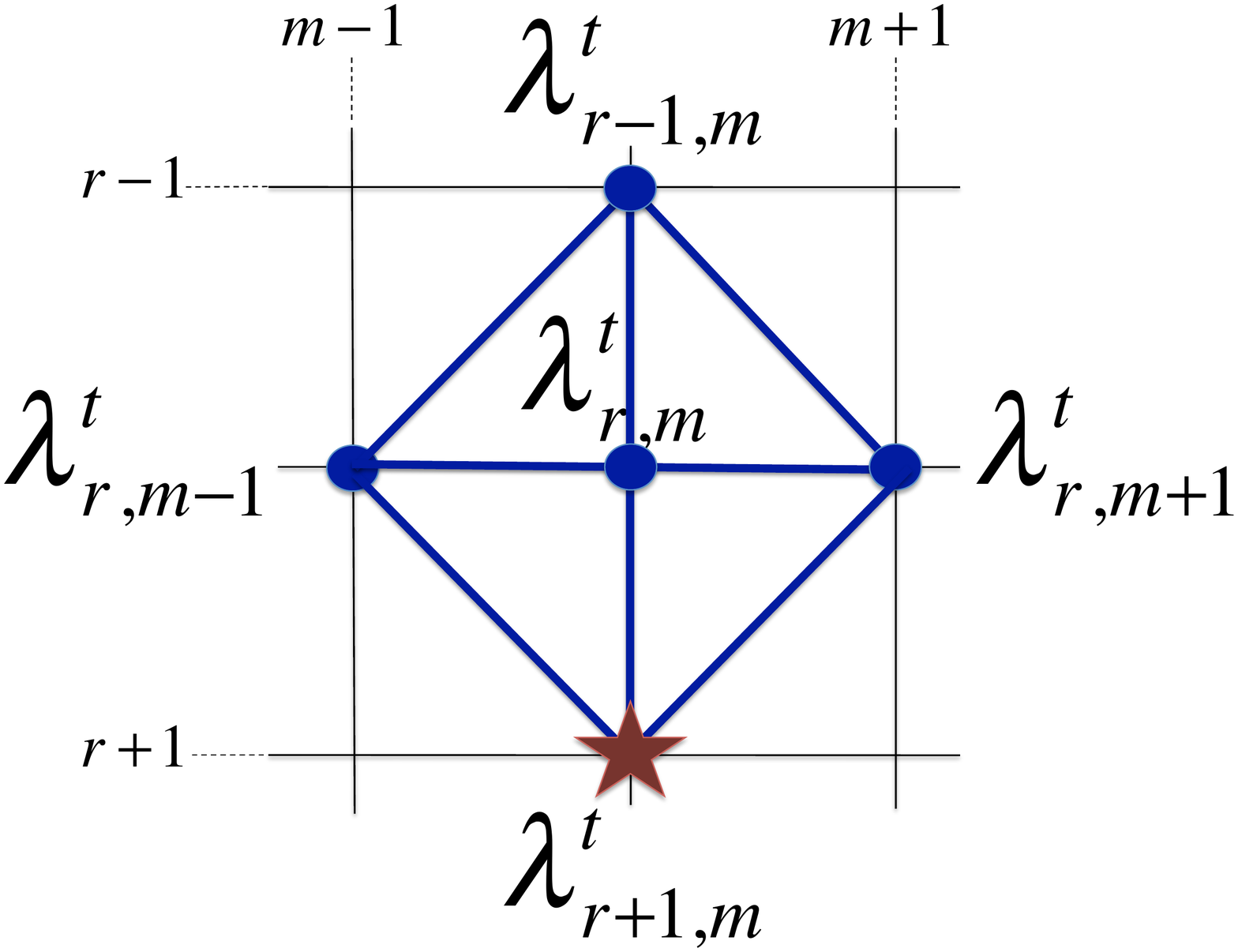} & 
\includegraphics[width=6.5cm]{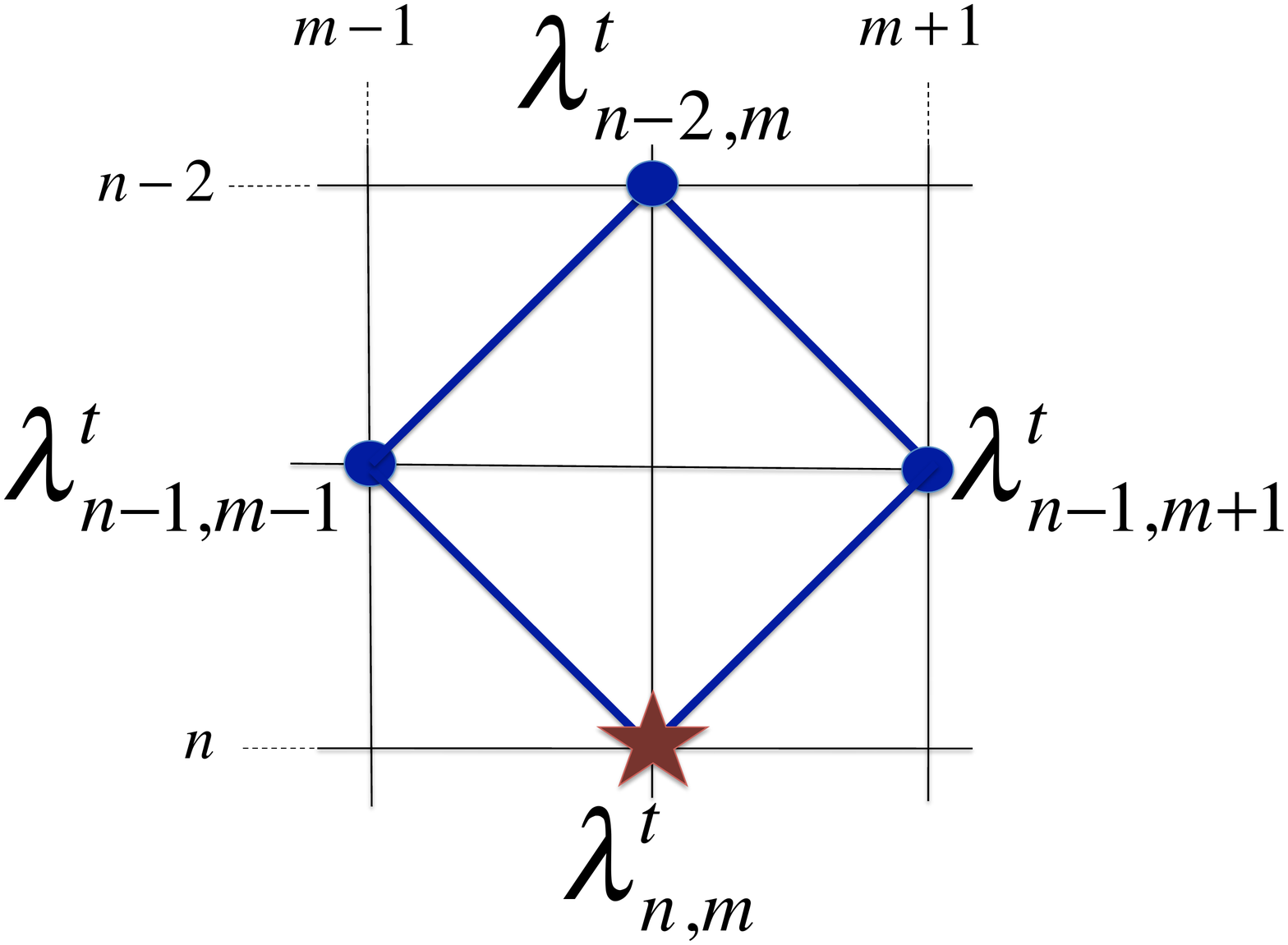} \\ \hline 
\end{tabular}
\end{center}
\end{figure}
\begin{lemma}\label{lemma2}
For the {\em $r$th-perturbed by dilatation} case $(r\geq 1)$, 
\begin{eqnarray}
&&\lambda^{d}_{n,n}=1,\ n\geq 0\ ;\ \lambda^{d}_{1,0}=0\ ;\label{GRRGenCoDilCondi}\\
&&\lambda^{d}_{r+1,m}=\frac{1}{4}\left(-\lambda_r\lambda^{d}_{r-1,m}+\lambda^{d}_{r,m+1}\right)+\lambda^{d}_{r,m-1}\ ,\ 0\leq m<r+1\ ;\label{GRRGenCoDilR}\\
&&\lambda^{d}_{n,m}=\frac{1}{4}\left(-\lambda^{d}_{n-2,m}+\lambda^{d}_{n-1,m+1}\right)+\lambda^{d}_{n-1,m-1},\ 0\leq m<n, n\neq r+1,n\geq2. \label{GRRGenCoDilN}
\end{eqnarray}
\end{lemma}
\begin{figure}[htbp]
\caption{Representation of recurrence relations (\ref{GRRGenCoDilR}) and (\ref{GRRGenCoDilN}) for the {\it $r$th-perturbed by dilatation case for $r\geq1$.}}
\label{GRR_Dil_r_n}
\begin{center}
\begin{tabular}{|c|c|}\hline
\includegraphics[width=6.5cm]{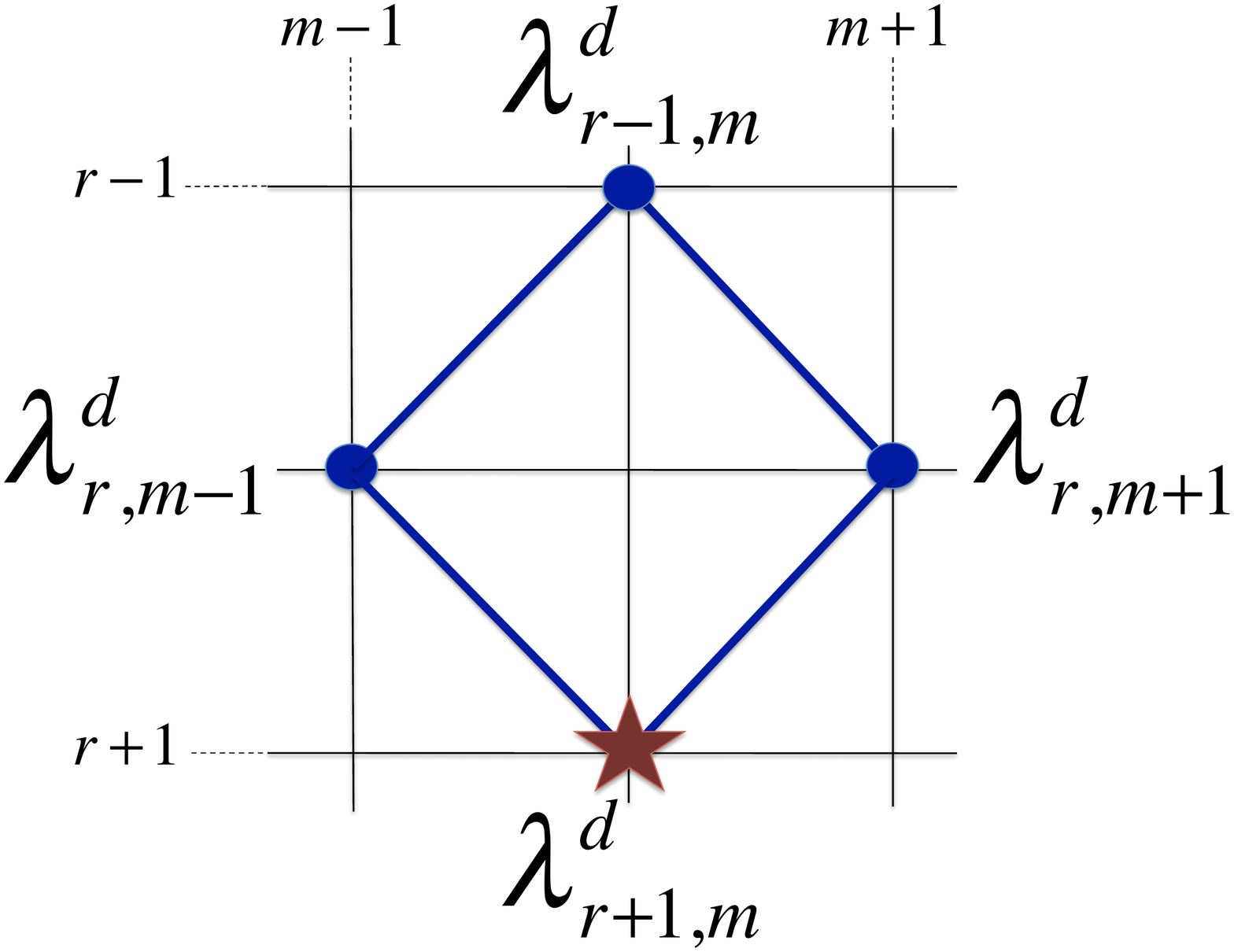} & 
\includegraphics[width=6.5cm]{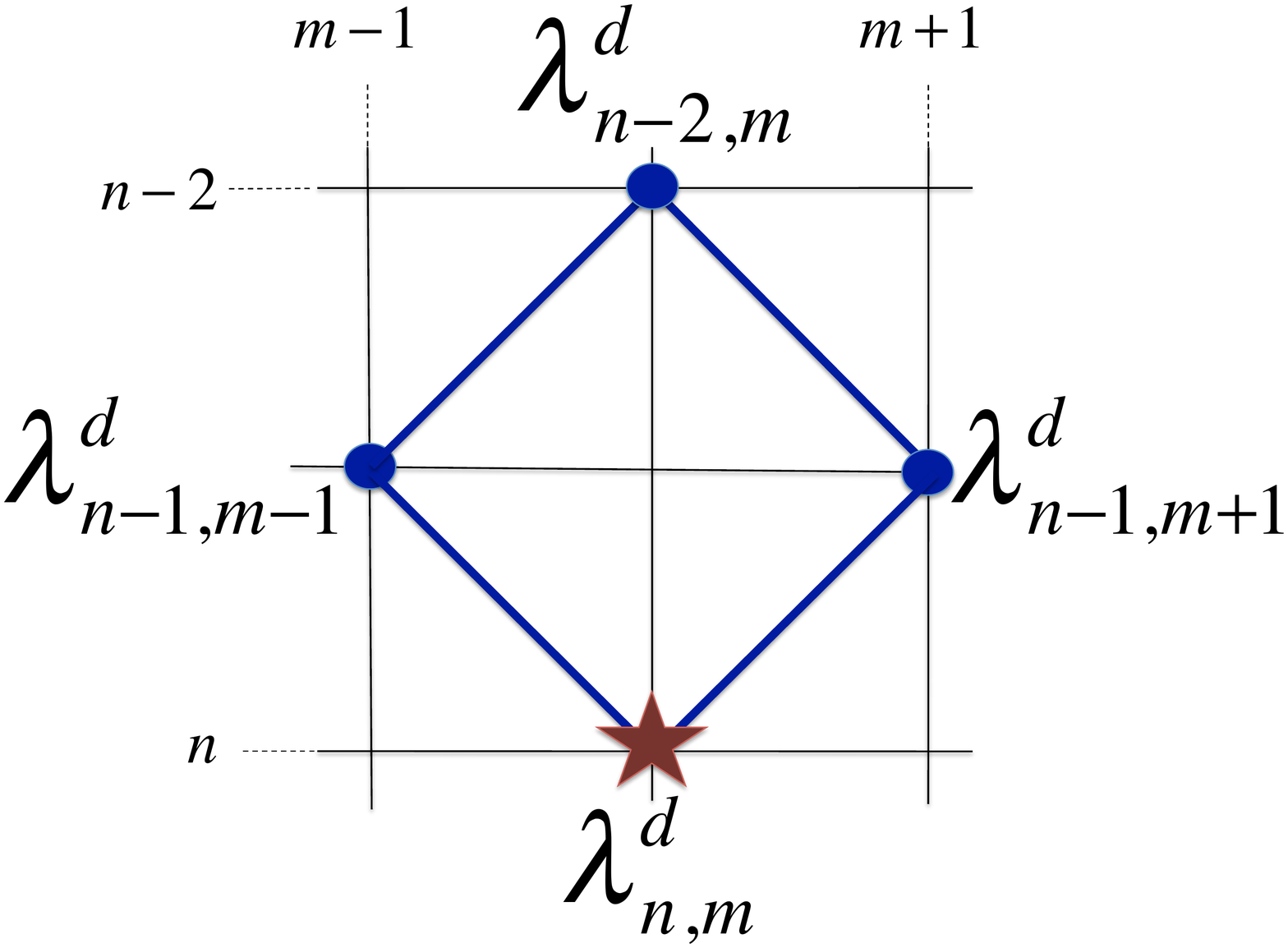} \\ \hline 
\end{tabular}
\end{center}
\end{figure}
\newpage
The software {\it CCOP - Connection Coefficients for Orthogonal Polynomials\/} \cite{ZeliaPascal2013_1,ZeliaPascal2013_2} (see, also, \cite{ZeliaPascal3}) written in the {\it Mathematica$^{\circledR}$} language, includes an implementation of the recurrence relations (\ref{linkcoeffnms})-(\ref{GenRecRel}) that allows the symbolic recursive computation of the first CC from the recurrence coefficients of the two polynomial sequences involved. In the cases under study, {\it CCOP\/} produced the results given in Tables \ref{TabTra3} and \ref{TabDil3} for the  perturbation of order $r=3$. From these results and others, we discover that CC are constant by downward diagonal with very simple expressions and it was easy to infer the closed formulas corresponding to an arbitrary order $r$ of perturbation and any nonnegative integers $n$ and $m$. It was by this procedure that we formulate Theorems \ref{CCTraCheDig} and \ref{CCDilCheDig}, and Propositions \ref{CCTraChe} and \ref{CCDilChe} presented next. Tables \ref{TabTra_odd} and \ref{TabTra_even} generalize Table \ref{TabTra3} for any odd or even orders ($r\geq 0$) in the translation case; Tables \ref{TabDil_odd} and \ref{TabDil_even} generalize Table \ref{TabDil3} for any odd or even orders ($r\geq 1$) in the dilatation case. In fact, there are slight differences depending on the parity of $r$.
\begin{table}[htbp] 
\caption{CC for the 3rd-perturbed {\it by translation} case with parameter $\mu_3$.}
\label{TabTra3}
\begin{center}
$\begin{array}{| l | l l l l | l ll  l |lllll}\hline
m  & 0& 1& 2& 3& 4& 5& 6&7 & 8&9 &10 & 11 & \hdots \\ 
n  & & & & & & & & & & & & &   \\ \hline
0 & 1 & \text{} & \text{} & \text{} & \text{} & \text{} & \text{} & \text{} & \text{} & \text{} & \text{} & \text{} & \\
1 & 0 & 1 & \text{} & \text{} & \text{} & \text{} & \text{} & \text{} & \text{} & \text{} & \text{} & \text{} &  \\
2 & 0 & 0 & 1 & \text{} & \text{} & \text{} & \text{} & \text{} & \text{} & \text{} & \text{} & \text{} &  \\
3 & 0 & 0 & 0 & 1 & \text{} & \text{} & \text{} & \text{} & \text{} & \text{} & \text{} & \text{} &  
\\ \hline
4 & 0 & 0 & 0 & -\mu _3 & 1 & \text{} & \text{} & \text{} & \text{} & \text{} & \text{} & \text{} &  \\
5 & 0 & 0 & -\frac{\mu _3}{4} & 0 & -\mu _3 & 1 & \text{} & \text{} & \text{} & \text{} & \text{} & \text{} &  \\
6 & 0 & -\frac{\mu _3}{16} & 0 & -\frac{\mu _3}{4} & 0 & -\mu _3 & 1 & \text{} & \text{} & \text{} & \text{} & \text{} &  \\
7 &  -\frac{\mu _3}{64} & 0 & -\frac{\mu _3}{16} & 0 & -\frac{\mu _3}{4} & 0 & -\mu _3 & 1 & \text{} & \text{} & \text{} & \text{} &  \\ \hline
8 &  0 & -\frac{\mu _3}{64} & 0 & -\frac{\mu _3}{16} & 0 & -\frac{\mu _3}{4} & 0 & -\mu _3 & 1 & \text{} & \text{} & \text{} &  \\
 9 & 0 & 0 & -\frac{\mu _3}{64} & 0 & -\frac{\mu _3}{16} & 0 & -\frac{\mu _3}{4} & 0 & -\mu _3 & 1 & \text{} & \text{} &  \\
10 &  0 & 0 & 0 & -\frac{\mu _3}{64} & 0 & -\frac{\mu _3}{16} & 0 & -\frac{\mu _3}{4} & 0 & -\mu _3 & 1 & \text{} &  \\
11 &  0 & 0 & 0 & 0 & -\frac{\mu _3}{64} & 0 & -\frac{\mu _3}{16} & 0 & -\frac{\mu _3}{4} & 0 & -\mu _3 & 1 &  \\ \hline
\vdots &  \vdots & \vdots & \vdots & \vdots & \vdots & \vdots & \vdots & \vdots & \vdots & 
\vdots & \vdots & \vdots &  \ddots\\ 
\end{array}$
\end{center}
\label{default}
\end{table}%
\begin{table}[htbp]
\caption{CC for the 3rd-perturbed {\it by dilatation} case with parameter $\lambda_3$.}
\label{TabDil3}
\begin{center}
$\begin{array}{|l| l l l l | l l l | llll | l  l}\hline
m  & 0 & 1& 2& 3& 4& 5& 6&7 & 8 &9 &10 &  \hdots \\ 
n  & & & & & & & & & & &  & \\ \hline
0 & 1 & \text{} & \text{} & \text{} & \text{} & \text{} & \text{} & \text{} & \text{} & \text{} & \text{} & \text{} \\
1 & 0 & 1 & \text{} & \text{} & \text{} & \text{} & \text{} & \text{} & \text{} & \text{} & \text{} & \text{} \\
2 &  0 & 0 & 1 & \text{} & \text{} & \text{} & \text{} & \text{} & \text{} & \text{} & \text{} & \text{} \\
3 &  0 & 0 & 0 & 1 & \text{} & \text{} & \text{} & \text{} & \text{} & \text{} & \text{} & \text{} \\ \hline
4 & 0 & 0 & \frac{\left(1-\lambda _3\right)}{4} & 0 & 1 & \text{} & \text{} & \text{} & \text{} & \text{} & \text{} & \text{} \\
5 &  0 & \frac{\left(1-\lambda _3\right)}{16} & 0 & \frac{\left(1-\lambda _3\right)}{4} & 0 & 1 & \text{} & \text{} & \text{} & \text{} & \text{} & \text{} \\
6 &  \frac{\left(1-\lambda _3\right)}{64}  & 0 & \frac{\left(1-\lambda _3\right)}{16} & 0 & \frac{\left(1-\lambda _3\right)}{4} & 0 & 1 & \text{} & \text{} & \text{} & \text{} & \text{} \\ \hline
7 &  0 & \frac{\left(1-\lambda _3\right)}{64}& 0 & \frac{\left(1-\lambda _3\right)}{16}  & 0 & \frac{\left(1-\lambda _3\right)}{4}  & 0 & 1 & \text{} & \text{} & \text{} & \text{} \\
8 &  0 & 0 & \frac{\left(1-\lambda _3\right)}{64} & 0 & \frac{\left(1-\lambda _3\right)}{16}  & 0 & \frac{\left(1-\lambda _3\right)}{4}  & 0 & 1 & \text{} & \text{} & \text{} \\
9 &  0 & 0 & 0 & \frac{\left(1-\lambda _3\right)}{64}  & 0 & \frac{\left(1-\lambda _3\right)}{16}  & 0 & \frac{\left(1-\lambda _3\right)}{4}  & 0 & 1 & \text{} & \text{} \\
10 &  0 & 0 & 0 & 0 & \frac{\left(1-\lambda _3\right)}{64}  & 0 & \frac{\left(1-\lambda _3\right)}{16} &  0 & \frac{\left(1-\lambda _3\right)}{4}  & 0 & 1 & \text{} \\ \hline
\vdots &  \vdots & \vdots & \vdots & \vdots & \vdots & \vdots & \vdots & \vdots & \vdots & 
\vdots & \vdots & \ddots\\ 
\end{array}$
\end{center}
\label{default}
\end{table}%

\begin{landscape}
\begin{table}[htbp] 
\caption{CC for the $r$th-perturbed {\it by translation} case with parameter $\mu_r$, for $r=2r'+1$, for $r'\geq0$.}
\label{TabTra_odd}
\begin{center}
$\begin{array}{|l|lllll l | l l   ll   l | l l l l | l}\hline
m  & 0& 1& 2& \ldots & r-1 & r & r+1 & r+2 & \ldots & 2r & 2r+1 & 2r+2 & 2r+3 &  \ldots  & 3r+2 & \ldots  \\ 
n  & & & & & & & & & & & & &  & & &  \\ \hline
0 & 1 & \text{} & \text{} & \text{} & \text{} & \text{} & \text{} & \text{} & \text{} & \text{} & \text{} & \text{} & & & &\\
1 & 0 & 1 & \text{} & \text{} & \text{} & \text{} & \text{} & \text{} & \text{} & \text{} & \text{} & \text{} &  & & &\\
2 & 0 & 0 & 1 & \text{} & \text{} & \text{} & \text{} & \text{} & \text{} & \text{} & \text{} & \text{} &  & & &\\
\vdots & \vdots & \vdots & \vdots & \ddots & \text{} & \text{} & \text{} & \text{} & \text{} & \text{} & \text{} & \text{} &  & & &\\
r-1 & 0 & 0 & 0 & \ldots & 1 & \text{} & \text{} & \text{} & \text{} & \text{} & \text{} & \text{} & & & &\\ 
r & 0 & 0 & 0 & \ldots & 0 &1 &  \text{} & \text{} & \text{} & \text{} & \text{} & \text{} &  & & & \\ \hline
r+1 & 0 & 0 & 0 & \ldots & 0 & -\mu _r & 1 & \text{}  & \text{} & \text{} & \text{} & \text{} &  & & & \\ 
r+2 & 0 & 0 & 0 & \ldots & -\frac{\mu _r}{4} & 0  & -\mu _r & 1 &  \text{} & \text{} & \text{} & & & & &\\ 
\vdots & \vdots & \vdots & \vdots & \vdots & \vdots & \ddots & \ddots & \ddots  &  \text{}  & \text{} &  \text{} & \text{} &  & & &\\
2r & 0 & -\frac{\mu _r}{4^{r-1}} & 0 & \ldots &0 & -\frac{\mu _r}{4^{r'} } &0 & -\frac{\mu _r}{4^{r'-1}}  & \ddots & 1 & \text{} & \text{} & & & \\
2r+1 &  -\frac{\mu _r}{4^r} & 0 & -\frac{\mu _r}{4^{r-1}} & \ldots & -\frac{\mu _r}{4^{r'+1}} & 0& -\frac{\mu _r}{4^{r'}} &0 & \ldots & -\mu _r & 1 & \text{} &  &  & & \\ \hline
2r+2  &  0 & -\frac{\mu _r}{4^r} & 0 & \ldots &0 &  -\frac{\mu _r}{4^{r'+1}} & 0 & -\frac{\mu _r}{4^{r'} }  & \ldots &  0 & -\mu _r & 1 & \text{} &  & &  \\
2r+3 & 0 &0 & -\frac{\mu _r}{4^r} & \ldots & -\frac{\mu _r}{4^{r'+2} } & 0 & -\frac{\mu _r}{4^{r'+1}} & 0 &  \ldots & \ldots & 0 & -\mu _r & 1 & & \text{}  &   \\
\vdots &  \vdots & \vdots & \vdots &\vdots & \vdots & \vdots & \vdots & \vdots  & \vdots & \vdots & \vdots & \vdots & \vdots & \ddots & &  \\
3r+2 &  0 & 0 & 0 &0 & 0 & 0 & -\frac{\mu _r}{4^r} & 0 & \ldots & -\frac{\mu _r}{4^{r'+1} } & 0& -\frac{\mu _r}{4^{r'} }  & 0 & \ldots  &  1 & \ldots  \\ \hline
\vdots &  \vdots & \vdots & \vdots & \vdots & \vdots & \vdots & \vdots & \vdots & \vdots & 
\vdots & \vdots & \vdots &  \vdots & \vdots & \vdots & \vdots \\ 
\end{array}$
\end{center}
\label{default}
\end{table}%
\end{landscape}
\begin{landscape}
\begin{table}[htbp] 
\caption{CC for the $r$th-perturbed {\it by translation} case with parameter $\mu_r$, for $r=2r'$, for $r'\geq1$.}
\label{TabTra_even}
\begin{center}
$\begin{array}{|l|lllll l | l l   ll   l | l l l l | l}\hline
m  & 0& 1& 2& \ldots & r-1 & r & r+1 & r+2 & \ldots & 2r & 2r+1 & 2r+2 & 2r+3 &  \ldots  & 3r+2 & \ldots  \\ 
n  & & & & & & & & & & & & &  & & &  \\ \hline
0 & 1 & \text{} & \text{} & \text{} & \text{} & \text{} & \text{} & \text{} & \text{} & \text{} & \text{} & \text{} & & & &\\
1 & 0 & 1 & \text{} & \text{} & \text{} & \text{} & \text{} & \text{} & \text{} & \text{} & \text{} & \text{} &  & & &\\
2 & 0 & 0 & 1 & \text{} & \text{} & \text{} & \text{} & \text{} & \text{} & \text{} & \text{} & \text{} &  & & &\\
\vdots & \vdots & \vdots & \vdots & \ddots & \text{} & \text{} & \text{} & \text{} & \text{} & \text{} & \text{} & \text{} &  & & &\\
r-1 & 0 & 0 & 0 & \ldots & 1 & \text{} & \text{} & \text{} & \text{} & \text{} & \text{} & \text{} & & & &\\ 
r & 0 & 0 & 0 & \ldots & 0 &1 &  \text{} & \text{} & \text{} & \text{} & \text{} & \text{} &  & & & \\ \hline
r+1 & 0 & 0 & 0 & \ldots & 0 & -\mu _r & 1 & \text{}  & \text{} & \text{} & \text{} & \text{} &  & & & \\ 
r+2 & 0 & 0 & 0 & \ldots & -\frac{\mu _r}{4} & 0  & -\mu _r & 1 &  \text{} & \text{} & \text{} & & & & &\\ 
\vdots & \vdots & \vdots & \vdots & \vdots & \vdots & \ddots & \ddots & \ddots  &  \text{}  & \text{} &  \text{} & \text{} &  & & &\\
2r & 0 & -\frac{\mu _r}{4^{r-1}} & 0 & \ldots & -\frac{\mu _r}{4^{r'} } &0 & -\frac{\mu _r}{4^{r'-1}} & 0 & \ddots & 1 & \text{} & \text{} & & & \\
2r+1 &  -\frac{\mu _r}{4^r} & 0 & -\frac{\mu _r}{4^{r-1}} & \ldots &0& -\frac{\mu _r}{4^{r'}} & 0   & -\frac{\mu _r}{4^{r'-1}}  & \ldots & -\mu _r & 1 & \text{} &  &  & & \\ \hline
2r+2  &  0 & -\frac{\mu _r}{4^r} & 0 & \ldots &  -\frac{\mu _r}{4^{r'+1}} & 0 &  -\frac{\mu _r}{4^{r'}} & 0 & \ldots &  0 & -\mu _r & 1 & \text{} &  & &  \\
2r+3 & 0 &0 & -\frac{\mu _r}{4^r} & \ldots & 0  & -\frac{\mu _r}{4^{r'+1}} & 0 &  -\frac{\mu _r}{4^{r'}} &  \ldots & -\frac{\mu _r}{4}   & 0 & -\mu _r & 1 & & \text{}  &   \\
\vdots &  \vdots & \vdots & \vdots &\vdots & \vdots & \vdots & \vdots & \vdots  & \vdots & \vdots & \vdots & \vdots & \vdots & \ddots & &  \\
3r+2 &  0 & 0 & 0 &0 & 0 & 0 & -\frac{\mu _r}{4^r} & 0 & \ldots &  0& -\frac{\mu _r}{4^{r'} }  &  0 &  -\frac{\mu _r}{4^{r'-1} }  & \ldots  &  1 & \ldots  \\ \hline
\vdots &  \vdots & \vdots & \vdots & \vdots & \vdots & \vdots & \vdots & \vdots & \vdots & 
\vdots & \vdots & \vdots &  \vdots & \vdots & \vdots & \vdots \\ 
\end{array}$
\end{center}
\label{default}
\end{table}%
\end{landscape}
\begin{landscape}
\begin{table}[htbp] 
\caption{CC for the $r$th-perturbed {\it by dilatation} case with parameter $\lambda_r$, for $r=2r'+1$, for $r'\geq0$.}
\label{TabDil_odd}
\begin{center}
$\begin{array}{|l|lllll l | l l   ll |  l  l l l l | l}\hline
m  & 0& 1& 2& \ldots & r-1 & r & r+1 & r+2 & \ldots & 2r & 2r+1 & 2r+2 & 2r+3 &  \ldots  & 3r+1 & \ldots  \\ 
n  & & & & & & & & & & & & &  & & &  \\ \hline
0 & 1 & \text{} & \text{} & \text{} & \text{} & \text{} & \text{} & \text{} & \text{} & \text{} & \text{} & \text{} & & & &\\
1 & 0 & 1 & \text{} & \text{} & \text{} & \text{} & \text{} & \text{} & \text{} & \text{} & \text{} & \text{} &  & & &\\
2 & 0 & 0 & 1 & \text{} & \text{} & \text{} & \text{} & \text{} & \text{} & \text{} & \text{} & \text{} &  & & &\\
\vdots & \vdots & \vdots & \vdots & \ddots & \text{} & \text{} & \text{} & \text{} & \text{} & \text{} & \text{} & \text{} &  & & &\\
r-1 & 0 & 0 & 0 & \ldots & 1 & \text{} & \text{} & \text{} & \text{} & \text{} & \text{} & \text{} & & & &\\ 
r & 0 & 0 & 0 & \ldots & 0 &1 &  \text{} & \text{} & \text{} & \text{} & \text{} & \text{} &  & & & \\ \hline
r+1 & 0 & 0& 0 & \ldots & \frac{(1-\lambda_r)}{4} & 0 & 1 & \text{}  & \text{} & \text{} & \text{} & \text{} &  & & & \\ 
r+2 & 0 & 0 & 0 & \ldots & 0 &  \frac{(1-\lambda_r)}{4} & 0 & 1 &  \text{} & \text{} & \text{} & & & & &\\ 
\vdots & \vdots & \vdots & \vdots & \vdots & \vdots & \ddots & \ddots & \ddots  &  \text{}  & \text{} &  \text{} & \text{} &  & & &\\
2r & \frac{(1-\lambda_r)}{4^r} & 0 & \frac{(1-\lambda_r)}{4^{r-1}} & \ldots &\frac{(1-\lambda_r)}{4^{r'+1}} &0 &\frac{(1-\lambda_r)}{4^{r'}} & 0 & \ddots & 1 & \text{} & \text{} & & & \\ \hline
2r+1 &  0 & \frac{(1-\lambda_r)}{4^r}  & 0 & \ldots &0 &  \frac{(1-\lambda _r)}{4^{r'+1}}& 0 &\frac{(1-\lambda _r)}{4^{r'}} & \ldots & 0 & 1 & \text{} &  &  & & \\ 
2r+2  &  0 & 0 & \frac{(1-\lambda_r)}{4^r} & \ldots &\frac{(1-\lambda _r)}{4^{r'+2}}  &  0& \frac{(1-\lambda _r)}{4^{r'+1} } & 0  & \ldots &   \frac{(1-\lambda _r)}{4} & 0 & 1 & \text{} &  & &  \\
\vdots &  \vdots & \vdots & \vdots &\vdots & \vdots & \vdots & \vdots & \vdots  & \vdots & \vdots & \vdots & \vdots & \vdots & \ddots & &  \\
3r+1 &  0 & 0 & 0 &0 & 0 & 0 & \frac{(1-\lambda_r)}{4^r} & 0 & \ldots &\frac{(1-\lambda _r)}{4^{r'+1} } &   0& \frac{(1-\lambda _r)}{4^{r'} }  & 0 & \ldots  &  1 & \ldots  \\ \hline
\vdots &  \vdots & \vdots & \vdots & \vdots & \vdots & \vdots & \vdots & \vdots & \vdots & 
\vdots & \vdots & \vdots &  \vdots & \vdots & \vdots & \vdots \\ 
\end{array}$
\end{center}
\label{default}
\end{table}%
\end{landscape}
\begin{landscape}
\begin{table}[htbp] 
\caption{CC for the $r$th-perturbed {\it by dilatation} case with parameter $\lambda_r$, for $r=2r'$, for $r'\geq1$.}
\label{TabDil_even}
\begin{center}
$\begin{array}{|l|lllll l | l l   ll |  l  l l l l | l}\hline
m  & 0& 1& 2& \ldots & r-1 & r & r+1 & r+2 & \ldots & 2r & 2r+1 & 2r+2 & 2r+3 &  \ldots  & 3r+1 & \ldots  \\ 
n  & & & & & & & & & & & & &  & & &  \\ \hline
0 & 1 & \text{} & \text{} & \text{} & \text{} & \text{} & \text{} & \text{} & \text{} & \text{} & \text{} & \text{} & & & &\\
1 & 0 & 1 & \text{} & \text{} & \text{} & \text{} & \text{} & \text{} & \text{} & \text{} & \text{} & \text{} &  & & &\\
2 & 0 & 0 & 1 & \text{} & \text{} & \text{} & \text{} & \text{} & \text{} & \text{} & \text{} & \text{} &  & & &\\
\vdots & \vdots & \vdots & \vdots & \ddots & \text{} & \text{} & \text{} & \text{} & \text{} & \text{} & \text{} & \text{} &  & & &\\
r-1 & 0 & 0 & 0 & \ldots & 1 & \text{} & \text{} & \text{} & \text{} & \text{} & \text{} & \text{} & & & &\\ 
r & 0 & 0 & 0 & \ldots & 0 &1 &  \text{} & \text{} & \text{} & \text{} & \text{} & \text{} &  & & & \\ \hline
r+1 & 0 & 0& 0 & \ldots & \frac{(1-\lambda_r)}{4} & 0 & 1 & \text{}  & \text{} & \text{} & \text{} & \text{} &  & & & \\ 
r+2 & 0 & 0 & 0 & \ldots & 0 &  \frac{(1-\lambda_r)}{4} & 0 & 1 &  \text{} & \text{} & \text{} & & & & &\\ 
\vdots & \vdots & \vdots & \vdots & \vdots & \vdots & \ddots & \ddots & \ddots  &  \text{}  & \text{} &  \text{} & \text{} &  & & &\\
2r & \frac{(1-\lambda_r)}{4^r} & 0 & \frac{(1-\lambda_r)}{4^{r-1}} & \ldots & 0 &\frac{(1-\lambda_r)}{4^{r'+1}}  & 0 & \frac{(1-\lambda_r)}{4^{r'}}  & \ddots & 1 & \text{} & \text{} & & & \\ \hline
2r+1 &  0 & \frac{(1-\lambda_r)}{4^r}  & 0 & \ldots & \frac{(1-\lambda _r)}{4^{r'+2}} & 0 & \frac{(1-\lambda _r)}{4^{r'+1}} & 0 & \ldots & 0 & 1 & \text{} &  &  & & \\ 
2r+2  &  0 & 0 & \frac{(1-\lambda_r)}{4^r} & \ldots &0  &  \frac{(1-\lambda _r)}{4^{r'+2}} & 0 & \frac{(1-\lambda _r)}{4^{r'+1} }   & \ldots &   \frac{(1-\lambda _r)}{4} & 0 & 1 & \text{} &  & &  \\
\vdots &  \vdots & \vdots & \vdots &\vdots & \vdots & \vdots & \vdots & \vdots  & \vdots & \vdots & \vdots & \vdots & \vdots & \ddots & &  \\
3r+1 &  0 & 0 & 0 &0 & 0 & 0 & \frac{(1-\lambda_r)}{4^r} & 0 & \ldots &0&\frac{(1-\lambda _r)}{4^{r'+1} }& 0 & \frac{(1-\lambda _r)}{4^{r'} }  & \ldots  &  1 & \ldots  \\ \hline
\vdots &  \vdots & \vdots & \vdots & \vdots & \vdots & \vdots & \vdots & \vdots & \vdots & 
\vdots & \vdots & \vdots &  \vdots & \vdots & \vdots & \vdots \\ 
\end{array}$
\end{center}
\label{default}
\end{table}%
\end{landscape}
\begin{theorem}\label{CCTraCheDig}
CC for the $r$th-perturbed {\it by translation} case $(r\geq 0)$ written by diagonal
\begin{eqnarray}
&& \textrm{Diagonal\ 0:}  \ \ \lambda^{t}_{n,n}=1\ ,\ n\geq 0\ ;\label{Diag0CoRec}\\
&& \textrm{Diagonal\ $2i-1$:} \ \  \lambda^{t}_{n,n-2i+1}=0\ ,\ 2i-1\leq n<r+i\ ; \label{Diag2im11CoRec}\\
&& \ \ \ \ \ \ \ \ \ \ \ \ \ \ \ \ \ \ \ \  \ \  \lambda^{t}_{n,n-2i+1}=-\frac{\mu_r}{4^{i-1}}\ ,\ n\geq r+i\ ;\ i=1(1)r+1\ ;\label{Diag2im12CoRec}\\
&& \textrm{Diagonal\ $2i$:} \ \ \lambda^{t}_{n,n-2i}=0\ ,\ n\geq 2i\ ,\ i=1(1)r\ ; \label{Diag2iCoRec}\\
&& \textrm{Diagonal\ $m$:} \ \  \lambda^{t}_{n,n-m}=0\ ,\ n\geq m\ ,\ m\geq 2r+2\ .\label{DiagmCoRec}
\end{eqnarray}
\end{theorem}
\begin{proof}
We shall do a demonstration by diagonal proving that the elements given by these formulas are solutions of recurrence relations of Lemma \ref{lemma1}. We remark that (\ref{GRRGenCoRecR}) involves five CC and (\ref{GRRGenCoRecN}) four. We shall reason by induction: first we treat the initial diagonals of orders 0, 1 and 2, after that, we deal with the diagonals $2i-1$ and $2i$ using as induction assumption the diagonals $2i-3$ and $2i-2$. In each one, we begin by proving an initial element and thereafter we show that a generic subsequent element coincides with the preceding one in the same diagonal. 

{\it Diagonal 0} : As both polynomials sequences are monic, (\ref{Diag0CoRec}) is verified.
Since the perturbation occurs at order $r$ of the recurrence coefficients, it will only affect polynomials of degrees greater than or equal to $r+1$ (see (\ref{recOrto})), in such a way that $P_n^t(\mu_r;r)\equiv P_n$, $n=0(1)r$, which is equivalent to 
\begin{equation}\label{zerotrian}
\lambda^{t}_{n,m}=0\ ,\ 0\leq m\leq n-1\ , \ \lambda^{t}_{n,n}=1\ ,\ n=0(1)r\ ,
\end{equation}
which correspond to the first $r+1$ rows in all tables.
\begin{center}
\begin{tabular}{|c|c|}\hline
\includegraphics[width=6.5cm]{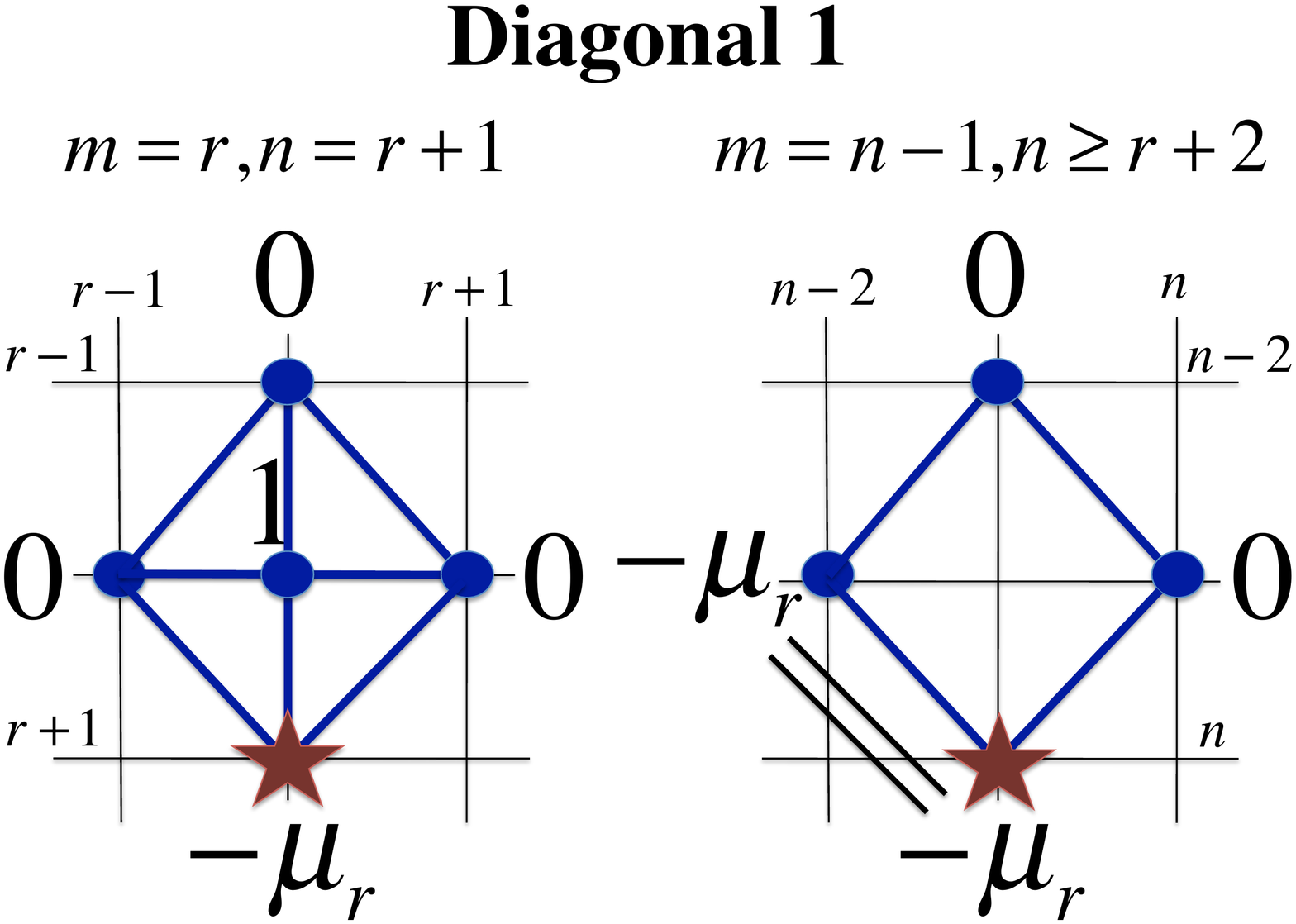} & 
\includegraphics[width=6.5cm]{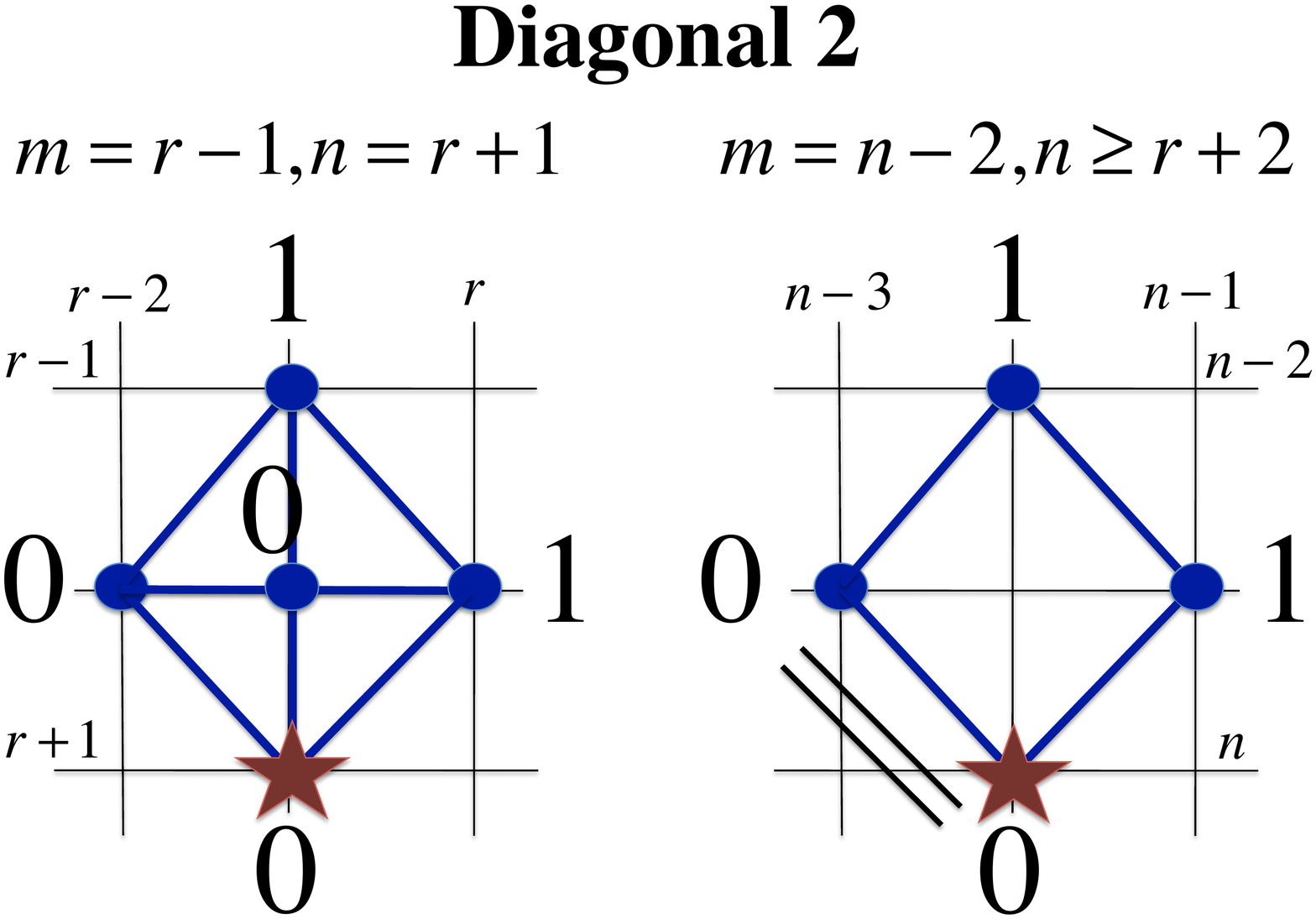} \\ \hline 
\end{tabular}
\end{center}
{\bf Diagonal 1}: Taking $i=1$ in (\ref{Diag2im11CoRec}) and (\ref{Diag2im12CoRec}), we get:
$$\lambda^{t}_{n,n-1}=0\ ,\ 1\leq n\leq r\ ;\ \lambda^{t}_{n,n-1}=-\mu_r\ ,\ n\geq r+1\ .$$
The first part is a consequence of (\ref{zerotrian}). With respect to the second part, for $n=r+1$, we obtain $\lambda^{t}_{r+1,r}=-\mu_r$ that satisfy the recurrence (\ref{GRRGenCoRecR}) for $m=r$ as follow
$$\lambda^{t}_{r+1,r}=-\mu_r\lambda^{t}_{r,r}+\frac{1}{4}\left(-\lambda^{t}_{r-1,r}+\lambda^{t}_{r,r+1}\right)+\lambda^{t}_{r,r-1}\Leftrightarrow  -\mu_r=-\mu_r\times 1+\frac{1}{4}\left(-0+0\right)+0\ ,$$
taking into account  (\ref{linkcoeffnms}), and (\ref{zerotrian}) for $n=r$ and $m=r-1$.
For $n\geq r+2$, we use the other recurrence (\ref{GRRGenCoRecN}) for $m=n-1$ and we get
$$\lambda^{t}_{n,n-1}=\frac{1}{4}\left(-\lambda^{t}_{n-2,n-1}+\lambda^{t}_{n-1,n}\right)+\lambda^{t}_{n-1,n-2}
=\frac{1}{4}\left(-0+0\right)+\lambda^{t}_{n-1,n-2}=\ldots= \lambda^{t}_{r+1,r},$$
considering (\ref{linkcoeffnms}).

{\bf Diagonal 2}: Supposing $r\geq 1$, taking $i=1$ in (\ref{Diag2iCoRec}), we obtain $\lambda^{t}_{n,n-2}=0,\ n\geq2$. For $n=2(1)r$, this is assured by (\ref{zerotrian}). In order to prove it for $n=r+1$, we should use the relation (\ref{GRRGenCoRecR}) for $m=r-1$
$$\lambda^{t}_{r+1,r-1}=-\mu_r\lambda^{t}_{r,r-1}+\frac{1}{4}\left(-\lambda^{t}_{r-1,r-1}+\lambda^{t}_{r,r}\right)+\lambda^{t}_{r,r-2}\Leftrightarrow 0=-\mu_r\times0+\frac{1}{4}\left(-1+1\right)+0\ ,$$
using (\ref{zerotrian}) for $n=r$ and $m=r-2,r-1$, and (\ref{Diag0CoRec}).
For $n\geq r+2$, we use the relation (\ref{GRRGenCoRecN}) for $m=n-2$
\begin{eqnarray}
\lambda^{t}_{n,n-2} & = &\frac{1}{4}\left(-\lambda^{t}_{n-2,n-2}+\lambda^{t}_{n-1,n-1}\right)+\lambda^{t}_{n-1,n-3}=\frac{1}{4}\left(-1+1\right)+\lambda^{t}_{n-1,n-3}=\notag\\
&&\lambda^{t}_{n-1,n-3}=\ldots =\lambda^{t}_{r+1,r-1}\ ,\notag
\end{eqnarray}
considering (\ref{Diag0CoRec}).
\begin{center}
\begin{tabular}{|c|c|}\hline
\includegraphics[width=6.5cm]{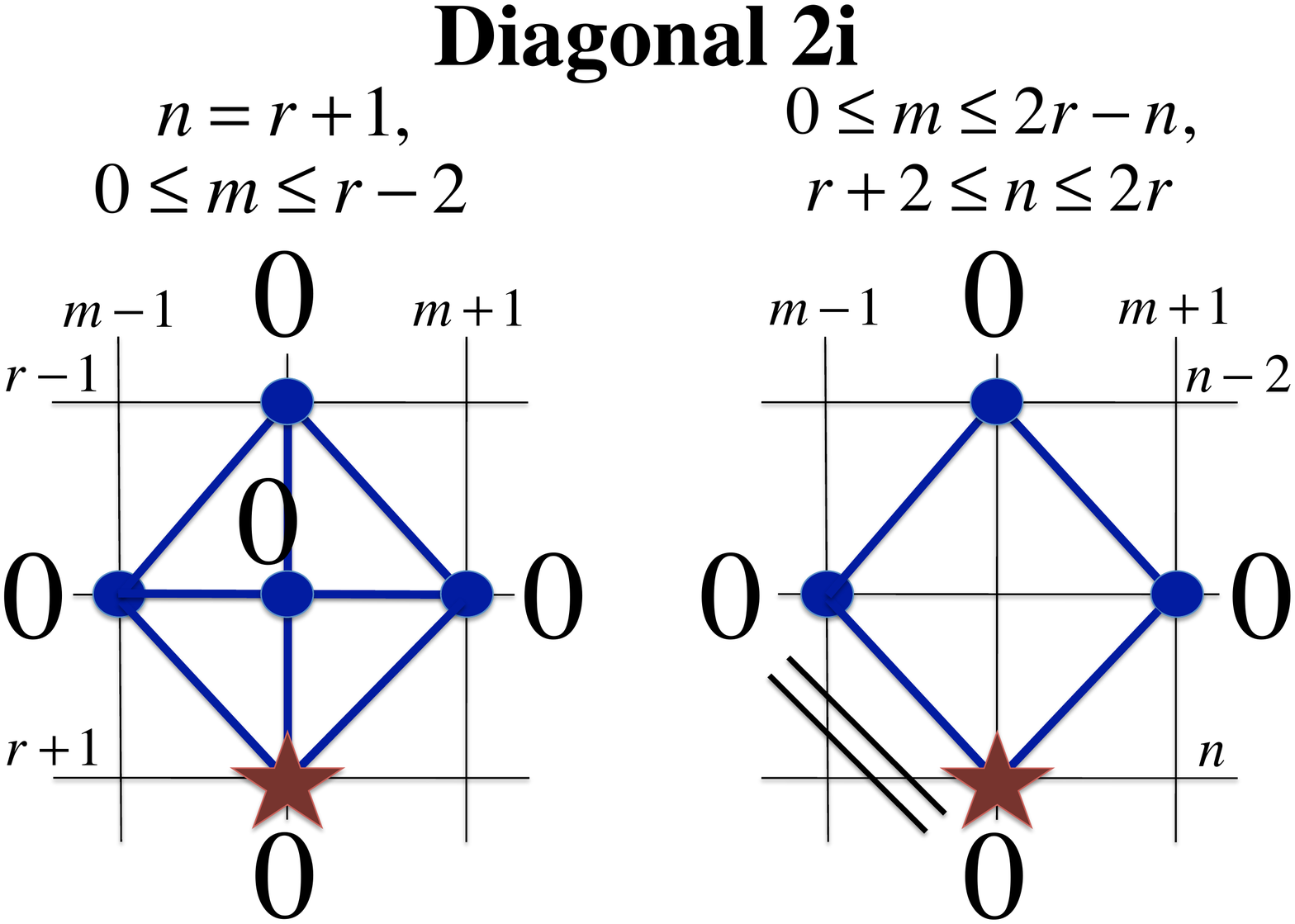} & 
\includegraphics[width=6.5cm]{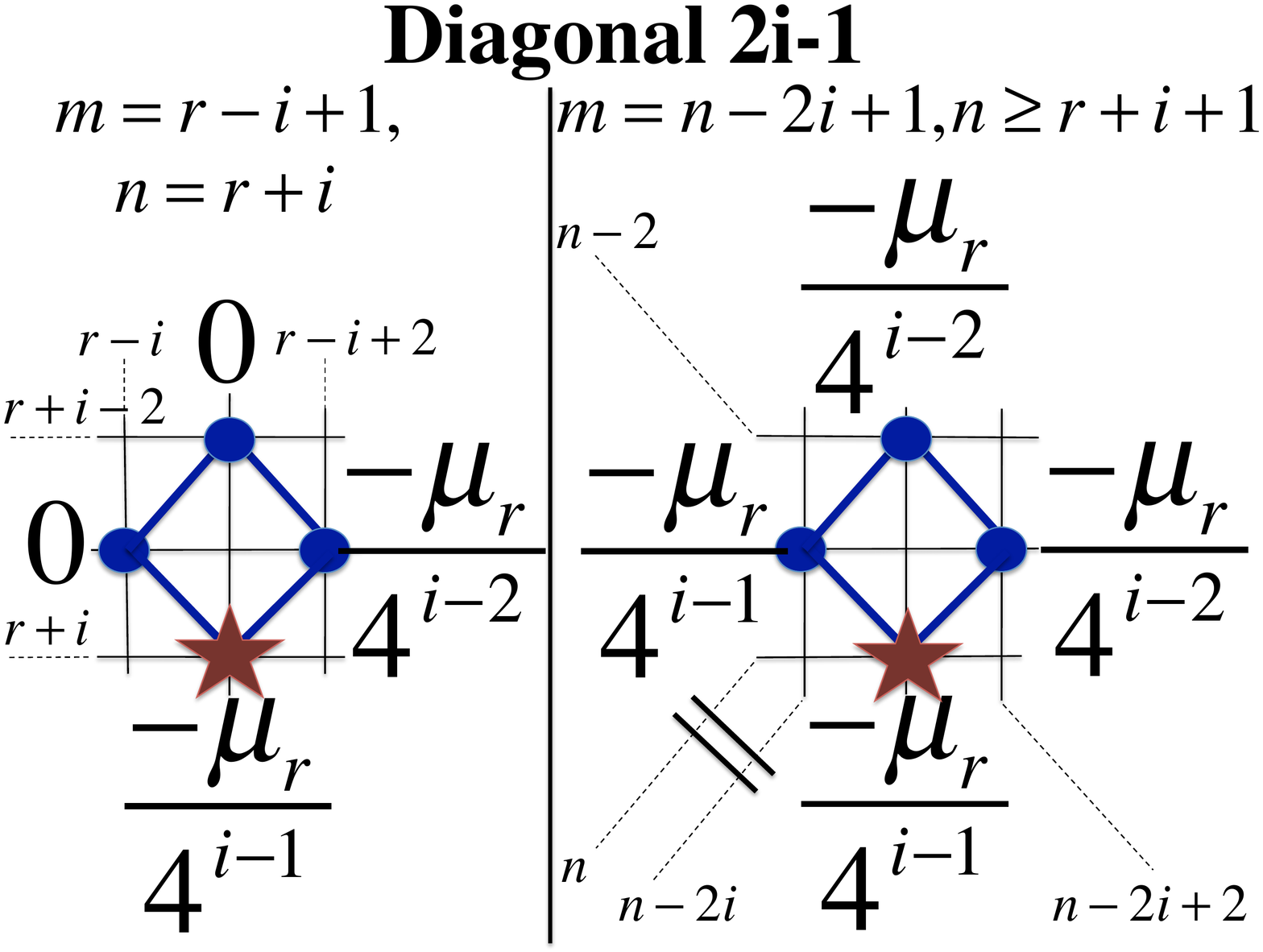} \\ \hline 
\end{tabular}
\end{center}

{\bf Diagonal $2i$, $i\geq 2$}:
At this point, we can easily conclude that $\lambda_{r+1,m}=0$, $0\leq m\leq r-2$ by application of the relation (\ref{GRRGenCoRecR}), because all elements involved are zero (from this point on, (\ref{GRRGenCoRecR}) will not be needed anymore).
The same situation occurs for $\lambda^{t}_{n,m}=0$, $r+2\leq n\leq 2r$, $0\leq m\leq 2r-n$ using this time (\ref{GRRGenCoRecN}). We have just proved the finite triangle of zeros appearing after row of order $r$.  Now, using (\ref{GRRGenCoRecN}), it is trivial to realise that all diagonals of even order $2i$ $(i\geq2)$ are null, due to the reason already invoked, therefore we have showed (\ref{Diag2iCoRec}) and a part of (\ref{DiagmCoRec}).

{\bf Diagonal $2i-1$, $1\leq i\leq r+1$}: The first part given by (\ref{Diag2im11CoRec}) belongs to the triangle of zeros. Let us work on the second part (\ref{Diag2im12CoRec}). We begin by proving the first nonzero element
for $n=r+i$ and $m=n-2i+1=r-i+1$, $\lambda^{t}_{r+i,r-i+1}=-\frac{\mu_r}{4^{i-1}}$. We use the relation 
(\ref{GRRGenCoRecN}) and we obtain
\begin{eqnarray}
&&\lambda^{t}_{r+i,r-i+1} = \frac{1}{4}\left(-\lambda^{t}_{r+i-2,r-i+1}+\lambda^{t}_{r+i-1,r-i+2}\right)+\lambda^{t}_{r+i-1,r-i}\Leftrightarrow\notag\\
&& -\frac{\mu_r}{4^{i-1}}=\frac{1}{4}\left(0-\frac{\mu_r}{4^{i-2}} \right)+0\ ,\notag
\end{eqnarray}
on accounting of (\ref{zerotrian}), the finite triangle of zeros and (\ref{linkcoeffnms}), and by the induction hypothesis for $i-1$, e.g., the preceding odd diagonal of order $2i-3$. For the rest of the diagonal, we write
 (\ref{GRRGenCoRecN}) for $m=n-2i+1$ and applying the induction assumption, we get
\begin{eqnarray}
&&\lambda^{t}_{n,n-2i+1} =\frac{1}{4}\left(-\lambda^{t}_{n-2,n-2i+1}+\lambda^{t}_{n-1,n-2i+2}\right)+\lambda^{t}_{n-1,n-2i}=\notag\\
&& \ \ \ \ \ \ \ \ \ \ \  \frac{1}{4}\left( \frac{\mu_r}{4^{i-2}} -\frac{\mu_r}{4^{i-2}}   \right)+\lambda^{t}_{n-1,n-2i}=
\lambda^{t}_{n-1,n-2i}=\ldots=\lambda^{t}_{r+i,r-i+1}=-\frac{\mu_r}{4^{i-1}}. \notag
\end{eqnarray}
\begin{center}
\begin{tabular}{|c|c|}\hline
\includegraphics[width=6.5cm]{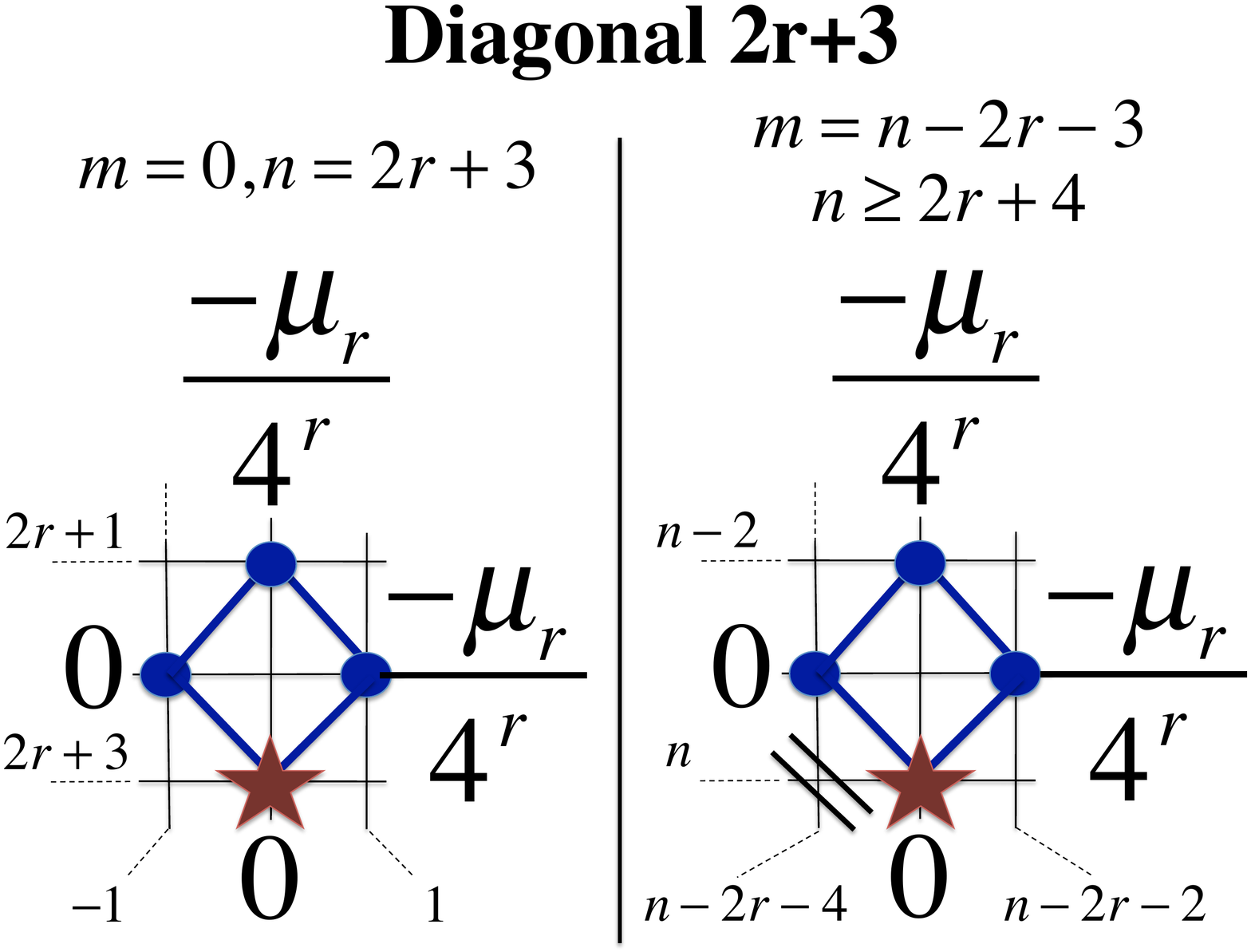} &
\includegraphics[width=6.5cm]{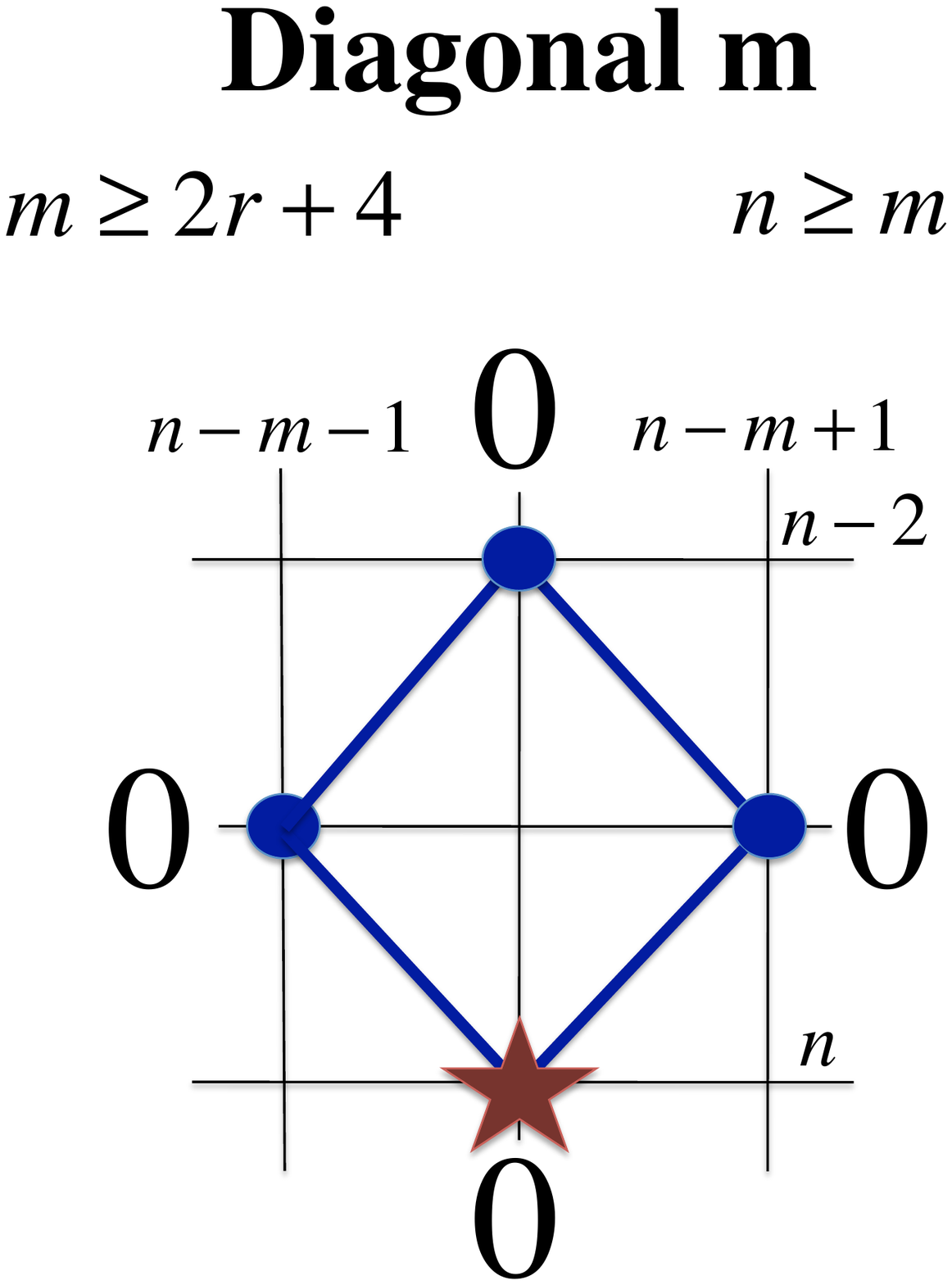}
 \\ \hline 
\end{tabular}
\end{center}

{\bf Diagonal $2r+3$}: 
We have to do a final effort to show that this diagonal is null unlike the odd preceding one of order $2r+1$; after that, applying (\ref{GRRGenCoRecN}), it will be trivial to realize that all odd diagonals of greater order are null and (\ref{DiagmCoRec}) will be entirely verified. Essentially this occurs because $\lambda^{t}_{2r+2,-1}=0$ from (\ref{linkcoeffnms}). In fact, using (\ref{GRRGenCoRecN}) for $n=2r+3$ and applying (\ref{Diag2im12CoRec}) for $i=r+1$, two elements of the diagonal $2r+1$ appear and we obtain
$$\lambda^{t}_{2r+3,0}=\frac{1}{4}\left(-\lambda^{t}_{2r+1,0}+\lambda^{t}_{2r+2,1}\right)+\lambda^{t}_{2r+2,-1}=\frac{1}{4}\left(\frac{\mu_r}{4^r}-\frac{\mu_r}{4^r}\right)+0=0.$$
For the other elements of this diagonal $2r+3$, from (\ref{GRRGenCoRecN}), for $n\geq2r+4$, follows
\begin{eqnarray}
&& \lambda^{t}_{n,n-2r-3}=\frac{1}{4}\left(-\lambda^{t}_{n-2,n-2r-3}+\lambda^{t}_{n-1,n-2r-2}\right)+\lambda^{t}_{n-1,n-2r-4}=\notag\\
&&
\ \ \ \ \ \ \ \ \ \ \  \frac{1}{4}\left(\frac{\mu_r}{4^r}-\frac{\mu_r}{4^r}\right)+\lambda^{t}_{n-1,n-2r-4}=\lambda^{t}_{n-1,n-2r-4}=\ldots=\lambda^{t}_{2r+3,0}=0\ .\notag
\end{eqnarray}
\end{proof}
\begin{theorem}\label{CCDilCheDig}
CC for the $r$th-perturbed {\it by dilatation} case $(r\geq 1)$ written by diagonal
\begin{eqnarray}
&&\textrm{Diagonal\ 0:}\ \lambda^{d}_{n,n}=1\ ,\ n\geq 0\ ;\notag\\
&& \textrm{Diagonal\ $2i-1$:}\ \lambda^{d}_{n,n-2i+1}=0\ ,\ n\geq 2i-1\ ,\ i=1(1)r\ ; \label{Diag2im10DilRec}\\
&& \textrm{Diagonal\ $2i$:}\ \lambda^{d}_{n,n-2i}=0\ ,\ 2i\leq n<r+i\ ;\label{Diag2im11DilRec}\\
&&\ \ \ \ \ \ \ \ \ \ \ \ \ \ \ \ \lambda^{d}_{n,n-2i}=\frac{(1-\lambda_r)}{4^{i}}\ ,\ n\geq r+i,\ i=1(1)r\ ;\label{Diag2im12DilRec}\\
&& \textrm{Diagonal\ $m$:}\ \lambda^{d}_{n,n-m}=0\ ,\ n\geq m\ ,\ m\geq 2r+1\ .\label{DiagmDilRec}
\end{eqnarray}

\end{theorem}
\begin{proof}
This proof is analogous to the demonstration of the preceding theorem, but
this case is more simple, because both sequences are symmetric and the two relations of Lemma \ref{lemma2} involve the same four CC. In this case, it is trivial to realize that all odd diagonals are null, because all CC involved in computations are zero; thus, (\ref{Diag2im10DilRec}) and (\ref{DiagmDilRec}) for $m$ odd are proved.
We omit the details concerning the initial triangle of zeros, which corresponds to a part of 
(\ref{Diag2im10DilRec}) and to (\ref{Diag2im11DilRec}), because they are trivial. With respect to even diagonals, we have to distinguish three situations corresponding to diagonals of orders 2, $2i$ and $2r+2$. As before, we shall apply a reasoning by induction.
\begin{center}
\begin{tabular}{|c|c|}\hline
\includegraphics[width=6.5cm]{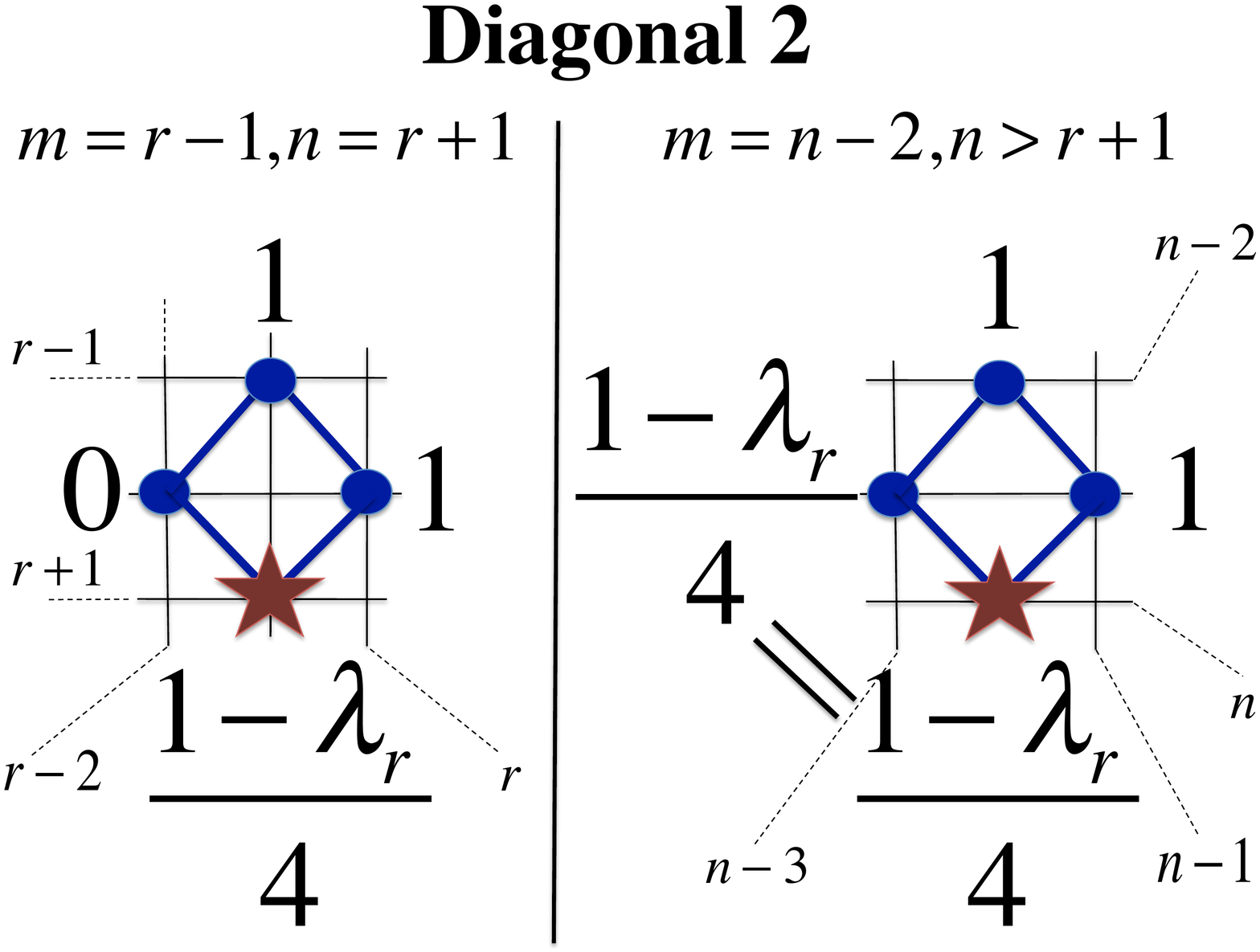}  \\ \hline 
\end{tabular}
\end{center}
{\bf Diagonal $2$}: Taking $i=1$ in (\ref{Diag2im12DilRec}), we get $\lambda^{d}_{n,n-2}=\frac{(1-\lambda_r)}{4}$, $n\geq r+1$. Let us prove the first element for $n=r+1$ and $m=r-1$ by the relation (\ref{GRRGenCoDilR})
$$\lambda^{d}_{r+1,r-1}=\frac{1}{4}\left( -\lambda_r\lambda^{d}_{r-1,r-1}+\lambda^{d}_{r,r}\right)+\lambda^{d}_{r,r-2}\Leftrightarrow \frac{(1-\lambda_r)}{4}=\frac{1}{4}\left(-\lambda_r\times1+1\right)+0,$$
on accounting of (\ref{linkcoeffnms}) and (\ref{Diag2im11DilRec}) already proved. 
For the other elements, we apply the relation (\ref{GRRGenCoDilN}) for $m=n-2$ and 
using (\ref{linkcoeffnms}), we get
\begin{eqnarray}
&&\lambda^{d}_{n,n-2}=\frac{1}{4}\left(-\lambda^{d}_{n-2,n-2}+\lambda^{d}_{n-1,n-1}\right)+\lambda^{d}_{n-1,n-3}
=\frac{1}{4}\left(-1+1\right)+\lambda^{d}_{n-1,n-3}=\notag\\
&& \ \ \ \ \ \ \ \ \ \ \  \lambda^{d}_{n-1,n-3}=\ldots=\lambda^{d}_{r+1,r-1}=\frac{(1-\lambda_r)}{4}\ . \notag
\end{eqnarray}
\begin{center}
\begin{tabular}{|c|c|}\hline
\includegraphics[width=6.5cm]{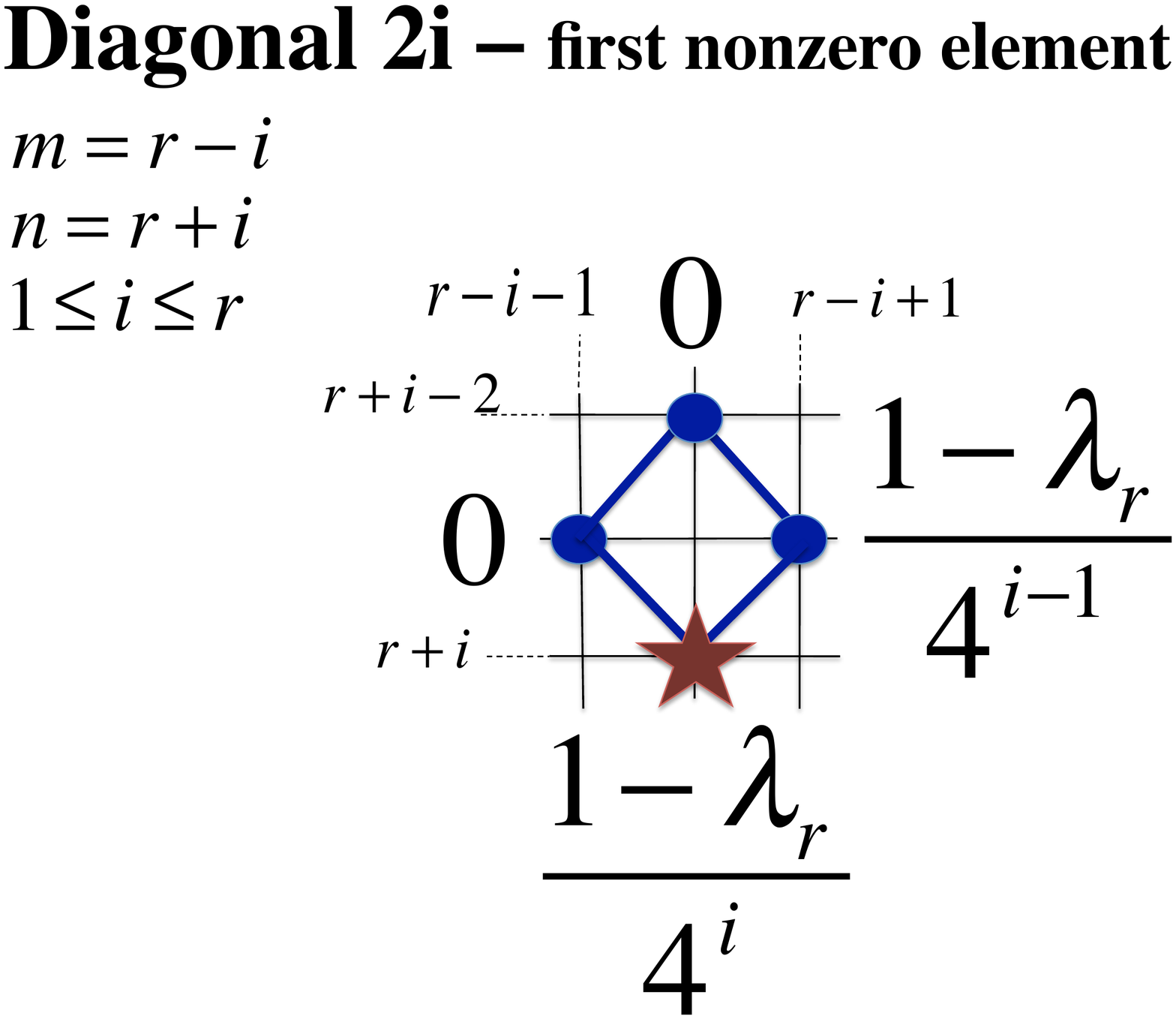} & 
\includegraphics[width=6.5cm]{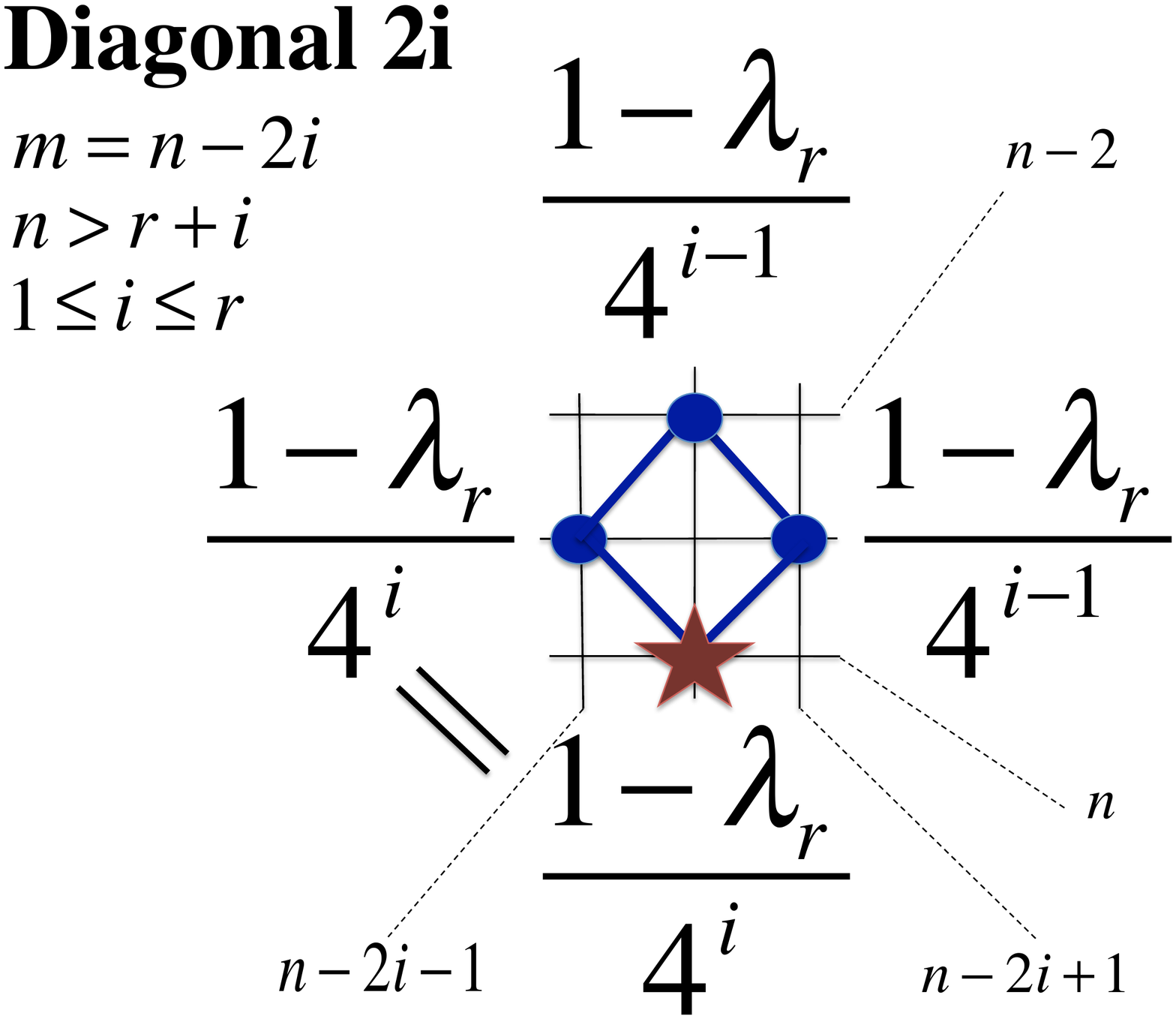} \\ \hline 
\end{tabular}
\end{center}
{\bf Diagonal $2i$, $1\leq i\leq r$}:  We begin by proving the first nonzero element: taking $n=r+i$ in (\ref{Diag2im12DilRec}), we get $\lambda^{d}_{r+i,r-i}=\frac{(1-\lambda_r)}{4^{i}}$.  Using (\ref{GRRGenCoDilN}) for $n=r+i$ and $m=r-i$, we obtain
\begin{eqnarray}
\lambda^{d}_{r+i,r-i} & = & \frac{1}{4}\left(-\lambda^{d}_{r+i-2,r-i}+\lambda^{d}_{r+i-1,r-i+1}\right)+\lambda^{d}_{r+i-1,r-i-1}\Leftrightarrow\notag\\
 \frac{(1-\lambda_r)}{4^{i}} & = & \frac{1}{4}\left(0+\frac{(1-\lambda_r)}{4^{i-1}}\right)+0\ , \notag
\end{eqnarray}
taking into account the triangle of zeros and the hypothesis of induction for $i-1$. For the rest of the diagonal ($n>r+i$), we write (\ref{GRRGenCoDilN}) for $m=n-2i$ and we apply the induction assumption two times getting
\begin{eqnarray}
\lambda^{d}_{n,n-2i}& =& \frac{1}{4}\left(-\lambda^{d}_{n-2,n-2i}+\lambda^{d}_{n-1,n-2i+1}\right)+\lambda^{d}_{n-1,n-2i-1}=\notag\\
&&\frac{1}{4}\left(-\frac{(1-\lambda_r)}{4^{i-1}}+\frac{(1-\lambda_r)}{4^{i}}\right)+\lambda^{d}_{n-1,n-2i-1}=\notag\\
&&\lambda^{d}_{n-1,n-2i-1}=\ldots=\lambda^{d}_{r+i,r-i}=\frac{(1-\lambda_r)}{4^{i}}\ .\notag
\end{eqnarray}
\begin{center}
\begin{tabular}{|c|c|}\hline
\includegraphics[width=6.5cm]{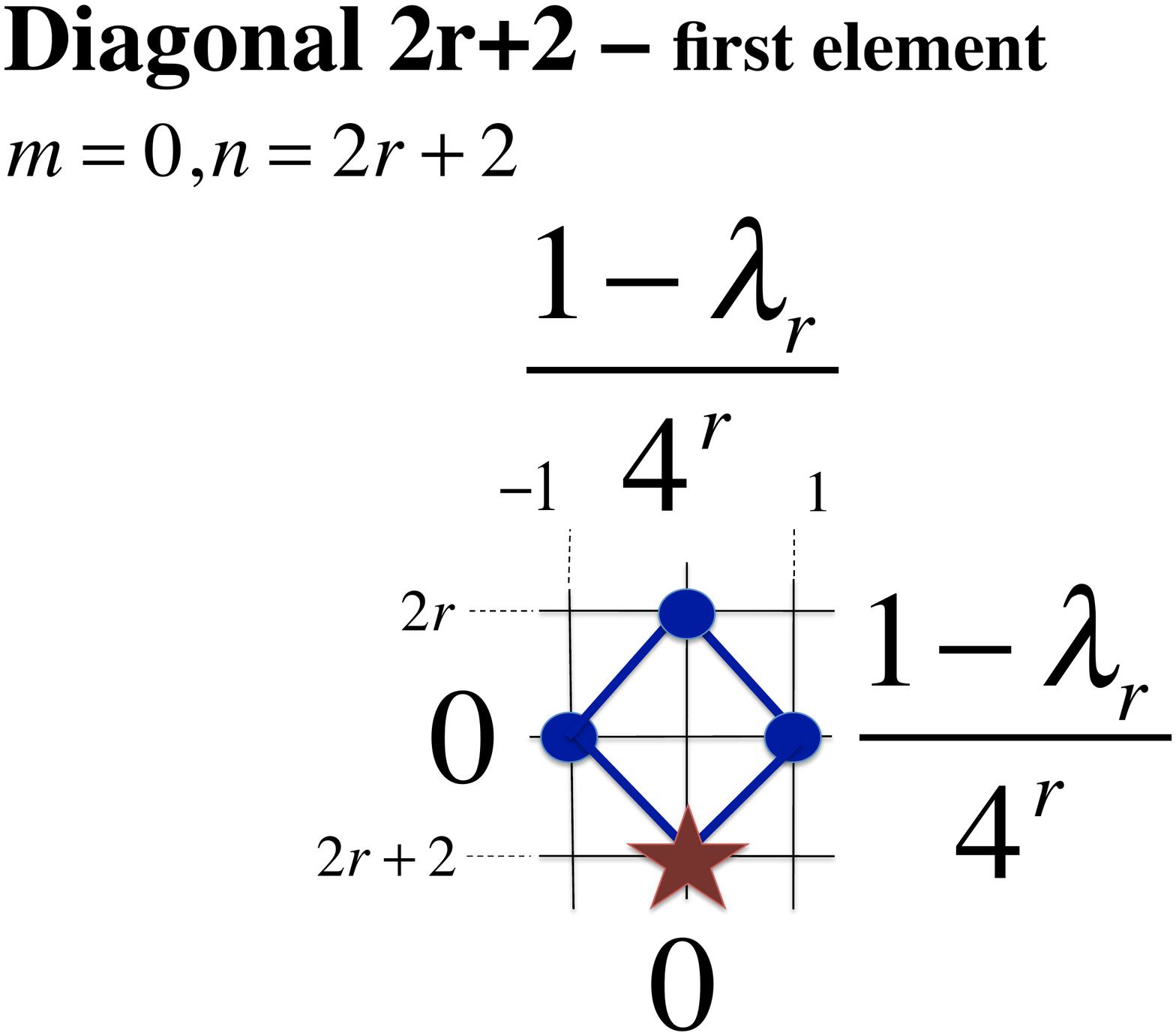} & 
\includegraphics[width=6.5cm]{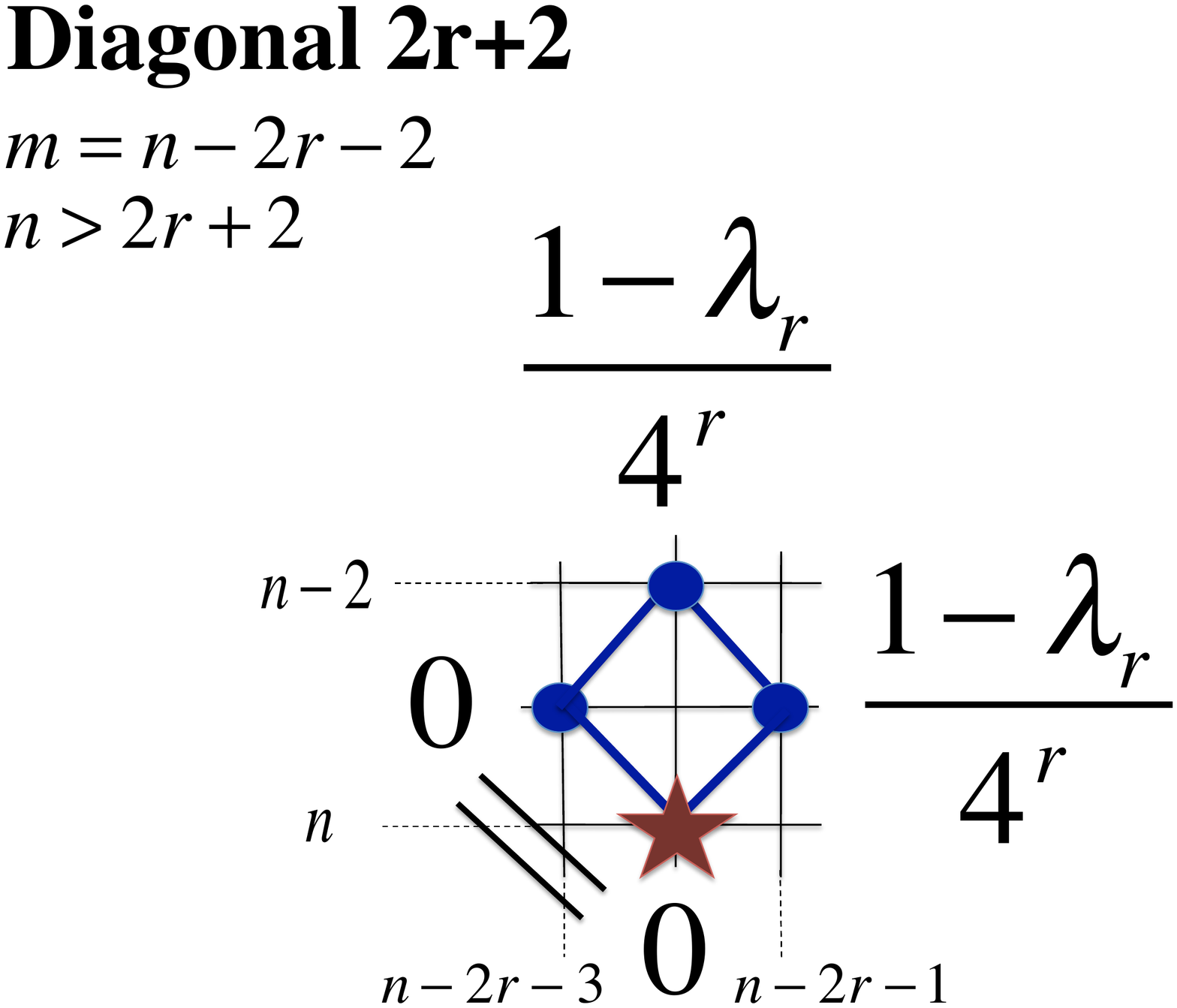} \\ \hline 
\end{tabular}
\end{center}
{\bf Diagonal $2r+2$}: As before, we begin by proving the first element. We write (\ref{GRRGenCoDilN}) for $n=2r+2$ and $m=0$, obtaining
$$\lambda^{d}_{2r+2,0}=\frac{1}{4}\left(-\lambda^{d}_{2r,0}+\lambda^{d}_{2r+1,1}\right)+\lambda^{d}_{2r+1,-1}=
\frac{1}{4}\left(-\frac{(1-\lambda_r)}{4^{r}}+\frac{(1-\lambda_r)}{4^{r}}\right)+0=0,$$
from (\ref{Diag2im12DilRec}) for $i=r$, $n=2r$ and $n=2r+1$; and (\ref{linkcoeffnms}). For the other elements, we write 
(\ref{GRRGenCoDilN}) for $m=n-2r-2$ and $n> 2r+2$, and we use (\ref{Diag2im12DilRec}) for $i=r$ getting
\begin{eqnarray}
\lambda^{d}_{n,n-2r-2}&=&\frac{1}{4}\left(-\lambda^{d}_{n-2,n-2r-2}+\lambda^{d}_{n-1,n-2r-1}\right)+\lambda^{d}_{n-1,n-2r-3}=\notag\\
&&
\frac{1}{4}\left(-\frac{(1-\lambda_r)}{4^{r}}+\frac{(1-\lambda_r)}{4^{r}}\right)+\lambda^{d}_{n-1,n-2r-3}=\notag\\
&&\lambda^{d}_{n-1,n-2r-3}=\ldots=\lambda^{d}_{2r+2,0}=0\ .\notag
\end{eqnarray}
After this, it is evident that all even diagonals of greater order are null, because all CC involved  in computations are zero; thus we have proved the remainder part of (\ref{DiagmDilRec}).
\end{proof}
In order to obtain the CR (\ref{PntildePnu}), we have to consider in the preceding tables CC $\lambda_{n,m}$ by row, e.g., with $n$ fixed and $m=0(1)n$. Doing this, we easily obtain next two propositions. In both cases, we remark that there is one set of $r+1$ trivial initial CR ((\ref{CR_Recursion_inic_cond_r_Tra}) and 
(\ref{CR_Dilatation_inic_cond_r_Dil})) and another set of initial CR ((\ref{CR_Recursion_k}) and (\ref{CR_Dilatation_k})) corresponding to the above mentioned triangle of zeros. 
Main CR (\ref{CR_Recursion_n2r1}) and (\ref{CR_Dilatation_n2r}) give a perturbed polynomial in terms of $r+2$ and $r+1$ Chebyshev polynomials, respectively.
From the degrees of polynomials involved in these relations, considering $n$ even or $n$ odd, and taking into account the symmetry of $\{P_n\}_{n\geq0}$, it is easy to realize a fact that we already know, that perturbed by dilatation are symmetric, but the same is not true for perturbed by translation. 
\begin{proposition}\label{CCTraChe}
CR and CC for the $r$th-perturbed {\it by translation} case $(r\geq 0)$
\begin{eqnarray}
&& P^{t}_{k}(\mu_r;r)(x)=P_k(x)\ , \ k=0(1)r\ ,\ \label{CR_Recursion_inic_cond_r_Tra} \\
&&  \ \ \ \ \ \  \lambda^{t}_{k,k}=1\ ; \ \lambda^{t}_{k,m}=0\ ,\ 0\leq m\leq k-1\ . \notag \\ 
&& P^{t}_{k}(\mu_r;r)(x)=P_k(x)-\mu_r\sum_{i=0}^{k-r-1}\frac{1}{4^i}P_{k-2i-1}(x)\ , \ k=r+1(1)2r\ ,\label{CR_Recursion_k} \\
&& \ \ \ \ \ \ \lambda^{t}_{k,k}=1\ ;\  
\lambda^{t}_{k,k-2i-1}=-\mu_r/4^i\ ,\ i=0(1)k-r-1\ ; \notag\\
&&\ \ \ \ \ \ \lambda^{t}_{k,k-2i-2}=0\ ,\ i=0(1)k-r-1\ ; \ \lambda^{t}_{k,m}=0\ ,\ 0\leq m \leq 2r-k-1\ .\notag \\ 
&& P^{t}_{n+2r+1}(\mu_r;r)(x)=P_{n+2r+1}(x)-\mu_r\sum_{i=0}^{r}\frac{1}{4^i}P_{n+2(r-i)}(x)\ , \ n\geq 0\ ,\label{CR_Recursion_n2r1} \\
&& \ \ \ \ \ \ \lambda^{t}_{n+2r+1,n+2r+1}=1\ ; \  \lambda^{t}_{n+2r+1,n+2(r-i)}=-\mu_r/4^i\ ,\ i=0(1)r\ ; \notag\\
&&\ \ \ \ \ \ \lambda^{t}_{n+2r+1,n+2(r-i)-1}=0\ ,\ i=0(1)r-1\ ;
  \ \lambda^{t}_{n+2r+1,m}=0\ ,\ 0\leq m\leq n-1\ . \notag 
\end{eqnarray}
\end{proposition}
\begin{proposition}\label{CCDilChe}
CR and CC for the $r$th-perturbed {\it by dilatation} case $(r\geq 1)$
\begin{eqnarray}
&& P^{d}_{k}(\lambda_r;r)(x)=P_k(x)\ , \ k=0(1)r\ ,\label{CR_Dilatation_inic_cond_r_Dil}\\
&&\ \ \ \ \ \  \lambda^{d}_{k,k}=1\ ; \ \lambda^{d}_{k,m}=0\ ,\ 0\leq m\leq k-1\ . \notag 
\end{eqnarray}
\begin{eqnarray} 
&& P^{d}_{k}(\lambda_r;r)(x)=P_k(x)+\frac{1-\lambda_r}{4}\sum_{i=1}^{k-r}\frac{1}{4^{i-1}}P_{k-2i}(x)\ , \ k=r+1(1)2r-1\ ,\label{CR_Dilatation_k}\\  
&& \ \ \ \ \ \ \lambda^{d}_{k,k}=1\ ; \
 \lambda^{d}_{k,k-2i+1}=0,\ i=1(1)k-r\ ; \notag \\ 
 && \ \ \ \ \ \ \lambda^{d}_{k,k-2i}=(1-\lambda_r)/4^i,\ i=1(1)k-r \ ; \ \lambda^{d}_{k,m}=0\ , \ 0\leq m \leq 2r-k-1\ .\notag\\
&& P^{d}_{n+2r}(\lambda_r;r)(x)=P_{n+2r}(x)+\frac{1-\lambda_r}{4}\sum_{i=1}^{r}\frac{1}{4^{i-1}}P_{n+2(r-i)}(x)\ , \ n\geq 0\ ,\label{CR_Dilatation_n2r} \\
&& \ \ \ \ \ \ \lambda^{d}_{n+2r,n+2r}=1\ ; \ \lambda^{d}_{n+2r,n+2(r-i)+1}=0\ ,\ i=1(1)r
; \notag \\ 
&& \ \ \ \ \ \  \lambda^{d}_{n+2r,n+2(r-i)}=(1-\lambda_r)/4^i\ , \ i=1(1)r\ ; \
 \lambda^{d}_{n+2r,m}=0\ ,\ 0\leq m\leq n-1\ . \notag 
\end{eqnarray}
\end{proposition}

\vspace{0.25cm}

The following two corollaries give CR for perturbations by translation of first orders.

\begin{corollary} \label{corollary1}

CR for the {\it co-recursive case}
\begin{eqnarray} 
&& P^{t}_0\left(\mu_0;0\right)\equiv P_0\ ,\  P^{t}_{n+1}\left(\mu_0;0\right)(x) =  P_{n+1}(x)-\mu_0P_{n}(x)\ ,\ n \geq 0\ . \label{rrcPer0Tra}
\end{eqnarray}
CR for the  perturbed of order 1 {\it by translation} case
\begin{eqnarray} 
&& P^{t}_0\left(\mu_1;1\right)\equiv P_0\ ,\ P^{t}_1\left(\mu_1;1\right)\equiv P_1\ ;\ 
P^{t}_{2}\left(\mu_1;1\right)(x) = P_2(x)-\mu_1P_1(x)\ ;\label{rrcPer1Tra1} \\
&& P^{t}_{n+3}\left(\mu_1;1\right)(x)  = 
 P_{n+3}(x)-\mu_1P_{n+2}(x)-\frac{\mu_1}{4}P_{n}(x)\ ,\ n\geq 0\ .\label{rrcPer1Tra2}
\end{eqnarray}
CR for the perturbed of order 2 {\it by translation} case
\begin{eqnarray} 
&& P^{t}_0\left(\mu_1;1\right)\equiv P_0\ ,\ P^{t}_1\left(\mu_1;1\right)\equiv P_1\ ,\ P^{t}_2\left(\mu_1;1\right)\equiv P_2\ ;\notag\\
&& P^{t}_{3}\left(\mu_1;1\right)(x) = P_3(x)-\mu_1P_2(x)\ ,\ P^{t}_{4}\left(\mu_1;1\right)(x) = P_4(x)-\mu_1P_3(x)-\frac{\mu_1}{4}P_{1}(x)\ ;\notag \\ 
&& P^{t}_{n+5}\left(\mu_1;1\right)(x)  = 
 P_{n+5}(x)-\mu_1P_{n+4}(x)-\frac{\mu_1}{4}P_{n+2}(x)-\frac{\mu_1}{4^2}P_{n}(x)\ ,\ n\geq 0\ .\notag
\end{eqnarray}
\end{corollary}
Notice that (\ref{rrcPer0Tra}) corresponds to a well known relation \cite{Chihara,MARO 91},
furthermore, it gives as particular cases (\ref{CR_VP}) and (\ref{CR_WP}).
\begin{corollary} \label{corollary2}
CR for the perturbed of order 1 {\it by dilatation} case
\begin{eqnarray} 
&& P^{d}_{0}\left(\lambda_1;1\right)\equiv P_0\ ,\ P^{d}_{1}\left(\lambda_1;1\right) \equiv P_1\ ,\notag\\
&& P^{d}_{n+2}\left(\lambda_1;1\right)(x)=P_{n+2}(x)+\frac{1}{4}(1-\lambda_1)P_{n}(x)\ ,\ n \geq 0. \label{rrcPer1Dil}
\end{eqnarray}
CR for the perturbed of order 2 {\it by dilatation} case
\begin{eqnarray} 
&& P^{d}_0\left(\lambda_2;2\right)\equiv P_0\ ,\ P^{d}_1\left(\lambda_2;2\right)\equiv P_1\ ,\
P^{d}_{2}\left(\lambda_2;2\right) \equiv P_2\ ,\notag\\
&& P^{d}_{3}\left(\lambda_2;2\right)(x)  = P_3(x)+\frac{1}{4}(1-\lambda_2)P_1(x)\ , 
\label{rrcPer2Dil1}  \\
&&P^{d}_{n+4}\left(\lambda_2;2\right)(x) =  
P_{n+4}(x)+\frac{1}{4}(1-\lambda_2)P_{n+2}(x)+\frac{1}{4^2}(1-\lambda_2)P_{n}(x)\ ,\ n\geq 0.
\label{rrcPer2Dil2}
\end{eqnarray}
\end{corollary}
Notice that (\ref{rrcPer1Dil}) admits as particular case (\ref{CR_TP}).

\section{Connection coefficients and connection relations in terms of the canonical basis}\label{CCChe2}

In this section, our goal is to explicit the CC for $rth$-perturbed by translation and by dilatation of the Chebyshev polynomials of second kind in terms of the canonical basis: 
$$\lambda^{tX}_{n,m}:=\lambda^{P^tX}_{n,m}=\lambda_{n,m}(P^{t}(\mu_r;r)\leftarrow X)\  {\rm and} \ \lambda^{dX}_{n,m}:=\lambda^{P^dX}_{n,m}=\lambda_{n,m}(P^{d}(\lambda_r;r)\leftarrow X).$$

Our starting point are the CR given by  Propositions \ref{CCTraChe} and \ref{CCDilChe}, and the CR and the CC $C_{n,m}:=C_{n,m}(P\leftarrow X)$ stated by (\ref{P2n})-(\ref{C2n1}). First, we need a lemma.
\begin{lemma}\label{lemmaCCCRCan} 
\begin{equation}\label{lemmaCan}
\Lambda_{n}^{m}=\sum_{\mu=n}^{m}a_\mu\sum_{\nu=0}^{\mu}b_{\nu}=
\sum_{\nu=0}^{n-1}b_\nu\sum_{\mu=n}^{m}a_{\mu}+
\sum_{\nu=n}^{m}b_\nu\sum_{\mu=\nu}^{m}a_{\mu}\ ,\ m\geq n\geq 0\ .
\end{equation}
\end{lemma}
\begin{proof}
The case $n=0$ corresponds to a well known situation
\begin{equation}\label{lemmaCan1}
\Lambda_{0}^{m}=\sum_{\mu=0}^{m}a_\mu\sum_{\nu=0}^{\mu}b_{\nu}=
\sum_{\nu=0}^{m}b_\nu\sum_{\mu=\nu}^{m}a_{\mu}\ ,\ m\geq 0\ .
\end{equation}
For $n\geq1$, adding and subtracting a same quantity, we can write
\begin{eqnarray}
&&\Lambda_{n}^{m}=\sum_{\mu=0}^{n-1}a_\mu\sum_{\nu=0}^{\mu}b_{\nu}-
\sum_{\mu=0}^{n-1}a_\mu\sum_{\nu=0}^{\mu}b_{\nu}+\sum_{\mu=n}^{m}a_\mu\sum_{\nu=0}^{\mu}b_{\nu}=\notag \\
&&\ \ \ \ \ \ \ \  \sum_{\mu=0}^{m}a_\mu\sum_{\nu=0}^{\mu}b_{\nu}-
\sum_{\mu=0}^{n-1}a_\mu\sum_{\nu=0}^{\mu}b_{\nu}\ .   \notag
\end{eqnarray}
Now, we apply (\ref{lemmaCan1}) two times to the last member of the preceding equality, and then, in the first term obtained, we separate each of the two sums into two ones considering that $m\geq n$, and we get the desired result
\begin{eqnarray}
&& \Lambda_{n}^{m}=\sum_{\nu=0}^{m}b_{\nu}\sum_{\mu=\nu}^{m}a_\mu-
\sum_{\nu=0}^{n-1}b_{\nu}\sum_{\mu=\nu}^{n-1}a_\mu=\notag\\
&& \ \ \ \ \ \ \ \ \sum_{\nu=0}^{n-1}b_{\nu}\Big\{    \sum_{\mu=\nu}^{n-1}a_\mu+\sum_{\mu=n}^{m}a_\mu \Big\}+\sum_{\nu=n}^{m}b_{\nu}\sum_{\mu=\nu}^{m}a_\mu-
\sum_{\nu=0}^{n-1}b_{\nu}\sum_{\mu=\nu}^{n-1}a_\mu\notag\ .
\end{eqnarray}
\end{proof}

\subsection{Translation case}

\begin{theorem}\label{CCTraCheCanonical}
CC $\lambda_{n,m}^{tX}:=\lambda_{n,m}(P^t\leftarrow X)$ for the $r$th-perturbed {\it by translation}  case $(r\geq 0)$ in terms of the canonical basis. 
For $k=0(1)r$, it holds
\begin{eqnarray}
&& \lambda^{tX}_{k,\nu}=C_{k,\nu}\ ,\ \nu=0(1)k\ . \label{condiniciais_tra1} 
\end{eqnarray}

If  $k=2k'$, then for $k'=r'+1(1)r$, being $r=2r'$ or $r=2r'+1$, it holds
\begin{eqnarray}
&& \lambda^{tX}_{2k',2\nu}  = C_{2k',2\nu}\ ,\ \nu=0(1)k'\ ,\label{CCTraCheCanonical_condinic11}\\ 
&&  \lambda^{tX}_{2k',2\nu+1} = -\frac{\mu_{r}}{4^{(k'-1)}}\sum_{\mu=r-k'}^{k'-1} 4^{\mu}C_{2\mu+1,2\nu+1}\  ,\ \nu=0(1)r-k'\ ,\label{CCTraCheCanonical_condinic12}\\
&& \lambda^{tX}_{2k',2\nu+1} = -\frac{\mu_{r}}{4^{(k'-1)}}\sum_{\mu=\nu}^{k'-1} 4^{\mu}C_{2\mu+1,2\nu+1}\ ,\ \nu=r-k'+1(1)k'-1\ .\label{CCTraCheCanonical_condinic13}
\end{eqnarray}

If $k=2k'+1$, then for $k'=r'(1)r-1$, if $r=2r'$; or for $k'=r'+1(1)r-1$, if $r=2r'+1$, it holds
\begin{eqnarray}
&&  \lambda^{tX}_{2k'+1,2\nu+1}  = C_{2k'+1,2\nu+1}\ ,\ \nu=0(1)k'\ ,\label{CCTraCheCanonical_condinic21}\\ 
&& \lambda^{tX}_{2k'+1,2\nu} = -\frac{\mu_{r}}{4^{k'}}\sum_{\mu=r-k'}^{k'} 4^{\mu}C_{2\mu,2\nu} \  , \ \nu=0(1)r-k'\ ,\label{CCTraCheCanonical_condinic22}\\
&&  \lambda^{tX}_{2k'+1,2\nu} = -\frac{\mu_{r}}{4^{k'}}\sum_{\mu=\nu}^{k'} 4^{\mu}C_{2\mu,2\nu}\ ,\ \nu=r-k'+1(1)k'\ .\label{CCTraCheCanonical_condinic23}
\end{eqnarray}

For $n\geq 0$, it holds
\begin{eqnarray}
&&\lambda^{tX}_{2(n+r)+1,2\nu+1}  = C_{2(n+r)+1,2\nu+1}\ ,\ \nu=0(1)n+r\ ,\label{CCTraCheCanonical_eq1}\\ 
&&\lambda^{tX}_{2(n+r)+1,2\nu}  =  -\frac{\mu_r}{4^{(n+r)}}\sum_{\mu=n}^{n+r}4^{\mu}C_{2\mu,2\nu}\ ,\ \nu=0(1)n\ ,\label{CCTraCheCanonical_eq2}\\
&&\lambda^{tX}_{2(n+r)+1,2\nu}  =  
-\frac{\mu_r}{4^{(n+r)}}
\sum_{\mu=\nu}^{n+r}4^{\mu}C_{2\mu,2\nu}\ ,\ \nu=n+1(1)n+r\ .\label{CCTraCheCanonical_eq3} \\ \notag \\
&&\lambda^{tX}_{2(n+r+1),2\nu}  = C_{2(n+r+1),2\nu}\ ,\ \nu=0(1)n+r+1\ ,\label{CCTraCheCanonical_eq4}\\ 
&&\lambda^{tX}_{2(n+r+1),2\nu+1}  =  -\frac{\mu_r}{4^{(n+r)}}\sum_{\mu=n}^{n+r}4^{\mu}C_{2\mu+1,2\nu+1}\ ,\ \nu=0(1)n\ ,\label{CCTraCheCanonical_eq5}\\
&&\lambda^{tX}_{2(n+r+1),2\nu+1}  =  
-\frac{\mu_r}{4^{(n+r)}}
\sum_{\mu=\nu}^{n+r}4^{\mu}C_{2\mu+1,2\nu+1}\ ,\ \nu=n+1(1)n+r\ .\label{CCTraCheCanonical_eq6} 
\end{eqnarray}
\end{theorem}
\begin{proof}
From (\ref{CR_Recursion_inic_cond_r_Tra}),  
(\ref{P2n}) and (\ref{P2n1}) immediately follow the initial conditions (\ref{condiniciais_tra1}).
For deducing the remainder initial conditions (\ref{CCTraCheCanonical_condinic11})-(\ref{CCTraCheCanonical_condinic23}), we must consider two cases corresponding to
$k=2k'$ ($r=2r'$ and $r=2r'+1$) and to $k=2k'+1$ ($r=2r'$ and $r=2r'+1$).
We are going to demonstrate the first case, the other is similar. 
Taking $k=2k'$ in (\ref{CR_Recursion_k}), we get
\begin{eqnarray}
&& P^{t}_{2k'}(\mu_{r};r)(x)=P_{2k'}(x)-\mu_{r}\sum_{i=0}^{2k'-r-1}\frac{1}{4^{i}}P_{2(k'-i)-1}\ , k=r+1(1)2r \ .\label{eqPtP_01} 
\end{eqnarray}
If $r=2r'$, then $2k'=2r'+2(2)2r\Leftrightarrow k'=r'+1(1)r$. If $r=2r'+1$, then $2k'=2r'+2(2)2r\Leftrightarrow k'=r'+1(1)r$.
In (\ref{eqPtP_01}), we do the change of variable $\mu=k'-i$, $i=0(1)2k'-r-1$ and we use (\ref{P2n1}); after that we apply Lemma \ref{lemmaCCCRCan}, thus the preceding sum becomes
\begin{eqnarray}
&& \frac{1}{4^{(k'-1)}}\sum_{\mu=r-k'}^{k'-1}4^{\mu}P_{2\mu+1}(x)=
\frac{1}{4^{(k'-1)}}\sum_{\mu=r-k'}^{k'-1}4^{\mu}\sum_{\nu=0}^{\mu}C_{2\mu+1,2\nu+1}x^{2\nu+1}=\notag\\
&& \frac{1}{4^{(k'-1)}} \Big\{      
\sum_{\nu=0}^{r-k'-1}  \sum_{\mu=r-k'}^{k'-1} \!\!4^{\mu}C_{2\mu+1,2\nu+1}x^{2\nu+1}+ 
\sum_{\nu=r-k'}^{k'-1}    \sum_{\mu=\nu}^{k'-1} 4^{\mu}C_{2\mu+1,2\nu+1}x^{2\nu+1}
\Big\}.\notag
\end{eqnarray}
Writing the first two polynomials of (\ref{eqPtP_01}) in the canonical basis by means of (\ref{PntildePnu}) and (\ref{P2n}), we obtain
\begin{eqnarray}
&&\sum_{\nu=0}^{2k'}\lambda^{tX}_{2k',\nu}x^{\nu}  =  \sum_{\nu=0}^{2k'}C_{2k',2\nu}x^{2\nu}-\notag\\
&&\frac{\mu_{r}}{4^{(k'-1)}}\Big\{      
\sum_{\nu=0}^{r-k'-1}  \sum_{\mu=r-k'}^{k'-1} \!\!4^{\mu}C_{2\mu+1,2\nu+1}x^{2\nu+1}+
\sum_{\nu=r-k'}^{k'-1}   \sum_{\mu=\nu}^{k'-1} 4^{\mu}C_{2\mu+1,2\nu+1}x^{2\nu+1}
\Big\}.\notag
\end{eqnarray}
By identifying the CC in both sides of this equation, we achieve to identities (\ref{CCTraCheCanonical_condinic11}) and (\ref{CCTraCheCanonical_condinic13}).

Now, we apply the same technique for deducing (\ref{CCTraCheCanonical_eq1})-(\ref{CCTraCheCanonical_eq3}).
Taking $n\leftarrow2n$ in (\ref{CR_Recursion_n2r1}), we obtain
\begin{equation}\label{eqPtP_1}
P^{t}_{2(n+r)+1}(\mu_r;r)(x)=P_{2(n+r)+1}(x)-\mu_r\sum_{i=0}^{r}\frac{1}{4^{i}}P_{2(n+r-i)}(x)\ .
\end{equation}
In the preceding sum, we do the change of variable $\mu=n+r-i$, $i=0(1)r$ and we use (\ref{P2n}); after that we apply Lemma \ref{lemmaCCCRCan}, thus we get
\begin{eqnarray}
&&\frac{1}{4^{(n+r)}}\sum_{\mu=n}^{n+r}4^{\mu}P_{2\mu}(x)=\frac{1}{4^{(n+r)}}\sum_{\mu=n}^{n+r}4^{\mu}\sum_{\nu=0}^{\mu}C_{2\mu,2\nu}x^{2\nu}=\notag \\
&&\frac{1}{4^{(n+r)}}\Big\{\sum_{\nu=0}^{n-1}\Big(\sum_{\mu=n}^{n+r}4^{\mu}C_{2\mu,2\nu}\Big)x^{2\nu} + 
\sum_{\nu=n}^{n+r}\Big(\sum_{\mu=\nu}^{n+r}4^{\mu}C_{2\mu,2\nu}\Big)x^{2\nu} \Big\}
\ .\notag
\end{eqnarray}
Writing the first two polynomials of (\ref{eqPtP_1}) in the canonical basis by means of (\ref{PntildeXn}) and (\ref{P2n1}), we obtain
\begin{eqnarray}
&& \sum_{\nu=0}^{2(n+r)+1}\lambda^{tX}_{2(n+r)+1,\nu}x^\nu  =  \sum_{\nu=0}^{n+r}C_{2(n+r)+1,2\nu+1}x^{2\nu+1}-\notag\\
&&\frac{\mu_r}{4^{(n+r)}}\Big\{\sum_{\nu=0}^{n-1}\Big(\sum_{\mu=n}^{n+r}4^{\mu}C_{2\mu,2\nu}\Big)x^{2\nu}+ \sum_{\nu=n}^{n+r}\Big(\sum_{\mu=\nu}^{n+r}4^{\mu}C_{2\mu,2\nu}\Big)x^{2\nu} \Big\} \ .\notag
\end{eqnarray}
By identifying the CC in both sides, we achieve to identities (\ref{CCTraCheCanonical_eq1})-(\ref{CCTraCheCanonical_eq3}).
In order to demonstrate the identities (\ref{CCTraCheCanonical_eq4})-(\ref{CCTraCheCanonical_eq6}), we must take $n\leftarrow2n+1$ in (\ref{CR_Recursion_n2r1}) and apply the same technique as before. 
We obtain
\begin{equation}\label{eqPtP_2}
P^{t}_{2(n+r+1)}(\mu_r;r)(x)=P_{2(n+r+1)}(x)-\mu_r\sum_{i=0}^{r}\frac{1}{4^{i}}P_{2(n+r-i)+1}(x)\ .
\end{equation}
Doing $\mu=n+r-i$, $i=0(1)r$ and using (\ref{P2n1}), the preceding sum can be written as 
\begin{eqnarray}
&&\frac{1}{4^{(n+r)}}\sum_{\mu=n}^{n+r}4^{\mu}P_{2\mu+1}(x)=\frac{1}{4^{(n+r)}}\sum_{\mu=n}^{n+r}4^{\mu}\sum_{\nu=0}^{\mu}C_{2\mu+1,2\nu+1}x^{2\nu+1}\ .\notag 
\end{eqnarray}
We write the first two polynomials of (\ref{eqPtP_2}) in the canonical basis by means of (\ref{PntildeXn}) and (\ref{P2n}), then we apply Lemma \ref{lemmaCCCRCan}; thus we get
\begin{eqnarray}
&&\sum_{\nu=0}^{2(n+r+1)}\lambda^{tX}_{2(n+r+1),\nu}x^\nu = \sum_{\nu=0}^{n+r+1}C_{2(n+r+1),2\nu}x^{2\nu}-\notag\\
&&\frac{\mu_r}{4^{(n+r)}}\Big\{\sum_{\nu=0}^{n-1}\Big(\sum_{\mu=n}^{n+r}4^{\mu}C_{2\mu+1,2\nu+1}\Big)x^{2\nu+1}+ \sum_{\nu=n}^{n+r}\Big(\sum_{\mu=\nu}^{n+r}4^{\mu}C_{2\mu+1,2\nu+1}\Big)x^{2\nu+1} \Big\}\ . \notag
\end{eqnarray}
By identifying the CC in both sides of this equation, we achieve to identities (\ref{CCTraCheCanonical_eq4})-(\ref{CCTraCheCanonical_eq6}).
\end{proof}
This theorem allows us to conclude that $\lambda^{tX}_{n,m}$ coincide with $C_{n,m}$, when $n$ and $m$ have same parity. Otherwise, if $n$ and $m$ have opposite parity, then $\lambda^{tX}_{n,m}$ depend on the parameter $\mu_r$ of perturbation. Replacing CC of Chebyshev $C_{2n,2\nu}$ and $C_{2n+1,2\nu+1}$ by their expressions given by (\ref{C2n}) and (\ref{C2n1}) and doing some simple combinatorial simplifications \cite{Riordan_1, Riordan1, Riordan2}, we derive next result that furnish explicit formulas for $\lambda^{tX}_{n,m}$ in terms of binomial coefficients.

\begin{theorem}\label{CCTraCheCanonical_binomial}
CC $\lambda_{n,m}^{tX}:=\lambda_{n,m}(P^t\leftarrow X)$ for the $r$-perturbed {\it by translation} case $(r\geq 0)$ in terms of the canonical basis.

If $k=2k'$, then for $k'=0(1)r'$, being $r=2r'$ or $r=2r'+1$, it holds
\begin{eqnarray}
&&\lambda^{tX}_{2k',2\nu}=\frac{(-1)^{k'-\nu}}{4^{(k'-\nu)}}\binom{k'+\nu}{k'-\nu}\ ,\ \nu=0(1)k'\ ,\label{CCTraCheCanonical_condinic101_b}\\ 
&& \lambda^{tX}_{2k',2\nu+1}=0\ ,\  \nu=0(1)k'-1\ .\label{CCTraCheCanonical_condinic102_b}
\end{eqnarray}

If $k=2k'+1$, then for $k'=0(1)r'$, being $r=2r'$ or $r=2r'+1$, it holds
\begin{eqnarray}
&&\lambda^{tX}_{2k'+1,2\nu+1}=\frac{(-1)^{k'-\nu}}{4^{(k'-\nu)}}\binom{k'+\nu+1}{k'-\nu}\ ,\ \nu=0(1)k'\ ,\label{CCTraCheCanonical_condinic103_b}\\
&& \lambda^{tX}_{2k'+1,2\nu}=0\ ,\ \nu=0(1)k'\ . \label{CCTraCheCanonical_condinic104_b}
\end{eqnarray}

If  $k=2k'$, then for $k'=r'+1(1)r$, being $r=2r'$ or $r=2r'+1$, it holds
\begin{eqnarray}
&& \lambda^{tX}_{2k',2\nu}  = \frac{(-1)^{k'-\nu}}{4^{(k'-\nu)}}\binom{k'+\nu}{k'-\nu}\ ,\ \nu=0(1)k'\ ,\label{CCTraCheCanonical_condinic11_a}\\ 
&&  \lambda^{tX}_{2k',2\nu+1} = -\frac{\mu_{r}}{4^{(k'-\nu-1)}}\sum_{\mu=r-k'}^{k'-1} 
(-1)^{\mu-\nu}\binom{\mu+\nu+1}{\mu-\nu}\  ,
\ \nu=0(1)r-k'-1\ ,\label{CCTraCheCanonical_condinic12_a}\\
&& \lambda^{tX}_{2k',2\nu+1} = -\frac{\mu_{r}}{4^{(k'-\nu-1)}}\sum_{\mu=\nu}^{k'-1} 
(-1)^{\mu-\nu}\binom{\mu+\nu+1}{\mu-\nu}\ ,\ \nu=r-k'(1)k'-1\ .\label{CCTraCheCanonical_condinic13_a}
\end{eqnarray}

If $k=2k'+1$, then for $k'=r'(1)r-1$, if $r=2r'$; or for $k'=r'+1(1)r-1$, if $r=2r'+1$, it holds
\begin{eqnarray}
&&  \lambda^{tX}_{2k'+1,2\nu+1}  = \frac{(-1)^{k'-\nu}}{4^{(k'-\nu)}}\binom{k'+\nu+1}{k'-\nu}\ ,\ \nu=0(1)k'\ ,\label{CCTraCheCanonical_condinic21_a}\\ 
&& \lambda^{tX}_{2k'+1,2\nu} = -\frac{\mu_{r}}{4^{(k'-\nu)}}\sum_{\mu=r-k'}^{k'} 
(-1)^{\mu-\nu}\binom{\mu+\nu}{\mu-\nu} \  , \ \nu=0(1)r-k'-1\ ,\label{CCTraCheCanonical_condinic22_a}\\
&&  \lambda^{tX}_{2k'+1,2\nu} = -\frac{\mu_{r}}{4^{(k'-\nu)}}\sum_{\mu=\nu}^{k'} 
(-1)^{\mu-\nu}\binom{\mu+\nu}{\mu-\nu}\ ,\ \nu=r-k'(1)k'\ .\label{CCTraCheCanonical_condinic23_a}
\end{eqnarray}

For $n\geq 0$, it holds
\begin{eqnarray}
&&\lambda^{tX}_{2(n+r)+1,2\nu+1}  = \frac{(-1)^{n+r-\nu}}{4^{(n+r-\nu)}}\binom{n+r+\nu+1}{n+r-\nu}\ ,\ \nu=0(1)n+r\ ;\label{CCTraCheCanonical_n1_a}\\ 
&&\lambda^{tX}_{2(n+r)+1,2\nu}  =  -\frac{\mu_r}{4^{(n+r-\nu)}}\sum_{\mu=n}^{n+r}
(-1)^{\mu-\nu}\binom{\mu+\nu}{\mu-\nu}\ ,\ \nu=0(1)n-1\ ;\label{CCTraCheCanonical_n2_a}\\
&&\lambda^{tX}_{2(n+r)+1,2\nu}  =  
-\frac{\mu_r}{4^{(n+r-\nu)}}
\sum_{\mu=\nu}^{n+r}(-1)^{\mu-\nu}\binom{\mu+\nu}{\mu-\nu}\ ,\ \nu=n(1)n+r\ .\label{CCTraCheCanonical_n3_a} \\ \notag \\
&&\lambda^{tX}_{2(n+r+1),2\nu}  = \frac{(-1)^{n+r+1-\nu}}{4^{(n+r+1-\nu)}}\binom{n+r+1+\nu}{n+r+1-\nu}\ ,\ \nu=0(1)n+r+1\ ;\label{CCTraCheCanonical_n4_a}\\ 
&&\lambda^{tX}_{2(n+r+1),2\nu+1}  =  -\frac{\mu_r}{4^{(n+r-\nu)}}\sum_{\mu=n}^{n+r}(-1)^{\mu-\nu}\binom{\mu+\nu+1}{\mu-\nu}\ ,\ \nu=0(1)n-1\ ;\label{CCTraCheCanonical_n5_a}\\
&&\lambda^{tX}_{2(n+r+1),2\nu+1}  =  
-\frac{\mu_r}{4^{(n+r-\nu)}}
\sum_{\mu=\nu}^{n+r}(-1)^{\mu-\nu}\binom{\mu+\nu+1}{\mu-\nu}\ ,\ \nu=n(1)n+r\ .\label{CCTraCheCanonical_n6_a} 
\end{eqnarray}
\end{theorem}
\begin{proof}
Identities  (\ref{CCTraCheCanonical_condinic101_b})-(\ref{CCTraCheCanonical_condinic104_b}) follow from (\ref{condiniciais_tra1}), (\ref{C2n}) and (\ref{C2n1}). 
Identities (\ref{CCTraCheCanonical_condinic11_a})-(\ref{CCTraCheCanonical_n6_a}) are derived  from
(\ref{CCTraCheCanonical_condinic11})-(\ref{CCTraCheCanonical_eq6}) replacing $C_{2n,2\nu}$ and $C_{2n+1,2\nu+1}$ by their expressions given by (\ref{C2n}) and (\ref{C2n1}) and doing some simple simplifications.
\end{proof}

From the last two theorems, we easily obtain the following corollary that gives CC for perturbations of first orders in terms $C_{n,m}$ and in terms of binomial coefficients. In simplifications, we use systematically the combinatorial identity \cite[p.11]{Riordan1}
$$\binom{n}{m+1}=\binom{n+1}{m+1}-\binom{n}{m}.$$
\begin{corollary} CC $\lambda_{n,m}^{tX}:=\lambda_{n,m}(P^t\leftarrow X)$ for perturbed  by translation. Recall that $\lambda_{n,n}^{tX}=1$, $\forall n\geq 0$.
\noindent For order $0$ (co-recursive case) and $n\geq 0$
\begin{eqnarray}
&& \lambda_{0,0}^{tX}=1\ .\notag\\ 
&& \lambda_{2n+1,2\nu+1}^{tX}=C_{2n+1,2\nu+1}\  ,\ 
\lambda_{2n+1,2\nu}^{tX}=-\mu_0C_{2n,2\nu}\ ,\ \nu=0(1)n\ .\notag\\
&& \lambda_{2n+2,2\nu}^{tX}=C_{2n+2,2\nu}=\frac{(-1)^{n+1-\nu}}{4^{(n-\nu+1)}} \binom{n+\nu+1}{n-\nu+1}\ ,\notag\\
&& \lambda_{2n+2,2\nu+1}^{tX}=-\mu_0C_{2n+1,2\nu+1}\ ,\ \nu=0(1)n\ .\notag
\end{eqnarray}
For order $1$ and $n\geq 0$
\begin{eqnarray}
&& \lambda_{k,\nu}^{tX}=C_{k,\nu}\ ,\ \nu=0(1)k\ ,\ k=0,1\ .\notag \\ 
&& \lambda_{2,2\nu}^{tX}=C_{2,2\nu}\ ,\ \nu=0,1 \Leftrightarrow \lambda_{2,0}^{tX}= -\frac{1}{4}\quad ;\quad  \lambda_{2,1}^{tX}=-\mu_1 \ .\notag\\ 
&& \lambda_{2n+3,2\nu+1}^{tX}=C_{2n+3,2\nu+1}=\frac{(-1)^{n-\nu+1}}{4^{(n-\nu+1)}} \binom{n+\nu+2}{n-\nu+1}\ ,\ \nu=0(1)n+1\ ;\notag\\
&& \lambda_{2n+3,2\nu}^{tX}=-\frac{\mu_1}{4}\big(C_{2n,2\nu}+4C_{2n+2,2\nu}\big)=
\frac{\mu_1(-1)^{n-\nu}}{4^{(n-\nu+1)}} \binom{n+\nu}{n-\nu+1}  \ ,\ \nu=0(1)n\ ; \notag\\
&& \lambda_{2n+3,2n+2}^{tX}=-\mu_1\ .\notag\\
&& \lambda_{2n+4,2\nu}^{tX}=C_{2n+4,2\nu}=\frac{(-1)^{n-\nu}}{4^{(n-\nu+2)}} \binom{n+\nu+2}{n-\nu+2}\ ,\ \nu=0(1)n+2\ ;\notag\\  
&& \lambda_{2n+4,2\nu+1}^{tX}=-\frac{\mu_1}{4}\big(C_{2n+1,2\nu+1}+4C_{2n+3,2\nu+1}\big)\notag\\
&& \ \ \ \ \ \ \ \ \ \ \ \ \ \,  =\frac{\mu_1(-1)^{n-\nu}}{4^{(n-\nu+1)}} \binom{n+\nu+1}{n-\nu+1} \
,\ \nu=0(1)n\quad ; \quad \lambda_{2n+4,2n+3}^{tX}=-\mu_1\ .\notag
\end{eqnarray}
For order $2$  and $n\geq 0$
\begin{eqnarray}
&& \lambda_{k,\nu}^{tX}=C_{k,\nu}\ ,\ \nu=0(1)k\ ,\ k=0(1)2\ .\notag\\ 
&& \lambda_{3,2\nu+1}^{tX}=C_{3,2\nu+1},\nu=0,1 \Leftrightarrow 
\lambda_{3,1}^{tX}=-\frac{1}{2}\ ; \ 
\lambda_{3,0}^{tX}=-\mu_2C_{2,0}=\frac{\mu_2}{4},\ \lambda_{3,2}^{tX}=-\mu_2\ ;\notag\\
&&\lambda_{4,2\nu}^{tX}=C_{4,2\nu}\ ,\ \nu=0(1)2 \Leftrightarrow 
\lambda_{4,0}^{tX}=\frac{1}{4^2} \  , \   \lambda_{4,2}^{tX}= -\frac{3}{4} \ ; \notag\\
&&\lambda_{4,1}^{tX}=-\frac{\mu_2}{4}(1+4C_{3,1})=\frac{\mu_2}{4}
\ ,\ \lambda_{4,3}^{tX}=-\mu_2\ .\notag\\
&& \lambda_{2n+5,2\nu+1}^{tX}=C_{2n+5,2\nu+1}=\frac{(-1)^{n-\nu}}{4^{(n-\nu+2)}}\binom{n+\nu+3}{n-\nu+2}\ ,\ \nu=0(1)n+2\ ; \notag
\end{eqnarray}
\begin{eqnarray}
&& \lambda_{2n+5,2\nu}^{tX}=-\frac{\mu_2}{4^{2}}\big(C_{2n,2\nu}+4C_{2n+2,2\nu}
+4^2C_{2n+4,2\nu}\big)\notag\\
&&\ \ \ \ \ \ \ \ \ \ \, =-\frac{\mu_2(-1)^{n-\nu}}{4^{(n-\nu+2)}}\Big\{ \binom{n+\nu}{n-\nu}+\binom{n+\nu+1}{n-\nu+2} \Big\} \ ,
\ \nu=0(1)n\ ;\notag\\
&& \lambda_{2n+5,2n+2}^{tX}=-\frac{\mu_2}{4}\big(1+4C_{2n+4,2n+2}\big)=\frac{\mu_2}{2}(n+1)\ ,
\ \lambda_{2n+5,2n+4}^{tX}=-\mu_2\ .\notag\\
&& \lambda_{2n+6,2\nu}^{tX}=C_{2n+6,2\nu}=\frac{(-1)^{n-\nu+3}}{4^{(n-\nu+3)}}\binom{n+\nu+3}{n-\nu+3}
\ ,\ \nu=0(1)n+3\ ;\notag\\  
&& \lambda_{2n+6,2\nu+1}^{tX}=-\frac{\mu_2}{4^{2}}\big(C_{2n+1,2\nu+1}+4C_{2n+3,2\nu+1}+4^2C_{2n+5,2\nu+1}\big)\ , \notag\\
&&\ \ \ \ \ \ \ \ \ \ \ \ \ \, =-\frac{\mu_2(-1)^{n-\nu}}{4^{(n-\nu+2)}}\Big\{ \binom{n+\nu+1}{n-\nu}+\binom{n+\nu+2}{n-\nu+2} \Big\} \ ,
\ \nu=0(1)n\ ;\notag\\
&& \lambda_{2n+6,2n+3}^{tX}=-\frac{\mu_2}{4}\big(1+4C_{2n+5,2n+3}\big)=
\frac{\mu_2}{4}(2n+3)\ ,
\ \lambda_{2n+6,2n+5}^{tX}=-\mu_2\ .\notag
\end{eqnarray}
\end{corollary}

\subsection{Dilatation case}

Let us present analogous of preceding results for the dilatation case.

\begin{theorem}\label{CCDilCheCanonical}
CC $\lambda_{n,m}^{dX}:=\lambda_{n,m}(P^d\leftarrow X)$ for the $r$th-perturbed {\it by dilatation} $(r\geq 1)$ in terms of the canonical basis. 
For $k=0(1)r$, it holds
\begin{eqnarray}
&& \lambda^{dX}_{k,\nu}=C_{k,\nu}\ ,\ \nu=0(1)k\  , \ k=0(1)r\ . \label{condiniciais_dil1} 
\end{eqnarray}

If  $k=2k'$, then for $k'=r'+1(1)r-1$, being $r=2r'$ or $r=2r'+1$, it holds
\begin{eqnarray}
&& \lambda^{dX}_{2k',2\nu+1}  = 0\ ,\ \nu=0(1)k'-1\ ,\label{CCDilCheCanonical_condinic11}\\ 
&&  \lambda^{dX}_{2k',2\nu} = C_{2k',2\nu}+\frac{1-\lambda_{r}}{4^{k'}}\sum_{\mu=r-k'}^{k'-1} 4^{\mu}C_{2\mu,2\nu}\  ,\ \nu=0(1)r-k'\ ,\label{CCDilCheCanonical_condinic12}\\
&& \lambda^{dX}_{2k',2\nu} = C_{2k',2\nu}+\frac{1-\lambda_{r}}{4^{k'}}\sum_{\mu=\nu}^{k'-1} 4^{\mu}C_{2\mu,2\nu}\ ,\ \nu=r-k'+1(1)k'-1\ .\label{CCDilCheCanonical_condinic13}
\end{eqnarray}

If $k=2k'+1$, then for $k'=r'(1)r-1$, if $r=2r'$; or for $k'=r'+1(1)r-1$, if $r=2r'+1$, it holds
\begin{eqnarray}
&&  \lambda^{dX}_{2k'+1,2\nu}  = 0\ ,\ \nu=0(1)k'\ ,\label{CCDilCheCanonical_condinic21}\\ 
&& \lambda^{dX}_{2k'+1,2\nu+1} = C_{2k'+1,2\nu+1}+\frac{1-\lambda_{r}}{4^{k'}}\!\!\!\! \sum_{\mu=r-k'-1}^{k'-1} \!\!\!\! 4^{\mu}C_{2\mu+1,2\nu+1},  \nu=0(1)r-k'-1,\label{CCDilCheCanonical_condinic22}\\
&&  \lambda^{dX}_{2k'+1,2\nu+1} = C_{2k'+1,2\nu+1}+\frac{1-\lambda_{r}}{4^{k'}}\!\! \sum_{\mu=\nu}^{k'-1}\! 4^{\mu}C_{2\mu+1,2\nu+1}, \nu=r-k'(1)k'-1. \label{CCDilCheCanonical_condinic23}
\end{eqnarray}

For $n\geq 0$, it holds
\begin{eqnarray}
&&\lambda^{dX}_{2(n+r),2\nu+1}  = 0\ ,\ \nu=0(1)n+r-1\ ,\label{CCDilCheCanonical_eq1}\\ 
&&\lambda^{dX}_{2(n+r),2\nu}  =C_{2(n+r),2\nu} +\frac{1-\lambda_r}{4^{(n+r)}}\sum_{\mu=n}^{n+r-1}4^{\mu}C_{2\mu,2\nu}\ ,\ \nu=0(1)n\ ,\label{CCDilCheCanonical_eq2}\\
&&\lambda^{dX}_{2(n+r),2\nu}  =  
C_{2(n+r),2\nu} +\frac{1-\lambda_r}{4^{(n+r)}}
\sum_{\mu=\nu}^{n+r-1}4^{\mu}C_{2\mu,2\nu}\ ,\ \nu=n+1(1)n+r-1\ .\label{CCDilCheCanonical_eq3} \\ \notag \\
&&\lambda^{dX}_{2(n+r)+1,2\nu}  = 0\ ,\ \nu=0(1)n+r\  ,\label{CCDilCheCanonical_eq4}\\ 
&&\lambda^{dX}_{2(n+r)+1,2\nu+1}\!  =\!C_{2(n+r)+1,2\nu+1}+\frac{1-\lambda_r}{4^{(n+r)}}\!\!\sum_{\mu=n}^{n+r-1}\!\!4^{\mu}C_{2\mu+1,2\nu+1},\nu=0(1)n,\label{CCDilCheCanonical_eq5}\\
&&\lambda^{dX}_{2(n+r)+1,2\nu+1} \! \!=\! C_{2(n+r)+1,2\nu+1}\!\!+ \!
\frac{1-\lambda_r}{4^{(n+r)}}\!\!\!
\sum_{\mu=\nu}^{n+r-1}\! 4^{\mu}C_{2\mu+1,2\nu+1},\nu=n+1(1)n+r-1.\label{CCDilCheCanonical_eq6} 
\end{eqnarray}
\end{theorem}
\begin{proof}
This demonstration is similar to the proof of Proposition \ref{CCTraCheCanonical}, but in this case all polynomials involved are symmetric.
From (\ref{CR_Dilatation_inic_cond_r_Dil}),  
(\ref{P2n}) and (\ref{P2n1}) immediately follow the initial conditions (\ref{condiniciais_dil1}).
For deducing the remainder initial conditions (\ref{CCDilCheCanonical_condinic11})-(\ref{CCDilCheCanonical_condinic23}), like before, we must consider two cases corresponding to
$k=2k'$ ($r=2r'$ and $r=2r'+1$) and to $k=2k'+1$ ($r=2r'$ and $r=2r'+1$).
We are going to demonstrate the first case, the other is similar. 
Taking $k=2k'$ in (\ref{CR_Dilatation_k}), we get
\begin{eqnarray}
&& P^{d}_{2k'}(\lambda_{r};r)=P_{2k'}+\frac{1-\lambda_{r}}{4}\sum_{i=1}^{2k'-r}\frac{1}{4^{(i-1)}}P_{2(k'-i)}, k=r+1(1)2r-1 .\label{eqPdP_01} 
\end{eqnarray}
If $r=2r'$, then $2k'=2r'+2(2)2r-2\Leftrightarrow k'=r'+1(1)r-1$. If $r=2r'+1$, then $2k'=2r'+2(2)2r-2\Leftrightarrow k'=r'+1(1)r-1$.
In (\ref{eqPdP_01}), we do the change of variable $\mu=k'-i$, $i=1(1)2k'-r$ and we use (\ref{P2n}); after that we apply Lemma \ref{lemmaCCCRCan}, thus the preceding sum becomes
\begin{eqnarray}
&& \frac{1}{4^{(k'-1)}}\sum_{\mu=r-k'}^{k'-1}4^{\mu}P_{2\mu}(x)=
\frac{1}{4^{(k'-1)}}\sum_{\mu=r-k'}^{k'-1}4^{\mu}\sum_{\nu=0}^{\mu}C_{2\mu,2\nu}x^{2\nu}=\notag\\
&& \frac{1}{4^{(k'-1)}} \Big\{      
\sum_{\nu=0}^{r-k'-1}  \sum_{\mu=r-k'}^{k'-1} \!\!4^{\mu}C_{2\mu,2\nu}x^{2\nu}+ 
\sum_{\nu=r-k'}^{k'-1}    \sum_{\mu=\nu}^{k'-1} 4^{\mu}C_{2\mu,2\nu}x^{2\nu}
\Big\}.\notag
\end{eqnarray}
Writing the first two polynomials of (\ref{eqPdP_01}) in the canonical basis by means of (\ref{PntildeXn}) and (\ref{P2n}),  taking into account the symmetry of $P^{d}_{2k'}$, we obtain
\begin{eqnarray}
&&\sum_{\nu=0}^{k'}\lambda^{dX}_{2k',2\nu}x^{2\nu}  =  \sum_{\nu=0}^{k'}C_{2k',2\nu}x^{2\nu}+\notag\\
&&\frac{1-\lambda_{r}}{4^{k'}}\Big\{      
\sum_{\nu=0}^{r-k'-1}  \Big(\sum_{\mu=r-k'}^{k'-1} 4^{\mu}C_{2\mu,2\nu}\Big)x^{2\nu}+
\sum_{\nu=r-k'}^{k'-1}  \Big( \sum_{\mu=\nu}^{k'-1} 4^{\mu}C_{2\mu,2\nu}\Big)x^{2\nu}
\Big\}.\notag
\end{eqnarray}
By identifying the CC in both sides of this equation, we achieve to identities (\ref{CCDilCheCanonical_condinic11})-(\ref{CCDilCheCanonical_condinic13}).

Now, we apply the same technique for deducing (\ref{CCDilCheCanonical_eq1})-(\ref{CCDilCheCanonical_eq3}).
Taking $n\leftarrow2n$ in (\ref{CR_Dilatation_n2r}), we obtain
\begin{equation}\label{eqPdP_12}
P^{d}_{2(n+r)}(\lambda_r;r)(x)=P_{2(n+r)}(x)+\frac{1-\lambda_r}{4}\sum_{i=1}^{r}\frac{1}{4^{(i-1)}}P_{2(n+r-i)}(x)\ .
\end{equation}
In the preceding sum, we do the change of variable $\mu=n+r-i$, $i=1(1)r$ and we use (\ref{P2n}); after that we apply Lemma \ref{lemmaCCCRCan}, thus we get
\begin{eqnarray}
&&\frac{1}{4^{(n+r-1)}}\sum_{\mu=n}^{n+r-1}4^{\mu}P_{2\mu}(x)=\frac{1}{4^{(n+r-1)}}\sum_{\mu=n}^{n+r-1}4^{\mu}\sum_{\nu=0}^{\mu}C_{2\mu,2\nu}x^{2\nu}=\notag \\
&&\frac{1}{4^{(n+r-1)}}\Big\{\sum_{\nu=0}^{n-1}\Big(\sum_{\mu=n}^{n+r-1}4^{\mu}C_{2\mu,2\nu}\Big)x^{2\nu} + 
\sum_{\nu=n}^{n+r-1}\Big(\sum_{\mu=\nu}^{n+r-1}4^{\mu}C_{2\mu,2\nu}\Big)x^{2\nu} \Big\}
\ .\notag
\end{eqnarray}
Writing the first two polynomials of (\ref{eqPdP_12}) in the canonical basis by means of (\ref{PntildeXn}) and 
(\ref{P2n}), taking into account the symmetry of $P^{d}_{2(n+r)}$, we obtain
\begin{eqnarray}
&& \sum_{\nu=0}^{n+r}\lambda^{dX}_{2(n+r),2\nu}x^{2\nu}  =  \sum_{\nu=0}^{n+r}C_{2(n+r),2\nu}x^{2\nu}+\label{eqPdP_1}\\
&&\frac{1-\lambda_r}{4^{(n+r)}}\Big\{\sum_{\nu=0}^{n-1}\Big(\sum_{\mu=n}^{n+r-1}4^{\mu}C_{2\mu,2\nu}\Big)x^{2\nu}+ \sum_{\nu=n}^{n+r-1}\Big(\sum_{\mu=\nu}^{n+r-1}4^{\mu}C_{2\mu,2\nu}\Big)x^{2\nu} \Big\} \ .\notag
\end{eqnarray}
By identifying the CC in both sides, we achieve to identities (\ref{CCDilCheCanonical_eq1})-(\ref{CCDilCheCanonical_eq3}).
In order to demonstrate the identities (\ref{CCDilCheCanonical_eq4})-(\ref{CCDilCheCanonical_eq6}), we must take $n\leftarrow2n+1$ in (\ref{CR_Dilatation_n2r}) and apply the same technique  as before. 
We obtain
\begin{equation}\label{eqPdP_2}
P^{d}_{2(n+r)+1}(\lambda_r;r)(x)=P_{2(n+r)+1}(x)+\frac{1-\lambda_r}{4}\sum_{i=1}^{r}\frac{1}{4^{(i-1)}}P_{2(n+r-i)+1}(x)\ .
\end{equation}
Doing $\mu=n+r-i$, $i=1(1)r$ and using (\ref{P2n1}), the preceding sum can be written as 
\begin{eqnarray}
&&\frac{1}{4^{(n+r-1)}}\sum_{\mu=n}^{n+r-1}4^{\mu}P_{2\mu+1}(x)=\frac{1}{4^{(n+r-1)}}\sum_{\mu=n}^{n+r-1}4^{\mu}\sum_{\nu=0}^{\mu}C_{2\mu+1,2\nu+1}x^{2\nu+1}\ .\notag 
\end{eqnarray}
We write the first two polynomials of (\ref{eqPdP_2}) in the canonical basis by means of (\ref{PntildeXn}) and (\ref{P2n1}), taking into account the symmetry of $P^{d}_{2(n+r)+1}$, then we apply Lemma \ref{lemmaCCCRCan}; thus we get
\begin{eqnarray}
&&\sum_{\nu=0}^{n+r}\lambda^{dX}_{2(n+r)+1,2\nu+1}x^{2\nu+1} = \sum_{\nu=0}^{n+r}C_{2(n+r)+1,2\nu+1}x^{2\nu+1}+\notag\\
&&\frac{1-\lambda_r}{4^{(n+r)}}\Big\{\sum_{\nu=0}^{n-1}\Big(\sum_{\mu=n}^{n+r-1}4^{\mu}C_{2\mu+1,2\nu+1}\Big)x^{2\nu+1}+ \sum_{\nu=n}^{n+r-1}\Big(\sum_{\mu=\nu}^{n+r-1}4^{\mu}C_{2\mu+1,2\nu+1}\Big)x^{2\nu+1} \Big\}\ . \notag
\end{eqnarray}
By identifying the CC in both sides of this equation, we achieve to identities (\ref{CCDilCheCanonical_eq4})-(\ref{CCDilCheCanonical_eq6}).
\end{proof}
Due to symmetry $\lambda^{dX}_{n,m}$ vanish when $n$ and $m$ have different parity, otherwise depend on the parameter $\lambda_r$ of perturbation.
In this theorem, replacing $C_{2n,2\nu}$ and $C_{2n+1,2\nu+1}$ by their expressions given by (\ref{C2n}) and (\ref{C2n1}), we derive next result that gives explicit formulas for $\lambda^{dX}_{n,m}$.

\begin{theorem}\label{CCDilCheCanonical_binomial}
CC $\lambda_{n,m}^{dX}:=\lambda_{n,m}(P^d\leftarrow X)$ for the $r$-perturbed {\it by dilatation} case $(r\geq 1)$ in terms of the canonical basis. 

If $k=2k'$, then for $k'=0(1)r'$, being $r=2r'$ or $r=2r'+1$, it holds
\begin{eqnarray}
&&\lambda^{dX}_{2k',2\nu}=\frac{(-1)^{k'-\nu}}{4^{(k'-\nu)}}\binom{k'+\nu}{k'-\nu}\ ,\ \nu=0(1)k'\ ,\label{CCDilCheCanonical_condinic101_b}\\ 
&& \lambda^{dX}_{2k',2\nu+1}=0\ ,\  \nu=0(1)k'-1\ .\label{CCDilCheCanonical_condinic102_b}
\end{eqnarray}

If $k=2k'+1$, then for $k'=0(1)r'$, being $r=2r'$ or $r=2r'+1$, it holds
\begin{eqnarray}
&&\lambda^{dX}_{2k'+1,2\nu+1}=\frac{(-1)^{k'-\nu}}{4^{(k'-\nu)}}\binom{k'+\nu+1}{k'-\nu}\ ,\ \nu=0(1)n'\ ,\label{CCDilCheCanonical_condinic103_b}\\
&& \lambda^{dX}_{2k'+1,2\nu}=0\ ,\ \nu=0(1)k'\ . \label{CCDilCheCanonical_condinic104_b}
\end{eqnarray}

If  $k=2k'$, then for $k'=r'+1(1)r-1$, being $r=2r'$ or $r=2r'+1$,  it holds
\begin{eqnarray}
&& \lambda^{dX}_{2k',2\nu+1}  = 0\ ,\ \nu=0(1)k'-1\ ,\label{CCDilCheCanonical_condinic11_b}\\ 
&&  \lambda^{dX}_{2k',2\nu} = \frac{1}{4^{(k'-\nu)}}\Big\{(-1)^{k'-\nu}\binom{k'+\nu}{k'-\nu}+
(1-\lambda_{r})\sum_{\mu=r-k'}^{k'-1} (-1)^{\mu-\nu}\binom{\mu+\nu}{\mu-\nu}\Big\}, \label{CCDilCheCanonical_condinic12_b}\\
&& \ \ \ \ \ \ \ \ \ \ \ \ \nu=0(1)r-k'-1\ ,\notag \\
&& \lambda^{dX}_{2k',2\nu} = \frac{1}{4^{(k'-\nu)}}\Big\{(-1)^{k'-\nu}\binom{k'+\nu}{k'-\nu}+(1-\lambda_{r})\sum_{\mu=\nu}^{k'-1} (-1)^{\mu-\nu}\binom{\mu+\nu}{\mu-\nu}\Big\}\ , \label{CCDilCheCanonical_condinic13_b}\\
&& \ \ \ \ \ \ \ \ \ \ \ \ \nu=r-k'(1)k'-1\ .\notag
\end{eqnarray}

If $k=2k'+1$, then for $k'=r'(1)r-1$, if $r=2r'$; or for $k'=r'+1(1)r-1$, if $r=2r'+1$, it holds
\begin{eqnarray}
&&  \lambda^{dX}_{2k'+1,2\nu}  = 0\ ,\ \nu=0(1)k'\ ,\label{CCDilCheCanonical_condinic21_b}\\ 
&& \lambda^{dX}_{2k'+1,2\nu+1} =\frac{1}{4^{(k'-\nu)}} \Big\{(-1)^{k'-\nu}\binom{k'+\nu+1}{k'-\nu}+(1-\lambda_{r})\!\!\!\sum_{\mu=r-k'-1}^{k'-1}\!\! \!(-1)^{\mu-\nu}\binom{\mu+\nu+1}{\mu-\nu}\Big\}\ ,  \notag\\
&& \ \ \ \ \ \ \ \ \ \ \ \ \ \ \ \ \ \ \nu=0(1)r-k'-2\ , \label{CCDilCheCanonical_condinic22_b}\\
&&  \lambda^{dX}_{2k'+1,2\nu+1} = \frac{1}{4^{(k'-\nu)}}\Big\{(-1)^{k'-\nu}\binom{k'+\nu+1}{k'-\nu}+(1-\lambda_{r}) \sum_{\mu=\nu}^{k'-1} (-1)^{\mu-\nu}\binom{\mu+\nu+1}{\mu-\nu}\Big\}\ , \notag \\
&& \nu=r-k'-1(1)k'-1\ .\label{CCDilCheCanonical_condinic23_b}
\end{eqnarray}
For $n\geq 0$, it holds
\begin{eqnarray}
&&\lambda^{dX}_{2(n+r),2\nu+1}  = 0\ ,\ \nu=0(1)n+r-1\ ,\label{CCDilCheCanonical_eq1_b}\\ 
&&\lambda^{dX}_{2(n+r),2\nu}  =\frac{1}{4^{(n+r-\nu)}}\Big\{(-1)^{n+r-\nu}\binom{n+r+\nu}{n+r-\nu} +(1-\lambda_r)\sum_{\mu=n}^{n+r-1}(-1)^{\mu-\nu}\binom{\mu+\nu}{\mu-\nu}\Big\},\notag\\
&& \ \ \ \ \ \ \ \ \ \ \ \ \ \ \ \ \nu=0(1)n-1\ ,\label{CCDilCheCanonical_eq2_b}\\
&&\lambda^{dX}_{2(n+r),2\nu}  =  
\frac{1}{4^{(n+r-\nu)}}\Big\{ (-1)^{n+r-\nu}\binom{n+r+\nu}{n+r-\nu} +(1-\lambda_r)
\sum_{\mu=\nu}^{n+r-1}(-1)^{\mu-\nu}\binom{\mu+\nu}{\mu-\nu}\Big\}, \notag \\
&& \ \ \ \ \ \ \ \ \ \ \ \ \ \ \ \ \nu=n(1)n+r-1\ .\label{CCDilCheCanonical_eq3_b} \notag \\
&&\lambda^{dX}_{2(n+r)+1,2\nu}  = 0\ ,\ \nu=0(1)n+r\  ,\label{CCDilCheCanonical_eq4_b}\\ 
&&\lambda^{dX}_{2(n+r)+1,2\nu+1}  =\frac{1}{4^{(n+r-\nu)}}\Big\{(-1)^{n+r-\nu}\binom{n+r+\nu+1}{n+r-\nu}  \notag\\
&& \ \ \ \ \ \ \ \ \ \ \ \ \ \ \ \ \ \ \ \ \  +(1-\lambda_r)\!\!\sum_{\mu=n}^{n+r-1}(-1)^{\mu-\nu}\binom{\mu+\nu+1}{\mu-\nu}\Big\}\ ,\  \nu=0(1)n-1,\label{CCDilCheCanonical_eq5_b}\\
&&\lambda^{dX}_{2(n+r)+1,2\nu+1} =\frac{1}{4^{(n+r-\nu)}}\Big\{(-1)^{n+r-\nu}\binom{n+r+\nu+1}{n+r-\nu}\notag\\
&&\ \ \ \ \ \ \ \ \ \ \ \ \ \ \ \ + (1-\lambda_r)\!\!
\sum_{\mu=\nu}^{n+r-1}(-1)^{\mu-\nu}\binom{\mu+\nu+1}{\mu-\nu}\Big\}\ , \  \nu=n(1)n+r-1.\label{CCDilCheCanonical_eq6_b} 
\end{eqnarray}
\end{theorem}
\begin{proof}
Identities  (\ref{CCDilCheCanonical_condinic101_b})-(\ref{CCDilCheCanonical_condinic104_b}) follow from (\ref{condiniciais_tra1}), (\ref{C2n}) and (\ref{C2n1}). 
Identities (\ref{CCDilCheCanonical_condinic11_b})-(\ref{CCDilCheCanonical_eq6_b}) are derived  from
(\ref{CCDilCheCanonical_condinic11})-(\ref{CCDilCheCanonical_eq6}) replacing $C_{2n,2\nu}$ and $C_{2n+1,2\nu+1}$ by their expressions given by (\ref{C2n}) and (\ref{C2n1}) and doing some simple simplifications.
\end{proof}


From last two theorems, we easily obtain $\lambda_{n,m}^{dX}$ for perturbations of first orders.

\begin{corollary}
CC $\lambda_{n,m}^{dX}:=\lambda_{n,m}(P^d\leftarrow X)$ for perturbed  by dilatation. Recall that $\lambda_{n,n}^{dX}=1$, $\forall n\geq 0$.
\noindent For order $1$ and $n\geq 0$
\begin{eqnarray}
&& \lambda_{k,\nu}^{dX}=C_{k,\nu}\ ,\ \nu=0(1)k\ ,\ k=0,1\ .\notag\\ 
&& \lambda_{2n+2,2\nu+1}^{dX}=0\  ,\  \nu=0(1)n\ .\notag\\
&&\lambda_{2n+2,2\nu}^{dX}=C_{2n+2,2\nu}+\frac{1-\lambda_1}{4}C_{2n,2\nu}\notag\\
&&\ \ \ \ \ \ \ \ \ \ \ =\frac{(-1)^{n-\nu+1}}{4^{(n-\nu+1)}}\Big\{\binom{n+\nu+1}{n-\nu+1}-(1-\lambda_1)\binom{n+\nu}{n-\nu}\Big\}
\ , \ \nu=0(1)n\ ;\notag\\
&& \lambda_{2n+3,2\nu}^{dX}=0\  ,\  \nu=0(1)n+1\ .\notag\\
&&\lambda_{2n+3,2\nu+1}^{dX}=C_{2n+3,2\nu+1}+\frac{1-\lambda_1}{4}C_{2n+1,2\nu+1}\notag\\
&& \ \ \ \ \ \ \ \ \ \ \ =\frac{(-1)^{n-\nu+1}}{4^{(n-\nu+1)}}\Big\{\binom{n+\nu+2}{n-\nu+1}-(1-\lambda_1)\binom{n+\nu+1}{n-\nu}\Big\}
\ ,\ \nu=0(1)n\ .\notag
\end{eqnarray}
For order $2$ and $n\geq 0$
\begin{eqnarray}
&& \lambda_{k,\nu}^{dX}=C_{k,\nu}\ ,\ \nu=0(1)k\ ,\ k=0(1)2\ .\notag\\ 
&& \lambda_{3,2\nu}^{dX}=0\ ,\ \nu=0,1\ ;\ \lambda_{3,1}^{dX}=C_{3,1}+\frac{1-\lambda_2}{4}=-1+\frac{1-\lambda_2}{4}\ .\notag\\
&& \lambda_{2n+4,2\nu+1}^{dX}=0\  ,\  \nu=0(1)n+1\ .\notag\\
&&\lambda_{2n+4,2\nu}^{dX}=C_{2n+4,2\nu}+\frac{1-\lambda_2}{4^{2}}\Big\{C_{2n,2\nu}+4C_{2n+2,2\nu}\Big\}\notag\\
&& \ \ \ \ \ \ \ \ \ \ \ =\frac{(-1)^{n-\nu}}{4^{(n-\nu+2)}}\Big\{\binom{n+\nu+2}{n-\nu+2}-(1-\lambda_2)\binom{n+\nu}{n-\nu+1}\Big\}
\ ,\ \nu=0(1)n\ ; \notag\\
&& \lambda_{2n+4,2n+2}^{dX}=C_{2n+4,2n+2}+\frac{1-\lambda_2}{4}=
\frac{1}{4}\big(-(2n+3)+(1-\lambda_2)\big)=-\frac{2n+2+\lambda_2}{4}\ . \notag\\
&& \lambda_{2n+5,2\nu}^{dX}=0\  ,\  \nu=0(1)n+2\ .\notag\\
&&\lambda_{2n+5,2\nu+1}^{dX}=C_{2n+5,2\nu+1}+\frac{1-\lambda_2}{4^{2}}\Big\{C_{2n+1,2\nu+1}+4C_{2n+3,2\nu+1}\Big\}\notag\\
&& \ \ \ \ \ \ \ \ \ \ \ \ \ =\frac{(-1)^{n-\nu}}{4^{(n-\nu+2)}}\Big\{\binom{n+\nu+3}{n-\nu+2}-(1-\lambda_2)\binom{n+\nu+1}{n-\nu+1}\Big\}\ , \ \nu=0(1)n\ ;\notag \\
&& \lambda_{2n+5,2n+3}^{dX}=C_{2n+5,2n+3}+\frac{1-\lambda_2}{4}=
-\frac{n+2}{2}+\frac{1-\lambda_2}{4}\ .\notag
\end{eqnarray}
\end{corollary}

\section{Some properties of zeros and intersection points}\label{ZerosIntPoints}

In this section, we are going to deduce some results concerning zeros and intersection points of perturbed Chebyshev polynomials valid for any order $r$ of perturbation. Our goal is to explain the main properties observed in the graphical representations presented as illustration 
at the end of this work. 
Let us note the zeros of $P^{t}_{n}(\mu_r;r)(x)$ and $P^{d}_{n}(\lambda_r;r)(x)$,
ordered by increasing size, by 
$\{{\xi^{t}}_k^{(n)}(\mu_r;r)\}_{k=1(1)n}$ and $\{{\xi^{d}}_k^{(n)}(\lambda_r;r)\}_{k=1(1)n}$, or
or simply by $\{{\xi^{t}}_k^{(n)}\}_{k=1(1)n}$ and $\{{\xi^{d}}_k^{(n)}\}_{k=1(1)n}$.

\subsection[Hadamard--Gershg\"orin location]{Hadamard--Gershg\"orin location of zeros}

Let us consider the symmetric tridiagonal Jacobi matrix associated with the MOPS $\{P_n(x)\}_{n\geq 0}$ given by (\ref{icrecOrto})-(\ref{recOrto}), 
$$J=
\left(
\begin{array}{ccccc}
\beta_0 & \alpha_1 & & &   \\
\alpha_1 & \beta_1 & \alpha_2 & &   \\
  & \alpha_2 & \beta_2 & \alpha_ 3 &  \\
 & & \ddots & \ddots & \ddots    \\
\end{array}
\right),
$$
such that, $\alpha_{n+1}^2=\gamma_{n+1}$, $n\geq 0$.
Denoting by $J_n$, $n\geq1$, the $n\times n$ matrix constituted by the first $n$ rows and $n$ columns of $J$, we have that, $P_n(x)=\det(xI_n-J_n)$, $n\geq 1$, where $I_n$ notes the identity matrix of order $n$. Thus zeros of $P_n$ are eigenvalues of $J_n$, $n\geq 1$ 
\cite[p.30]{Chihara}. By definition the  Gershg\"orin discs of any matrix $A=(a_{ij})_{1\leq i,j\leq n}$ are  
${\mathcal D}^{(n)}_i=\{z\in {\mathbb C}:|z-a_{ii}|\leq \sum_{j=1}^n|a_{ij}|\}$, $1\leq i\leq n$.
The Gerschgorin discs of $J_n$ are 
\begin{eqnarray}
&& {\mathcal D}^{(1)}_1=\{\beta_0\}\ . \notag\\  
&& {\mathcal D}^{(n)}_1=\{z\in {\mathbb C}:|z-\beta_{0}|\leq |\alpha_{1}|\} \ , n\geq 2\ ; \notag\\  
&& {\mathcal D}^{(n)}_i=\{z\in {\mathbb C}:|z-\beta_{i-1}|\leq |\alpha_{i-1}|+|\alpha_{i}|\}\ , \ 2\leq i\leq n-1\ , n\geq 2\ ; \notag\\  
&& {\mathcal D}^{(n)}_n=\{z\in {\mathbb C}:|z-\beta_{n-1}|\leq |\alpha_{n-1}|\}\ , \ n\geq 2\ . \notag
\end{eqnarray}
The location of Hadamard--Gershg\"orin assures that all eigenvalues of $J_n$ (zeros of $P_n(x)$)  are in 
$${\mathcal D}^{(n)}=\cup_{i=1}^{n} {\mathcal D}^{(n)}_{i}\ , \ n\geq 1\ ,$$
and if there are $m$ discs disjoints from the others, then their union contains exactly $m$ eigenvalues \cite{Golub_1996}. 

Gershg\"orin intervals of the Jacobi matrix, ${\mathcal J}$, of Chebyshev polynomials of second kind and their union
are 
\begin{eqnarray}
&& {\mathcal D}^{(1)}_1={\mathcal D}^{(1)}=\{0\}\  . \quad
 {\mathcal D}^{(2)}_1={\mathcal D}^{(2)}_2={\mathcal D}^{(2)}=[ -\frac{1}{2},\frac{1}{2} ]\ . \notag\\
&& {\mathcal D}^{(n)}_1=[ -\frac{1}{2},\frac{1}{2} ]\ ; \
{\mathcal D}^{(n)}_i=[ -1,1 ]\ , \ 2\leq i\leq n-1\ ; \
{\mathcal D}^{(n)}_n=[ -\frac{1}{2},\frac{1}{2} ]\ ,\ n\geq 3\ .\notag\\
&& {\mathcal D}^{(n)}=[-1,1]\ ,\ n\geq 3\ . \notag
\end{eqnarray}

Let us denote the Jacobi matrices associated  with
$\{P^{t}_n(\mu_r;r)(x)\}_{n\geq 0}$ and \newline $\{P^{d}_n(\lambda_r;r)(x)\}_{n\geq 0}$ by 
$J^{t}(\mu_r;r)$ and $J^{d}(\lambda_r;r)$, the corresponding Gershg\"orin discs by
${\mathcal D}^{t(n)}_i(\mu_r;r)$ and ${\mathcal D}^{d(n)}_i(\lambda_r;r)$, $0\leq i \leq n$, and their union by 
$${\mathcal D}^{t(n)}(\mu_r;r)=\cup_{i=1}^{n} {\mathcal D}^{t(n)}_i(\mu_r;r)\quad,\quad{\mathcal D}^{d(n)}(\lambda_r;r)=\cup_{i=1}^{n}{\mathcal D}^{d(n)}_i(\lambda_r;r)\ .$$ 
From (\ref{rcoefTT2}), (\ref{rrc_tra}) and (\ref{rrc_dil}), we have
$$J^{t}(\mu_r;r)=
\left(
\begin{array}{ccccccccc}
0 & \frac{1}{2} & & & & & & &\\
\frac{1}{2} & 0 & \frac{1}{2} & & & & & &\\
 & \ddots & \ddots & \ddots & & & & &\\
  &  & \frac{1}{2}  & \mu_r &  \frac{1}{2} & & & & \\
 & & &  & \ddots & \ddots & \ddots  &  &\\
  & & &  &   & \frac{1}{2}   & 0  &  \frac{1}{2} & \\
  & & &  &   &  & \ddots  &  \ddots &  \ddots \\
\end{array}
\right)\ ,
\begin{array}{c}
\\ \\ \\
 \leftarrow {\rm line} 
 \ r+1\\ \\ \\ \\ \\
\end{array}
$$
$$J^{d}(\lambda_r;r)=
\left(
\begin{array}{ccccccccc}
0 & \frac{1}{2} & & & & & & & \\
\frac{1}{2} & 0 & \frac{1}{2} & & & & & & \\
 & \ddots & \ddots & \ddots & & & & & \\
  &  & \frac{1}{2}  & 0 &  \frac{\sqrt{\lambda_r}}{2} & & & & \\
  &  & &  \frac{\sqrt{\lambda_r}}{2} & 0 & \frac{1}{2}  &  & & \\
 & & &  & \ddots & \ddots & \ddots  & & \\
  & & &  &   & \frac{1}{2}   & 0  &  \frac{1}{2}  & \\
    & & &  &   &    &  \ddots  &  \ddots &  \ddots \\
\end{array}
\right).
\begin{array}{l}
\\ \\
 \leftarrow {\rm line} 
 \ r\\  \leftarrow {\rm line} 
 \ r+1
\\ \\ \\
\end{array}
$$
If $\mu_r, \lambda_r\in{\mathbb R}$ and $\lambda_r>0$, we are in the positive definite case, these matrices are real (and symmetric), then their eigenvalues are real and distinct and Gershg\"orin discs are intervals \cite{Golub_1996}.
Perturbation by translation modify only the row of order $r+1$ of ${\mathcal J}$, then Gershg\"orin dics are
\begin{eqnarray}
&& {\mathcal D}^{t(1)}_{1}(\mu_0;0)=\{\mu_0\}\ ,\
{\mathcal D}^{t(n)}_{1}(\mu_0;0)=[\mu_0-\frac{1}{2}, \mu_0+\frac{1}{2}]\ ,\ n\geq 2 \ ;\notag\\
&&{\mathcal D}^{t(r+1)}_{r+1}(\mu_r;r)=[\mu_r-\frac{1}{2}, \mu_r+\frac{1}{2}]\ ,\ 
{\mathcal D}^{t(n)}_{r+1}(\mu_r;r)=[\mu_r-1, \mu_r+1],\ n\geq r+2,\ r\geq 1;\notag\\
&& {\mathcal D}^{t(n)}_{i}(\mu_r;r)={\mathcal D}^{(n)}_{i}\ , \ i\neq r+1\ ,\ 1\leq i\leq n\ ,\ n\geq 1\ ,\ r\geq 0\ ;  \notag
\end{eqnarray} 
and their union is the following, for $r=0$, 
\begin{eqnarray}
&&  {\mathcal D}^{t(1)}(\mu_0;0)=\{\mu_0\}\ ;\ {\mathcal D}^{t(2)}(\mu_0;0)=[-\frac{1}{2},\frac{1}{2}]\cup[\mu_0-\frac{1}{2},\mu_0+\frac{1}{2}]\ ; \notag\\
&& {\mathcal D}^{t(n)}(\mu_0;0)=[-1,1]\cup[\mu_0-\frac{1}{2},\mu_0+\frac{1}{2}]\ ,\ n\geq 3 \ ;  \notag
\end{eqnarray} 
for $r=1$,
\begin{eqnarray}
&&{\mathcal D}^{t(1)}(\mu_1;1)=\{0\}\ ;\ 
{\mathcal D}^{t(2)}(\mu_1;1)=[-\frac{1}{2},\frac{1}{2}]\cup[\mu_1-\frac{1}{2},\mu_1+\frac{1}{2}]\ ; \notag\\
&&{\mathcal D}^{t(3)}(\mu_1;1)=[-\frac{1}{2},\frac{1}{2}]\cup[\mu_1-1,\mu_1+1]\ ;\notag \\ 
&&{\mathcal D}^{t(n)}(\mu_1;1)=[-1,1]\cup[\mu_1-1,\mu_1+1]\ ,\ n\geq 4\ ; \notag
\end{eqnarray} 
for $r=2$, 
\begin{eqnarray}
&&{\mathcal D}^{t(1)}(\mu_2;2)=\{0\}\ ;\ 
{\mathcal D}^{t(2)}(\mu_2;2)=[-\frac{1}{2},\frac{1}{2}]\ ; \
{\mathcal D}^{t(3)}(\mu_2;2)=[-1,1]\cup[\mu_2-\frac{1}{2},\mu_2+\frac{1}{2}];\notag \\ 
&&{\mathcal D}^{t(n)}(\mu_2;2)=[-1,1]\cup[\mu_2-1,\mu_2+1]\ ,\ n\geq 4\ ; \notag
\end{eqnarray} 
and for $r\geq 3$, 
\begin{eqnarray}
&&{\mathcal D}^{t(1)}(\mu_r;r)=\{0\}\ ;\ 
{\mathcal D}^{t(2)}(\mu_r;r)=[-\frac{1}{2},\frac{1}{2}]\ ; \
{\mathcal D}^{t(k)}(\mu_r;r)=[-1,1]\ ,\ k=3(1)r\ ;\notag\\
&& {\mathcal D}^{t(r+1)}(\mu_r;r)=[-1,1]\cup[\mu_r-\frac{1}{2},\mu_r+\frac{1}{2}];\notag \\ 
&&{\mathcal D}^{t(n)}(\mu_2;2)=[-1,1]\cup[\mu_r-1,\mu_r+1]\ ,\ n\geq r+2\ . \notag
\end{eqnarray} 
Perturbation by dilatation affects  rows
of orders $r$ and $r+1$  of ${\mathcal J}$, then Gershg\"orin discs are
\begin{eqnarray}
&& {\mathcal D}^{d(1)}_{1}(\lambda_r;r)=\{0\}\ ,\ r\geq 1\ ;\
{\mathcal D}^{d(n)}_{1}(\lambda_1;1)=[-\frac{\sqrt{\lambda_1}}{2}, \frac{\sqrt{\lambda_1}}{2}]\ ,\ n\geq 2\ ;\notag\\
&& {\mathcal D}^{d(n)}_{r}(\lambda_{r};r)=[-\frac{1+\sqrt{\lambda_r}}{2},\frac{1+\sqrt{\lambda_r}}{2}]\ ,\  n\geq r+1\ ,\ r\geq 2\ ;  \notag\\
&&{\mathcal D}^{d(r+1)}_{r+1}(\lambda_{r};r)=[-\frac{\sqrt{\lambda_{r}}}{2}, \frac{\sqrt{\lambda_{r}}}{2}]\ ,
 \ r\geq 1\ ;\notag\\
&& {\mathcal D}^{d(n)}_{r+1}(\lambda_{r};r)=[-\frac{1+\sqrt{\lambda_r}}{2},\frac{1+\sqrt{\lambda_r}}{2}]\ ,\  n\geq r+2\ ,\ r\geq 1\ ;  \notag\\
&& {\mathcal D}^{d(n)}_{i}(\lambda_r;r)={\mathcal D}^{(n)}_{i}\ ,\ i\neq r\ ,\ i\neq r+1\ ,\ 1\leq i\leq n\ ,\ n\geq 1\ ,\ r\geq 1\ ;  \notag
\end{eqnarray} 
and their union is the following, for r=1, 
\begin{eqnarray}
&& {\mathcal D}^{d(1)}(\lambda_1;1)=\{0\}\ ;\
{\mathcal D}^{d(2)}(\lambda_1;1)=[-\frac{\sqrt{\lambda_1}}{2}, \frac{\sqrt{\lambda_1}}{2}]\ ;\notag\\
&& {\mathcal D}^{d(3)}(\lambda_{1};1)=[-\frac{1}{2},\frac{1}{2}]\cup[-\frac{1+\sqrt{\lambda_1}}{2},\frac{1+\sqrt{\lambda_1}}{2}]=[-\frac{1+\sqrt{\lambda_1}}{2},\frac{1+\sqrt{\lambda_1}}{2}]\ ;  \notag\\
&&{\mathcal D}^{d(n)}(\lambda_{1};1)=[-1,1]\cup[-\frac{1+\sqrt{\lambda_1}}{2},\frac{1+\sqrt{\lambda_1}}{2}]\ ,\  n\geq 4\ ;  \notag
\end{eqnarray} 
for $r=2$, 
\begin{eqnarray}
&& {\mathcal D}^{d(1)}(\lambda_2;2)=\{0\}\ ;\
{\mathcal D}^{d(2)}(\lambda_2;2)=[-\frac{1}{2},\frac{1}{2}]\ ;\notag\\
&&{\mathcal D}^{d(3)}(\lambda_2;2)={\mathcal D}^{d(4)}(\lambda_2;2)=[-\frac{1}{2},\frac{1}{2}]\cup[-\frac{1+\sqrt{\lambda_2}}{2},\frac{1+\sqrt{\lambda_2}}{2}]=[-\frac{1+\sqrt{\lambda_2}}{2},\frac{1+\sqrt{\lambda_2}}{2}]\ ;  \notag\\
&&{\mathcal D}^{d(n)}(\lambda_{2};2)=[-1,1]\cup[-\frac{1+\sqrt{\lambda_2}}{2},\frac{1+\sqrt{\lambda_2}}{2}]\ ,\  n\geq 5\ ;  \notag
\end{eqnarray} 
and for $r\geq 3$, 
\begin{eqnarray}
&& {\mathcal D}^{d(1)}(\lambda_r;r)=\{0\}\ ;\
{\mathcal D}^{d(2)}(\lambda_r;r)=[-\frac{1}{2},\frac{1}{2}]\ ;\ {\mathcal D}^{d(k)}(\lambda_r;r)=[-1,1],\ k=3(1)r;  \notag\\
&&{\mathcal D}^{d(n)}(\lambda_{r};r)=[-1,1]\cup[-\frac{1+\sqrt{\lambda_r}}{2},\frac{1+\sqrt{\lambda_r}}{2}]\ ,\  n\geq r+1\ .  \notag
\end{eqnarray} 
Depending on the values of parameters of perturbation ($\mu_r$ and $\lambda_r$), we can obtain some more information about the location of zeros, provided by next two propositions.
\begin{proposition} For the $r$th-perturbed {\it by translation} case $(r\geq 0)$.
If $\mu_r\in{\mathbb R}$, $\mu_r\neq 0$, $r\geq 0$, it holds
\begin{enumerate}
\item 
For  $r=0$, it holds:
\begin{enumerate}
\item For $n=1$,  ${\mathcal D}^{t(1)}(\mu_0;0)=\{\mu_0\}$.
\item For $n=2$, it holds the same as \it {2.(a) (see next)}, but with $\mu_0$ instead of $\mu_1$.

\item For $n\geq 3$, it holds the same as \it{3.(a) (see next)}, but with $\mu_0$ instead of $\mu_r$.

\end{enumerate}
\item 
For  $r=1$, it holds:
\begin{enumerate}
\item For $n=2$, it holds:

- If $-1\leq\mu_1<0$, then ${\mathcal D}^{t(2)}(\mu_{1};1)=[\mu_1-\frac{1}{2},\frac{1}{2}]$.

- If $0<\mu_1\leq 1$, then ${\mathcal D}^{t(2)}(\mu_{1};1)=[-\frac{1}{2},\mu_1+\frac{1}{2}]$.

- If $\mu_1> 1$ or $\mu_1<-1$, then $[-\frac{1}{2},\frac{1}{2}]\cap [\mu_1-\frac{1}{2},\mu_1+\frac{1}{2}]=\emptyset$,  there is one zero of $P^{t}_{2}(\mu_1;1)(x)$ in $[\mu_1-\frac{1}{2},\mu_1+\frac{1}{2}]$, the other one is in $[-\frac{1}{2},\frac{1}{2}]$.
\item 

For $n=3$, it holds:

- If $-\frac{3}{2}\leq\mu_1\leq -\frac{1}{2}$, then ${\mathcal D}^{t(3)}(\mu_{1};1)=[\mu_1-1,\frac{1}{2}]$.

- If $-\frac{1}{2}\leq\mu_1\leq\frac{1}{2}$, then ${\mathcal D}^{t(3)}(\mu_{1};1)=[\mu_1-1,\mu_1+1]$.

- If $\frac{1}{2}\leq\mu_r\leq \frac{3}{2}$, then ${\mathcal D}^{t(3)}(\mu_{1};1)=[-\frac{1}{2},\mu_1+1]$.

- If $\mu_1> \frac{3}{2}$ or $\mu_1<-\frac{3}{2}$, then $[-\frac{1}{2},\frac{1}{2}]\cap [\mu_1-1,\mu_1+1]=\emptyset$,  there is a unique zero of $P^{t}_{3}(\mu_1;1)(x)$ in $[\mu_1-1,\mu_1+1]$, the others ones are in $[-\frac{1}{2},\frac{1}{2}]$.

\item For $n\geq 4$, it holds the same as \it{3.(b) (see next)}, but with $\mu_1$ instead of $\mu_r$.

\end{enumerate}
\item 
For  $r\geq 2$, it holds:
\begin{enumerate}
\item 
For $n=r+1$, it holds: 

- If $-\frac{3}{2}\leq\mu_r\leq -\frac{1}{2}$, then ${\mathcal D}^{t(n)}(\mu_{r};r)=[\mu_r-\frac{1}{2},1]$.

- If $-\frac{1}{2}\leq\mu_r\leq \frac{1}{2}$, then ${\mathcal D}^{t(n)}(\mu_{r};r)=[-1,1]$.

- If $\frac{1}{2}\leq\mu_r\leq \frac{3}{2}$, then ${\mathcal D}^{t(n)}(\mu_{r};r)=[-1,\mu_r+\frac{1}{2}]$.

- If $\mu_r> \frac{3}{2}$ or $\mu_r<-\frac{3}{2}$, then $[-1,1]\cap [\mu_r-\frac{1}{2},\mu_r+\frac{1}{2}]=\emptyset$,  there is a unique zero of $P^{t}_{n}(\mu_r;r)(x)$ in $[\mu_r-\frac{1}{2},\mu_r+\frac{1}{2}]$, the others ones are in $[-1,1]$.

\item 
For for $n\geq r+2$, it holds: 

- If $-2\leq\mu_r<0$, then ${\mathcal D}^{t(n)}(\mu_{r};r)=[\mu_r-1,1]$.

- If $0<\mu_r\leq 2$, then ${\mathcal D}^{t(n)}(\mu_{r};r)=[-1,\mu_r+1]$.

- If $\mu_r> 2$ or $\mu_r<-2$, then $[-1,1]\cap [\mu_r-1,\mu_r+1]=\emptyset$, there is a unique zero of $P^{t}_n(\mu_r;r)(x)$ in $[\mu_r-1,\mu_r+1]$, the others ones are in $[-1,1]$.
\end{enumerate}
\end{enumerate}
\end{proposition}
\begin{proposition} For the $r$th-perturbed {\it by dilatation} case $(r\geq 1)$.
If $\lambda_r\in{\mathbb R}$, $\lambda_r>0$, $\lambda_r\neq 1$, $r\geq 1$, it holds
\begin{enumerate}
\item 
For $r=1$, it holds:

If $\lambda_1>0$, then ${\mathcal D}^{d(3)}(\lambda_1;1)=[-\frac{1+\sqrt{\lambda_1}}{2},\frac{1+\sqrt{\lambda_1}}{2}]$.

If $0<\lambda_1<1$, then ${\mathcal D}^{d(n)}(\lambda_1;1)=[-1,1]$, $n\geq 4$.

If $\lambda_1>1$, then ${\mathcal D}^{d(n)}(\lambda_1;1)=[-\frac{1+\sqrt{\lambda_1}}{1},\frac{1+\sqrt{\lambda_1}}{1}]$, $n\geq 4$.
\item 
For $r=2$, it holds:

If $\lambda_2>0$, then ${\mathcal D}^{d(3)}(\lambda_2;2)={\mathcal D}^{d(4)}(\lambda_2;2)=[-\frac{1+\sqrt{\lambda_2}}{2},\frac{1+\sqrt{\lambda_2}}{2}]$.

If $0<\lambda_2<1$, then ${\mathcal D}^{d(n)}(\lambda_2;2)=[-1,1]$, $n\geq 5$.

If $\lambda_2>1$, then ${\mathcal D}^{d(n)}(\lambda_2;2)=[-\frac{1+\sqrt{\lambda_2}}{2},\frac{1+\sqrt{\lambda_2}}{2}]$, $n\geq 5$.

\item 
For $r\geq 3$,  it holds:
  
If $0<\lambda_r< 1$, then ${\mathcal D}^{d(n)}(\lambda_{r};r)=[-1,1]$, $n\geq r+1$. 

If $\lambda_r>1$, then  ${\mathcal D}^{d(n)}(\lambda_{r};r)=[-\frac{1+\sqrt{\lambda_r}}{2},\frac{1+\sqrt{\lambda_r}}{2}]$, $n\geq r+1$.
\end{enumerate}
\end{proposition}
From these two propositions and (\ref{TnVnWn}), we can deduce the following Hadamard--Gershg\"orin location of zeros for the others families of Chebyshev 
\begin{eqnarray} 
T_1(x),\ T_2(x)\ : && {\cal D}^{d(1)}(2;1)=\{0\}\ ,\ {\cal D}^{d(2)}(2;1)=[-\frac{\sqrt{2}}{2},\frac{\sqrt{2}}{2}]
\ ,\notag\\
T_n(x)\ , \ n\geq 3\ : && {\cal D}^{d(n)}(2;1)=[-\frac{1+\sqrt{2}}{2},\frac{1+\sqrt{2}}{2}]
\ , \ n\geq 3\ . \notag\\
V_1(x)\ , \ V_2(x)\ :  &&  {\cal D}^{t(2)}\left(\frac{1}{2};0\right)=\left\{ \frac{1}{2}\right\}\ ,\ {\cal D}^{t(2)}\left(\frac{1}{2};0\right)=[-\frac{1}{2},1]\ ,\notag\\
V_n(x)\ , \ n\geq 3\ :  &&{\cal D}^{t(n)}\left(\frac{1}{2};0\right)=[-1,1]\ ,\ n\geq 3 \ . \notag\\
W_1(x)\ , \ W_2(x)\ :  &&  {\cal D}^{t(2)}\left(-\frac{1}{2};0\right)=\left\{- \frac{1}{2}\right\}\ ,\ {\cal D}^{t(2)}\left(-\frac{1}{2};0\right)=[-1,\frac{1}{2}]\ ,\notag\\
W_n(x)\ , \ n\geq 3\ :  && {\cal D}^{t(n)}\left(-\frac{1}{2};0\right)=[-1,1]\ ,\ n\geq 3 \ . \notag
\end{eqnarray} 
Notice that from (\ref{zerosTnPn})-(\ref{zerosVnWn}), we know that the sets of zeros of $\{T_n\}_{n\geq 0}$, $\{V_n\}_{n\geq 0}$ and $\{W_n\}_{n\geq 0}$ are contained in $[-1,1]$.

\begin{remark}
For degree $n\leq 3$, we could investigate the location of zeros of first polynomials from their explicit expressions that depend on the parameters of perturbation.
\end{remark}

\subsection[Zeros at the origin]{Zeros at the origin}

From the CC in terms of the canonical basis given by Theorems \ref{CCTraCheCanonical_binomial} and \ref{CCDilCheCanonical_binomial} and Vi\`ete's formulas (\ref{CC_Can_Zeros_Pn}), we can derive some information about zeros 
of perturbed Chebyshev polynomials 
at the origin.

\begin{proposition} \label{corollaryCC_Tra_Can_nn1n0}
For the $r$th-perturbed {\it by translation} case $(r\geq 0)$, it holds
\begin{eqnarray} 
&& \lambda_{k,k-1}^{tX}=0=\sum_{i=1}^{k} {\xi^{t}}_i^{(k)}\ ,\ k=0(1)r\ ,\label{CCTraCheCanonical_Cor_n0}\\
&&\lambda_{n,n-1}^{tX}=-\mu_r\Leftrightarrow \sum_{i=1}^{n} {\xi^{t}}_i^{(n)}=\mu_r\ ,\ n\geq r+1\ . \label{CCTraCheCanonical_Cor_n1}\\
&&\lambda^{tX}_{2n,0}=\frac{(-1)^{n}}{2^{2n}}\Leftrightarrow\prod_{k=1}^{2n} {\xi^{t}}_k^{(2n)}=\frac{(-1)^{n}}{2^{2n}} \Rightarrow
P^{t}_{2n}(\mu_r;r)(0)\neq0 \ , \ n\geq 0   \ .\label{CCTraCheCanonical_Cor_n0_1} 
\end{eqnarray}
If $r=2r'+1$, then 
\begin{eqnarray}
&& \lambda^{tX}_{2n+1,0}=0\Leftrightarrow\prod_{k=1}^{2n+1} {\xi^{t}}_k^{(2n+1)}=0\Leftrightarrow
P^{t}_{2n+1}(\mu_r;r)(0)=0\ ,\ n\geq 0\ . \label{CCTraCheCanonical_Cor_n0_2}
\end{eqnarray}
If $r=2r'$, then  
\begin{eqnarray}
&& \lambda^{tX}_{2k'+1,0}=0\Leftrightarrow\prod_{k=1}^{2k'+1} {\xi^{t}}_k^{(2k'+1)}=0
\Leftrightarrow
P^{t}_{2k'+1}(\mu_r;r)(0)=0,\ k'=0(1)r'-1, \label{CCTraCheCanonical_Cor_n0_2_1}\\
&& \lambda^{tX}_{2n+1,0} \! =\!\mu_r\frac{(-1)^{n+1}}{2^{2n}}\!\!\Leftrightarrow\!\!
\prod_{k=1}^{2n+1} {\xi^{t}}_k^{(2n+1)}\!=\!\mu_r\frac{(-1)^{n}}{2^{2n}}\!\!\Rightarrow\!\!
P^{t}_{2n+1}(\mu_r;r)(0)\neq0,n\geq r'. \label{CCTraCheCanonical_Cor_n31}
\end{eqnarray}
\end{proposition}
\begin{proof}
We calculate $\lambda_{n,0}^{tX}$ and $\lambda_{n,n-1}^{tX}$ using mainly Theorem \ref{CCTraCheCanonical_binomial} and we establish the relationship with the sum and the product of zeros $\{\xi_k^{t(n)}\}_{k=1(1)n}$ by means of Vi\`ete formulas (\ref{CC_Can_Zeros_Pn}).
From the product, we conclude about $P^{t}_{n}(\mu_r;r)(0)$. 

From (\ref{condiniciais_tra1}) and (\ref{CCTT2_0n1}), we get (\ref{CCTraCheCanonical_Cor_n0}). From (\ref{condiniciais_tra1}) and (\ref{CCTT2_0n1}), we obtain (\ref{CCTraCheCanonical_Cor_n0_2}), for $n=0(1)r'$; and (\ref{CCTraCheCanonical_Cor_n0_2_1}). These results stem from the fact that first $r+1$ perturbed polynomials coincide with Chebyshev polynomials, as stated by (\ref{CR_Recursion_inic_cond_r_Tra}).

Taking $\nu=k'-1$ in (\ref{CCTraCheCanonical_condinic13_a}), we get $\lambda^{tX}_{2k',2k'-1}  =-\mu_r$.
Doing $\nu=k'$ in (\ref{CCTraCheCanonical_condinic23_a}), we have $\lambda^{tX}_{2k'+1,2k'} = -\mu_r$. 
Taking $\nu=n+r$ in (\ref{CCTraCheCanonical_n3_a}) and in (\ref{CCTraCheCanonical_n6_a}), we obtain
$\lambda^{tX}_{2(n+r)+1,2(n+r)}  =  -\mu_r$ and  $\lambda^{tX}_{2(n+r+1),2(n+r)+1}= -\mu_r$. Thus (\ref{CCTraCheCanonical_Cor_n1}) is proved.

Taking $\nu=0$ in (\ref{CCTraCheCanonical_condinic101_b}), (\ref{CCTraCheCanonical_condinic11_a}) and  
(\ref{CCTraCheCanonical_n4_a}), we obtain (\ref{CCTraCheCanonical_Cor_n0_1}) for 
$n=0(1)r'$, for $n=r'+1(1)r$ and for $n\geq r+1$, respectively, e.g., for $n\geq 0$.

Doing $\nu=0$ in (\ref{CCTraCheCanonical_condinic22_a}), we get
$$\lambda^{tX}_{2k'+1,0} =-\frac{\mu_r}{2^{2k'}} \sum_{\mu=r-k'}^{k'}(-1)^\mu=
-\frac{(-1)^{r-k'}\mu_r}{2^{2k'}} \sum_{\mu=0}^{2k'-r}(-1)^\mu.$$

If $r=2r'$, then $\sum_{\mu=0}^{2(k'-r')}(-1)^\mu=1$ and we obtain (\ref{CCTraCheCanonical_Cor_n31}), for $n=r'(1)r-1$.
If $r=2r'+1$, then $\sum_{\mu=0}^{2(k'-r')+1}(-1)^\mu=0$ and we obtain (\ref{CCTraCheCanonical_Cor_n0_2}), for $n=r'+1(1)r-1$.
At last, taking $\nu=0$ in (\ref{CCTraCheCanonical_n2_a}), we have
$$ \lambda^{tX}_{2(n+r)+1,0}  = -\frac{\mu_r}{2^{2(n+r)}}\sum_{\mu=n}^{n+r}(-1)^{\mu}=  
-\frac{(-1)^n\mu_r}{2^{2(n+r)}}\sum_{\mu=0}^{r}(-1)^{\mu}\ .$$
If $r=2r'$, then $\sum_{\mu=0}^{2r'}(-1)^\mu=1$ and we obtain (\ref{CCTraCheCanonical_Cor_n31}), for $n\geq r$. If $r=2r'+1$, then $\sum_{\mu=0}^{2r'+1}(-1)^\mu=0$ and we obtain (\ref{CCTraCheCanonical_Cor_n0_2}), for $n\geq r$. We remark that in (\ref{CCTraCheCanonical_Cor_n0_2}), $P^{t}_{2k'+1}(\mu_r;r)(x)$ is not symmetric.  
\end{proof}
\begin{proposition} \label{corollaryCC_Dil_Can_nn1n0}
For the $r$th-perturbed {\it by dilatation} case $(r\geq 1)$, it holds 
\begin{eqnarray} 
&&P^{d}_{2n+1}(\lambda_r;r)(0)=0 \Leftrightarrow  \prod_{i=1}^{2n+1} {\xi^{d}}_i^{(2n+1)}=0 \Leftrightarrow
\lambda^{dX}_{2n+1,0}=0
\ ,\ n\geq 0 \ . \label{Dil_P2n2n1_0}
\end{eqnarray}
\begin{eqnarray}
&&\lambda_{n,n-1}^{dX} = 0 = \sum_{i=1}^{n} {\xi^{d}}_i^{(n)}=0\ ,\ n\geq 0\ . \label{CCDilCheCanonical_Cor_n0}
\end{eqnarray}
If $r=2r'$, then 
\begin{eqnarray} 
&& \lambda^{dX}_{2n,0} = \frac{(-1)^{n}}{2^{2n}}=\prod_{i=1}^{2n} {\xi^{d}}_i^{(2n)}\ ,\ n\geq 0\ .\label{CCDilCheCanonical_Cor_n41}
\end{eqnarray}
If $r=2r'+1$, then  
\begin{eqnarray}
&&\lambda^{dX}_{2k',0}=\frac{(-1)^{k'}}{2^{2k'}}= \prod_{i=1}^{2k'} {\xi^{d}}_i^{(2k')}
\ ,\ k'=0(1)r'\ ,\label{CCDilCheCanonical_condinic_Cor_101_b} \\
&&\lambda^{dX}_{2k',0}  = \frac{(-1)^{k'}}{2^{2k'}}\lambda_r =\prod_{i=1}^{2k'} {\xi^{d}}_i^{(2k')}\ ,\ k'=r'+1(1)r-1\ ,\label{CCDilCheCanonical_Cor_n12}\\
&&\lambda^{dX}_{2(n+r),0}  =  -\frac{(-1)^{n}}{2^{2(n+r)}}\lambda_r=\prod_{i=1}^{2n} {\xi^{d}}_i^{(2n)}\ ,\ n\geq 0\ .\label{CCDilCheCanonical_Cor_n42}
\end{eqnarray}
For $r\geq1$,
\begin{eqnarray}
&& P^{d}_{2n}(\lambda_r;r)(0)\neq0\ ,\ n\geq 0\ . \label{CCDilCheCanonical_Cor_n10}
\end{eqnarray}
\end{proposition}
\begin{proof} 
The symmetry of the sequence $\{P^{d}_{n}(\lambda_r;r)(x)\}_{n\geq 0}$ implies that (\ref{Dil_P2n2n1_0}) and (\ref{CCDilCheCanonical_Cor_n0}) are verified. 
We calculate $\lambda_{n,0}^{dX}$ and $\lambda_{n,n-1}^{dX}$ using mainly Theorem \ref{CCDilCheCanonical_binomial}. The rest of the proof is similar with the preceding one.
\end{proof}

\subsection[Location of extremal zeros]{Location of extremal zeros of perturbed Chebyshev polynomials}


We notice that for any monic polynomials, we have
\begin{equation}\label{signspolyinf}  
\lim_{x\rightarrow +\infty}P_n(x)=+\infty\ ,\ \lim_{x\rightarrow-\infty}P_{2n}(x)=+\infty\ ,\ 
 \lim_{x\rightarrow-\infty}P_{2n+1}(x)=-\infty.
\end{equation}
From symmetry follows that $P_{2n+1}(0)=0$ and real zeros are symmetric with respect to the origin. Regular orthogonality ensures that two polynomials of consecutive degrees can not have a common zero  \cite{Chihara}. 
In the positive definite case, all zeros are distinct real numbers and an {\em interlacing property} holds between zeros of $P_n$ and $P_{n+1}$; also there are some {\em monotonicity properties} to consider \cite{Chihara}. 
{\em Semi-classical} character, in particular, {\em classical} character, by means of the structure relation 
$$\Phi(x)P'_{n+1}(x)=\frac{1}{2}\left(C_{n+1}(x)-C_0(x)\right)P_{n+1}(x)+\gamma_{n+1}D_{n+1}P_n(x)\ ,$$
 $D_{n+1}(x)\neq 0$, $n\geq0$ \cite[p.123]{MARO 91}, guarantees that zeros are simple \cite[pp.235-236]{MARO 90}.
Chebyshev forms are classical and they admit integral representations
 with positive weights in $[-1,1]$ \cite{Mason_2003}, thus they are positive definite and then zeros of Chebyshev polynomials satisfy all cited properties in $[-1,1]$. 

Next two propositions provide some information about the location of 
the smallest and the greatest zeros (the extremal zeros) of perturbed Chebyshev polynomials 
with respect to the extremal zeros of Chebyshev polynomials of second kind with the same degree. Results are obtained from the CR of Propositions \ref{CCTraChe} and \ref{CCDilChe} and depend on the signs of $\mu_r$ and $1-\lambda_r$.
Let us note the zeros of $P^{t}_{k}(\mu_r;r)(x)$ and $P^{d}_{k}(\lambda_r;r)(x)$ by $\{{\xi^{t}}_i^{(k)}\}_{i=1(1)k}$ and $\{{\xi^{d}}_i^{(k)}\}_{i=1(1)k}$, ordered by increasing size.
\begin{proposition}\label{Pro_Tra_Zeros}  For the $r$th-perturbed {\it by translation} case $(r\geq 0)$.
If $\mu_r\in \mathbb{R}$, for $k\geq r+1$, it holds
\begin{enumerate}
\item  
$sgn\big[P^{t}_{k}(\mu_r;r)(\xi_{k}^{(k)})\big]=-sgn(\mu_r)$. 
\item  
$sgn\big[P^{t}_{k}(\mu_r;r)(\xi_{1}^{(k)})\big]=(-1)^k sgn(\mu_r)$. 
\item
If $\mu_r>0$, then:
\begin{enumerate}
\item 
 $\exists\ i$, $1\leq i\leq k$: ${\xi^{t}}_i^{(k)}>\xi_{k}^{(k)}$; 
$P^{t}_{k}(\mu_r;r)({\xi^{t}}_{i}^{(k)})=0$.
\item 
The number of zeros of $P^{t}_{k}(\mu_r;r)(x)$ greater than $\xi_{k}^{(k)}$ is odd.
\item
$\forall  x\leq \xi_{1}^{(k)}:  (-1)^kP^{t}_{k}(\mu_r;r)(x)>0.$ 
\item
There are no zeros of  $P^{t}_{k}(\mu_r;r)(x)$ less than $\xi_{1}^{(k)}$.
\end{enumerate}
\item
If $\mu_r<0$, then:
\begin{enumerate}
\item
$\forall  y\geq \xi_{k}^{(k)}:  P^{t}_{k}(\mu_r;r)(y)>0.$ 
\item
There are no zeros of  $P^{t}_{k}(\mu_r;r)(x)$ greater than $\xi_{k}^{(k)}$.
\item
$\exists\ i$, $1\leq i\leq k$: ${\xi^{t}}_i^{(k)}<\xi_{1}^{(k)}$; 
$P^{t}_{k}(\mu_r;r)({\xi^{t}}_i^{(k)})=0$.
\item 
The number of zeros of $P^{t}_{k}(\mu_r;r)(x)$ less than $\xi_{1}^{(k)}$ is odd.
\end{enumerate}
\end{enumerate}
\end{proposition}
\begin{proof}
Let us deal with the CR (\ref{CR_Recursion_k}). Similar reasonings can be applied to the CR 
(\ref{CR_Recursion_n2r1}) leading to exactly the same conclusions. We remark that in (\ref{CR_Recursion_k}) the polynomials under sum have degrees smaller than the degree $k$ of $P_k(x)$ and of $P^{t}_{k}(\mu_r;r)(x)$ and have opposite parity with respect to it; moreover all their zeros belong to $]\xi_{1}^{(k)},\xi_{k}^{(k)}[$. 
Next, we will evaluate (\ref{CR_Recursion_k}) at the first, $\xi_{1}^{(k)}$, and the last, $\xi_{k}^{(k)}$, zeros of $P_k(x)$, and at any $x<\xi_{1}^{(k)}$ and any $y>\xi_{k}^{(k)}$, and
we note the signs of polynomials at those points according with (\ref{signspolyinf}). 
\begin{eqnarray}
P^{t}_{k}(\mu_r;r)(\xi_{k}^{(k)}) & = & \underbrace{P_k(\xi_{k}^{(k)})}_{=\ 0}-\mu_r\sum_{i=0}^{k-r-1}\frac{1}{4^i}\underbrace{P_{k-2i-1}(\xi_{k}^{(k)})}_{>\ 0}\ , \label{zeros_tra_kk}\\
P^{t}_{k}(\mu_r;r)(y) & = & \underbrace{P_k(y)}_{>\ 0}-\mu_r\sum_{i=0}^{k-r-1}\frac{1}{4^i}\underbrace{P_{k-2i-1}(y)}_{>\ 0}\ , \ \forall y > \xi_{k}^{(k)}\ , \label{zeros_tra_y}\\
P^{t}_{2k}(\mu_r;r)(\xi_{1}^{(2k)}) & = & \underbrace{P_{2k}(\xi_{1}^{(2k)})}_{=\ 0}-\mu_r\sum_{i=0}^{2k-r-1}\frac{1}{4^i}\underbrace{P_{2k-2i-1}(\xi_{1}^{(2k)})}_{<\  0}\ , \label{zeros_tra_2k1}\\
P^{t}_{2k+1}(\mu_r;r)(\xi_{1}^{(2k+1)}) & = & \underbrace{P_{2k+1}(\xi_{1}^{(2k+1)})}_{=\ 0}-\mu_r\sum_{i=0}^{2k-r}\frac{1}{4^i}\underbrace{P_{2k-2i}(\xi_{1}^{(2k+1)})}_{>\  0}\ . \label{zeros_tra_2k11}\\
P^{t}_{2k}(\mu_r;r)(x) & = & \underbrace{P_{2k}(x)}_{>\ 0}-\mu_r\sum_{i=0}^{2k-r-1}\frac{1}{4^i}\underbrace{P_{2k-2i-1}(x)}_{<\  0},  \ \forall x < \xi_{1}^{(2k)},\label{zeros_tra_2kx}\\
P^{t}_{2k+1}(\mu_r;r)(x) & = & \underbrace{P_{2k+1}(x)}_{<\ 0}-\mu_r\sum_{i=0}^{2k-r}\frac{1}{4^i}\underbrace{P_{2k-2i}(x)}_{>\  0}, \ \forall x < \xi_{1}^{(2k+1)}.\label{zeros_tra_2k1x}
\end{eqnarray}
From the above considerations, we can easily deduce the following conclusions.
{\it Item 1} follows from (\ref{zeros_tra_kk}), and {\it item 2} follows from (\ref{zeros_tra_2k1})-(\ref{zeros_tra_2k11}).
If $\mu_r>0$, then:  (\ref{zeros_tra_kk}) $\Longrightarrow P^{t}_k(\mu_r;r)(\xi^{(k)}_k)<0 \Longrightarrow$ {\it 3.(a)} $\Longrightarrow$  {\it 3.(b)}, by (\ref{signspolyinf});  (\ref{zeros_tra_2k1})-(\ref{zeros_tra_2k1x}) $\Longrightarrow$ {\it 3.(c)}  $\Longrightarrow$ {\it 3.(d)}.
If $\mu_r<0$, then: 
 (\ref{zeros_tra_kk}) $\Longrightarrow$ $P^{t}_k(\mu_r;r)(\xi^{(k)}_k)>0$ with
 (\ref{zeros_tra_y}) $\Longrightarrow$  {\it 4.(a)} $\Longrightarrow$  {\it 4.(b)}; 
 (\ref{zeros_tra_2k1}) $\Longrightarrow$ $P^{t}_{2k}(\mu_r;r)(\xi^{(2k)}_1)<0$ and 
(\ref{zeros_tra_2k11}) $\Longrightarrow$ $P^{t}_{2k+1}(\mu_r;r)(\xi^{(2k+1)}_1)>0$, then, by 
(\ref{signspolyinf}), we have {\it 4.(c)};  {\it 4.(c)} $\Longrightarrow$ {\it 4.(d)}, again by
(\ref{signspolyinf}). 
\end{proof}
\begin{proposition}\label{Pro_Dil_Zeros} For the $r$th-perturbed {\it by dilatation} case $(r\geq 1)$.
If $\lambda_r\in\mathbb{R}$, for $k\geq r+1$, it holds
\begin{enumerate}
\item
$sgn\big[P^{d}_{k}(\lambda_r;r)(\xi_{k}^{(k)})\big]=sgn(1-\lambda_r)$.
\item
$sgn\big[P^{d}_{k}(\lambda_r;r)(\xi_{1}^{(k)})\big]=(-1)^k sgn(1-\lambda_r)$.
\item
If $\lambda_r<1$, then:
\begin{enumerate}
\item
$\forall y\geq \xi_{k}^{(k)}:\ P^{d}_{k}(\lambda_r;r)(y)>0$.
\item
There are no real zeros of $P^{d}_{k}(\lambda_r;r)(x)$ greater than $\xi_{k}^{(k)}$.
\item
$\forall x\leq \xi_{1}^{(k)}:\ (-1)^kP^{d}_{k}(\lambda_r;r)(x)>0$.
\item
There are no real zeros of $P^{d}_{k}(\lambda_r;r)(x)$ less than  $\xi_{1}^{(k)}$.
\item
All real zeros of $P^{d}_{k}(\lambda_r;r)(x)$ are in $]\xi_{1}^{(k)},\xi_{k}^{(k)}[$.
\end{enumerate}
\item
If $\lambda_r>1$, then:
\begin{enumerate}
\item
$\exists\ i$, $1\leq i\leq k$ : ${\xi^{d}}_i^{(k)}>\xi_{k}^{(k)}$;
$P^{d}_{k}(\lambda_r;r)({\xi^{d}}_{i}^{(k)})=0$ and 
$P^{d}_{k}(\lambda_r;r)(-{\xi^{d}}_{i}^{(k)})=0$, 
$-{\xi^{d}}_i^{(k)}<-\xi_{k}^{(k)}=\xi_{1}^{(k)}$. 

\item The numbers of real zeros of $P^{d}_{k}(\lambda_r;r)(x)$ less than $\xi_{1}^{(k)}$ and greater than $\xi_{k}^{(k)}$ are odd (they are necessarily equal).
\end{enumerate}
\end{enumerate}
\end{proposition}
\begin{proof} 
As perturbed polynomials $P^{d}_{k}(\lambda_r;r)(x)$ are symmetric, their zeros are symmetric with respect to the origin. This fact is important in {\it item 4} and allows to obtain {\it 3.(d)} from {\it 3.(b)}.
We deal with the RC (\ref{CR_Dilatation_k}), the same can be done with (\ref{CR_Dilatation_n2r}). We remark that in (\ref{CR_Dilatation_k})
the polynomials under sum have degrees smaller than the degree $k$ of $P_k(x)$ 
and have the same parity of it; moreover all their zeros belong to $]\xi_{1}^{(k)},\xi_{k}^{(k)}[$. 
We will evaluate (\ref{CR_Dilatation_k}) at the first and the last zeros of $P_k(x)$, and at certain $x$ and $y$, and we follow the same method of the proof of preceding proposition. 
\end{proof}
About similar properties of extremal zeros for perturbed orthogonal polynomials with a different type of perturbation of the second recurrence coefficient $\tilde\gamma_{r}=\gamma_r+\lambda_r$ see \cite{Leopold_1998}.


\subsection[Zeros and interception points]{Zeros and interception points of perturbed Chebyshev polynomials}

We point out some results about zeros and interception points of perturbed polynomials with different parameters of perturbation and same degree; for that we need to use the explicit formulas of zeros of Chebyshev polynomials of second kind (\ref{zerosTnPn}).

\begin{proposition}\label{Pro_Tra_Ipt} For the $r$th-perturbed {\it by translation} case $(r\geq 0)$.

\noindent 
\begin{enumerate}
\item It holds,
\begin{eqnarray}
&& P^{t}_{n}(\mu_r;r)\equiv P_n\  ,\ 0\leq n\leq r\ ,\notag\\
&& P^{t}_{n}(\mu_r;r)(x)=P_n(x)-\mu_rP_r(x)P_{n-r-1}(x)\ ,\ n\geq r+1\ .\label{rrPtnPn}
\end{eqnarray}
\item
For $n\geq r+1$, the polynomials
\begin{equation}\label{int_points_tra}
P^{t}_{n}(\mu_r;r)\ , \  P^{t}_{n}(\mu'_r;r)\ ,\  P_n\ ,\  \mu_r\neq \mu'_r,\ \mu_r\neq 0,\ \mu_r'\neq 0\ ,
\end{equation} 
intersect each other at the zeros of $P_r$ and at the zeros of $P_{n-r-1}$. 
\item 
$\xi$ is a double interception point of (\ref{int_points_tra}) if and only if $\xi$ is a common zero of $P_r$ and $P_{n-r-1}$.
\item
For $n=i(r+1)+r$, $i\geq 1$, $r\geq 1$; all zeros of $P_r$ are double interception points of (\ref{int_points_tra}) - in fact they are double common zeros (see item 8) - and the number of distinct interception points is $n-r-1$. 
\item
If $r$ is even, then: if $n$ is even, then the origin is a simple interception point of (\ref{int_points_tra}); if $n$ is odd, then the origin is not an interception point of (\ref{int_points_tra}).

If $r$ is odd, then: if $n$ is even, then the origin is a simple interception point of (\ref{int_points_tra}); if $n$ is odd, then the origin is a double interception point of (\ref{int_points_tra}).
\item
$\xi$ is a common zero of (\ref{int_points_tra}) if and only if 
$\xi$ is a common zero of $P_n$ and $P_{r}$ or 
$\xi$ is a common zero of $P_n$ and $P_{n-r-1}$.
\item 
$\xi$ is a double common zero of (\ref{int_points_tra}) if and only if $\xi$ is a common zero of 
$P_{r}$, $P_{n}$ and $P_{n-r-1}$.
\item 
For $n=i(r+1)+r$, $i\geq 1$, $r\geq 1$, all zeros of $P_r$ are double common zeros of (\ref{int_points_tra}).
\item
If $n$ and $r$ are odd, then the origin is a double common zero of (\ref{int_points_tra}). If $n$ is odd and $r$ is even, or if $n$ is even, then the origin is not a common zero of (\ref{int_points_tra}).
\end{enumerate}
\end{proposition}
\begin{proof}
\begin{enumerate}
\item
From $P^{t}_{n}(\mu_r;r)(x)=P_n(x)-\mu_rP_r(x)P_{n-r-1}^{(r+1)}(x)$ \cite[p.205]{Marcellan_1990} and (\ref{rcoefTT2}), we obtain immediately the relation (\ref{rrPtnPn}), taking into account that $\{P_n\}_{n\geq0}$ is self-associated, thus $P_{n-r-1}^{(r+1)}\equiv P_{n-r-1}$.
\item
From (\ref{rrPtnPn}), we consider
$A^{t}(x)=P^{t}_{n}(\mu_r;r)(x)-P_n(x)$,  $B^{t}(x)=P^{t}_{n}(\mu'_r;r)(x)-P_n(x)$ and $C^{t}(x)=P^{t}_{n}(\mu_r;r)(x)-P^{t}_{n}(\mu'_r;r)(x)$. $\exists\ \xi$:
$A^{t}(\xi)=0 \wedge B^{t}(\xi)=0\wedge C^{t}(\xi)=0 \Leftrightarrow P_r(\xi)P_{n-r-1}(\xi)=0 \Leftrightarrow 
P_r(\xi)=0\vee P_{n-r-1}(\xi)=0$.
\item $\xi$ is a double interception point of (\ref{int_points_tra}) if and only if 
$A^{t}(\xi)=0 \wedge (A^{t})'(\xi)=0$, $B^{t}(\xi)=0 \wedge (B^{t})'(\xi)=0$ and $C^{t}(\xi)=0 \wedge (C^{t})'(\xi)=0$. These conditions are equivalent to 
\begin{eqnarray}
P_r(\xi)=0 & \vee & P_{n-r-1}(\xi)=0 \label{traeq1}\\
P'_r(\xi)P_{n-r-1}(\xi) & = & -P_r(\xi)P'_{n-r-1}(\xi) \label{traeq2}.
\end{eqnarray}
If $\xi$ is not be a zero of $P_r$, $P_r(\xi)\neq0$, then by (\ref{traeq1}), we must have $P_{n-r-1}(\xi)=0$;  in that case we know that $P'_{n-r-1}(\xi)\neq0$, because zeros of Chebyshev polynomials are simple. Then,  (\ref{traeq2}) would be equivalent to a contradiction. The same reasoning can be applied if if we suppose that $\xi$ is not be a zero of $P_{n-r-1}$. In conclusion, we must have $P_r(\xi)=0  \wedge P_{n-r-1}(\xi)=0$. 
\item All zeros of $P_{r}(x)$ are also zeros of $P_{n-r-1}(x)$ if and only if $\ \exists i\geq 1$: $n=i(r+1)+r$, $r\geq 1$. In fact, from (\ref{zerosTnPn}), the zeros of $P_r(x)$ and of $P_{n-r-1}(x)$ are respectively 
$\xi_{k}^{(r)}=\cos(k\frac{\pi}{r+1})$, $k=1(1)r$ and $\xi_{k}^{(n-r-1)}=\cos(k\frac{\pi}{n-r})$, $k=1(1)n-r$, for $n\geq r+1$. Thus, it must $\exists i\geq 1$: $\frac{\pi}{r+1}=i \frac{\pi}{n-r}\Leftrightarrow n-r=i(r+1)$. The number of distinct interception points is equal to the number of distinct zeros of $P_r(x)P_{n-r-1}(x)$, e.g., $\deg P_r(x)+\deg P_{n-r-1}(x) -r=n-r-1$. Remark that, it must be $n-r-1\geq r$. 
\item
Writing (\ref{rrPtnPn}) for $r\rightarrow 2r'$, $n\rightarrow 2n'$, we obtain 
$P^{t}_{2n'}(\mu_{2r'};2r')(x)-P_{2n'}(x)=-\mu_{2r'}P_{2r'}(x)P_{2n'-2r'-1}(x)$. Now, from symmetry, we know that $P_{2r'}(x)$ is even, so the origin is not a zero, and $P_{2n'-2r'-1}(x)$ is odd, so the origin is a zero and it is simple, then we get the conclusion. The other cases are similar.
\item Follows from (\ref{rrPtnPn}).
\item 
A double common zero is a special double interception point. Then the result follows from items 3 and 6.
\item We would like to apply item 7 concerning all zeros of $P_{r}$. All zeros of $P_{r}$ are also zeros of $P_{n}$ if and only if $\ \exists j\geq 2$: $n=j(r+1)-1$, $r\geq 1$ if and only if all zeros of $P_{r}$ are also zeros of $P_{n-r-1}$, because $n=j(r+1)-1=(j-1)(r+1)+r$ and we can take $i=j-1$ in such a way item 4 is also verified. 
\item
From symmetry, the origin is a zero of $P_m(x)$ if and only if $m$ is odd. The parity of $P_{n-r-1}$ is determined by the parity of $n$ and $r$. If $n$ and $r$ are odd, then $n-r-1$ is odd, and the origin is a common zero of $P_r$, $P_n$ and $P_{n-r-1}$;
then we get the first part of the result from item 7.  If $n$ is odd and $r$ is even, then $n-r-1$ is even and, from item 6, we get the conclusion. If $n$ is even, we apply again item 6.
\end{enumerate} 
\end{proof}
\begin{proposition}\label{Pro_Dil_Ipt} For the $r$th-perturbed {\it by dilatation} case $(r\geq 1)$.

\noindent 
\begin{enumerate}
\item It holds,
\begin{eqnarray}
&& P^{d}_{n}(\lambda_r;r)\equiv P_n\ ,\ 0\leq n\leq r\ ,\notag\\
&& P^{d}_{n}(\lambda_r;r)(x)=P_n(x)+\frac{1-\lambda_r}{4}P_{r-1}(x)P_{n-r-1}(x)\ ,\ n\geq r+1\ . \label{rrPdnPn}
\end{eqnarray}
\item
For $n\geq r+1$, the polynomials
\begin{equation}\label{int_points_dil}
P^{d}_{n}(\lambda_r;r)\ , \ P^{d}_{n}(\lambda'_r;r)\ ,\  P_n\ ,\ \lambda_r\neq \lambda'_r,\ 
\lambda_r\neq 1,\ \lambda'_r\neq 1\ ,
\end{equation} 
intersect each other at the zeros of $P_{r-1}$ and at the zeros of $P_{n-r-1}$. 
\item 
$\xi$ is a double interception point of (\ref{int_points_dil}) if and only if $\xi$ is a common zero of $P_{r-1}$ and $P_{n-r-1}$.
\item For $n=r(i+1)$, $i\geq 1$, $r\geq 1$; all zeros of $P_{r-1}$ are double interception points of (\ref{int_points_dil}) and the number of distinct interception points is $n-r-1$.
\item
If $r$ is even, then: if $n$ is even, then the origin is a double interception point of (\ref{int_points_dil}); if $n$ is odd, then the origin is a simple interception point of (\ref{int_points_dil}).

If $r$ is odd, then: if $n$ is even, then the origin is is not an interception point of (\ref{int_points_dil}); if $n$ is odd, then the origin is a simple interception point of (\ref{int_points_dil}).
\item
$\xi$ is a common zero of (\ref{int_points_dil}) if and only if 
$\xi$ is a common zero of $P_n$ and $P_{r-1}$ or 
$\xi$ is a common zero of $P_n$ and $P_{n-r-1}$.
\item 
For $n=jr-1$,  $j\geq 2$, $r\geq 1$, all zeros of $P_{r-1}$ are common zeros of (\ref{int_points_dil}).
\item 
$\xi$ is a double common zero of (\ref{int_points_dil}) if and only if $\xi$ is a common zero of 
$P_{r-1}$, $P_{n}$ and $P_{n-r-1}$.
\item 
All zeros of $P_{r-1}$ can not be simultaneously double common zeros of (\ref{int_points_dil}).
\item
If $n$ is odd, then the origin is a simple common zero of (\ref{int_points_dil}). If $n$ is even, then the origin is not a common zero of (\ref{int_points_dil}).
\end{enumerate}
\end{proposition}
\begin{proof}
It is analogous to the proof of the preceding proposition. Nevertheless, we point out some details in some items.
\begin{enumerate}
\item
From 
$P^{d}_{n}(\lambda_r;r)(x)=P_n(x)+(1-\lambda_r)\gamma_rP_{r-1}(x)P_{n-r-1}^{(r+1)}(x)$
\cite[p.206]{Marcellan_1990} and (\ref{rcoefTT2}), we get the recurrence relation (\ref{rrPdnPn}),  because $P_{n-r-1}^{(r+1)}\equiv P_{n-r-1}$.
\item [4.] All zeros of $P_{r-1}$ are also zeros of $P_{n-r-1}$ if and only if $\exists$ $i\geq1$: $\frac{\pi}{r}=i \frac{\pi}{n-r}\Leftrightarrow n-r=ir \Leftrightarrow n=r(i+1)$. Now, apply item 3.
The number of distinct interception points is $\deg P_{r-1}(x)+\deg P_{n-r-1}(x)-(r-1)=n-r-1$. Remark that it must be $n-r-1\geq r-1$.
\item [9.] 
In fact, from item 8, all zeros of $P_{r-1}$ are common zeros of $P_{n}$ and $P_{n-r-1}$ if and only if $\exists i, j \in {\mathbb N}$, $i\geq1$, $j\geq2$: $n=r(i+1)=jr-1\Leftrightarrow
j=i+1+\frac{1}{r}$, from items 4 and 7. But in that case $j\notin {\mathbb N}$.
\item [10.]
Follows from the symmetry of all polynomials involved. If $n$ is odd, then polynomials (\ref{int_points_dil}) are all odd, and we get the first part of the result.  We are going to show that the origin can not be a double common zero of (\ref{int_points_dil}). If $r$ is even, then $r-1$ is odd,  but $n-r-1$ is even, so the condition of item 8 fails. If $r$ is odd, then $r-1$ is even, then that condition fails again. If $n$ is even, the origin is not a zero of $P_n$.
\end{enumerate} 
\end{proof}
\begin{remark}
Comparing 
(\ref{CR_Recursion_n2r1}) with (\ref{rrPtnPn}), or (\ref{CR_Dilatation_k}) 
with (\ref{rrPdnPn}), we immediately obtain the following linearization formula for the Chebyshev polynomials of second kind
\begin{eqnarray}
P_{r}(x)P_{n+r}(x)& = & \sum_{i=0}^{r}\frac{1}{4^{i}}P_{n+2(r-i)}(x)\ ,\ n\geq 0\ ,\ r\geq 0\ . 
\notag 
\end{eqnarray}
Here $r$ is just a degree and lost the meaning of order of perturbation. This 
formula is a particular case of a linearization formula for Gegenbauer polynomials given in \cite{Dougall_1919,Tcheutia_2014}.
\end{remark}

\section{Graphical representations}\label{Gra_Repre}

In this section, we present some graphical representations with comments in order to illustrate  results for zeros and interception points given in the preceding section.
Figures \ref{Fig_TT2_Per_Tra_0_n56}, \ref{Fig_TT2_Per_Tra_1_n56}, \ref{Fig_PerDil_1} and 
\ref{Fig_PerDil_2} concern Propositions \ref{Pro_Tra_Zeros} and \ref{Pro_Dil_Zeros}.
Figures \ref{Figure_TT2_Tra_r_5}, \ref{Figure_TT2_Dil_r_6}, \ref{Figure_TT2_Tra_r_5_n_17} and \ref{Figure_TT2_Dil_r_6_n_18} refer to Propositions \ref{Pro_Tra_Ipt} and \ref{Pro_Dil_Ipt}. In all figures, we can observe properties satisfied by perturbed polynomials at the origin given in Propositions \ref{corollaryCC_Tra_Can_nn1n0} and \ref{corollaryCC_Dil_Can_nn1n0}.


\vspace{0.75cm}

\noindent {\bf \large{Acknowledgements}}

\vspace{0.25cm}

I am very grateful to Pascal Maroni for several discussions during the development of this work.

\vspace{0.25cm}

The author was partially supported by CMUP (UID/MAT/00144/2013), which is funded by FCT (Portugal) with national (MEC) and European structural funds through the programs FEDER, under the partnership agreement PT2020.






\begin{figure}[htbp]
\begin{center}
\includegraphics[keepaspectratio]{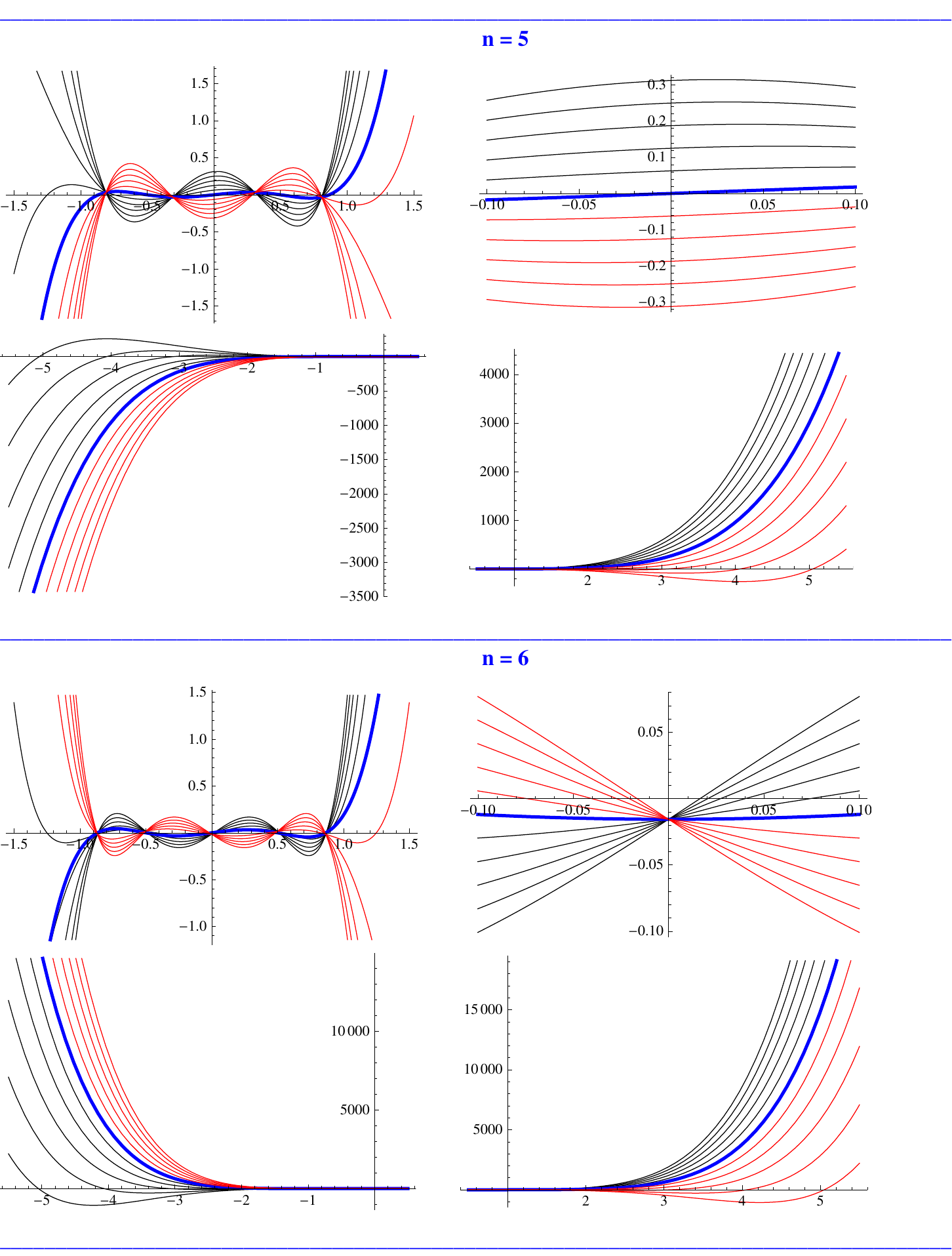}
\caption{Some perturbed of {\bf order 0 by translation} with negative parameters $\mu_0=-5(1)-1$ (in black), 
positive parameters $\mu_0=1(1)5$ (in red) of Chebyshev polynomials of second kind (in bleu) of degrees $n=$ $5$, $6$.}\label{Fig_TT2_Per_Tra_0_n56}
\end{center}
\end{figure}

\begin{figure}[htbp]
\begin{center}
\includegraphics[keepaspectratio]{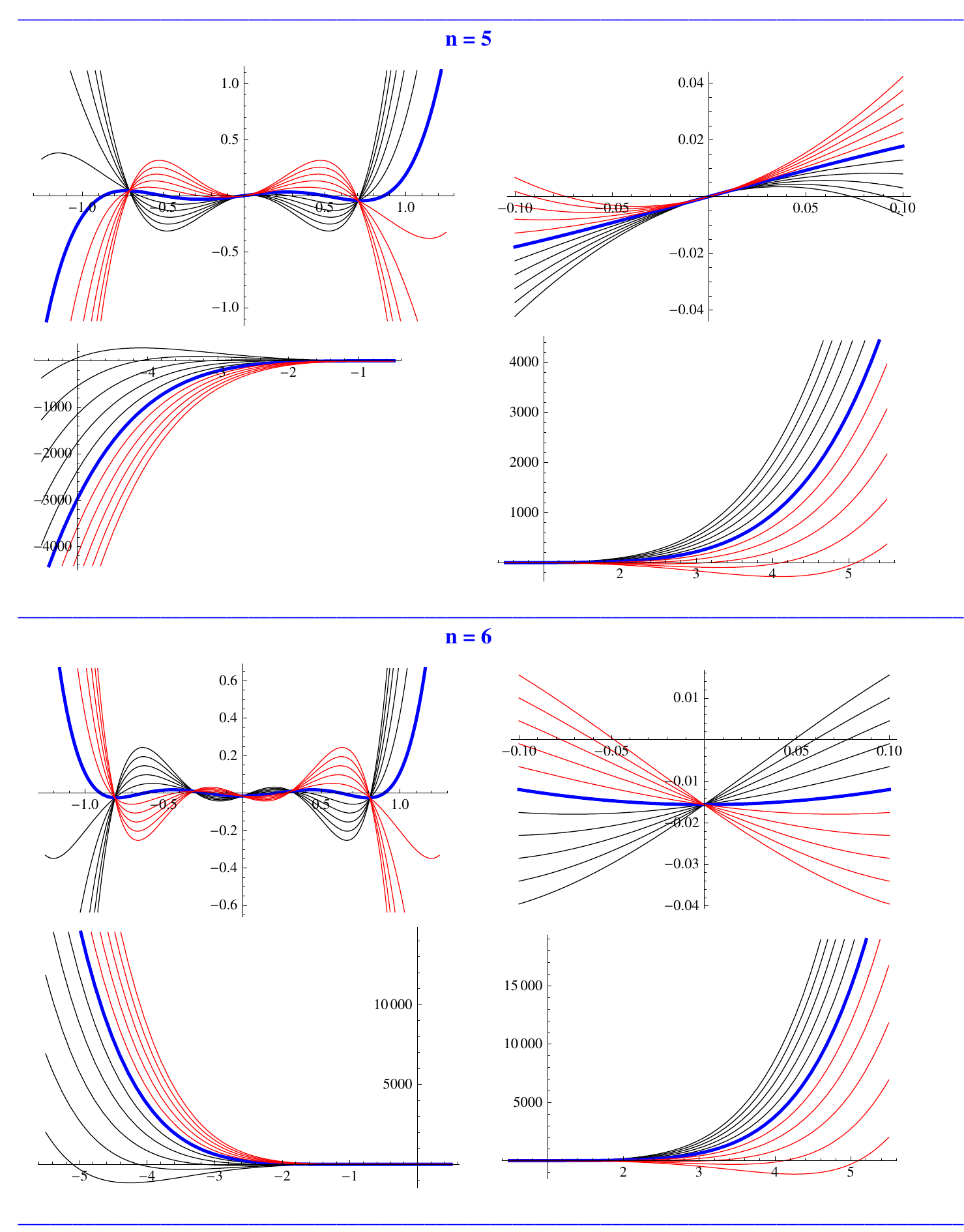}
\caption{Some perturbed of {\bf order 1 by translation} with negative parameters $\mu_0=-5(1)-1$ (in black), 
positive parameters $\mu_0=1(1)5$ (in red) of Chebyshev polynomials of second kind (in bleu) of degrees $n=$ $5$, $6$.}\label{Fig_TT2_Per_Tra_1_n56}
\end{center}
\end{figure}



\begin{figure}[htbp]
\caption{Some perturbed of {\bf order 1 by dilatation} with parameters $\lambda_1=-5(1)-1\ {\bf < 1}$ (in black) and $\lambda_1=3(1)7 \ {\bf > 1}$ (in red) of Chebyshev polynomials of second kind (in bleu) of degrees $n=$ $5$, $7$, $6$ and $8$.}
\label{Fig_PerDil_1}
\begin{center}
\includegraphics[keepaspectratio]{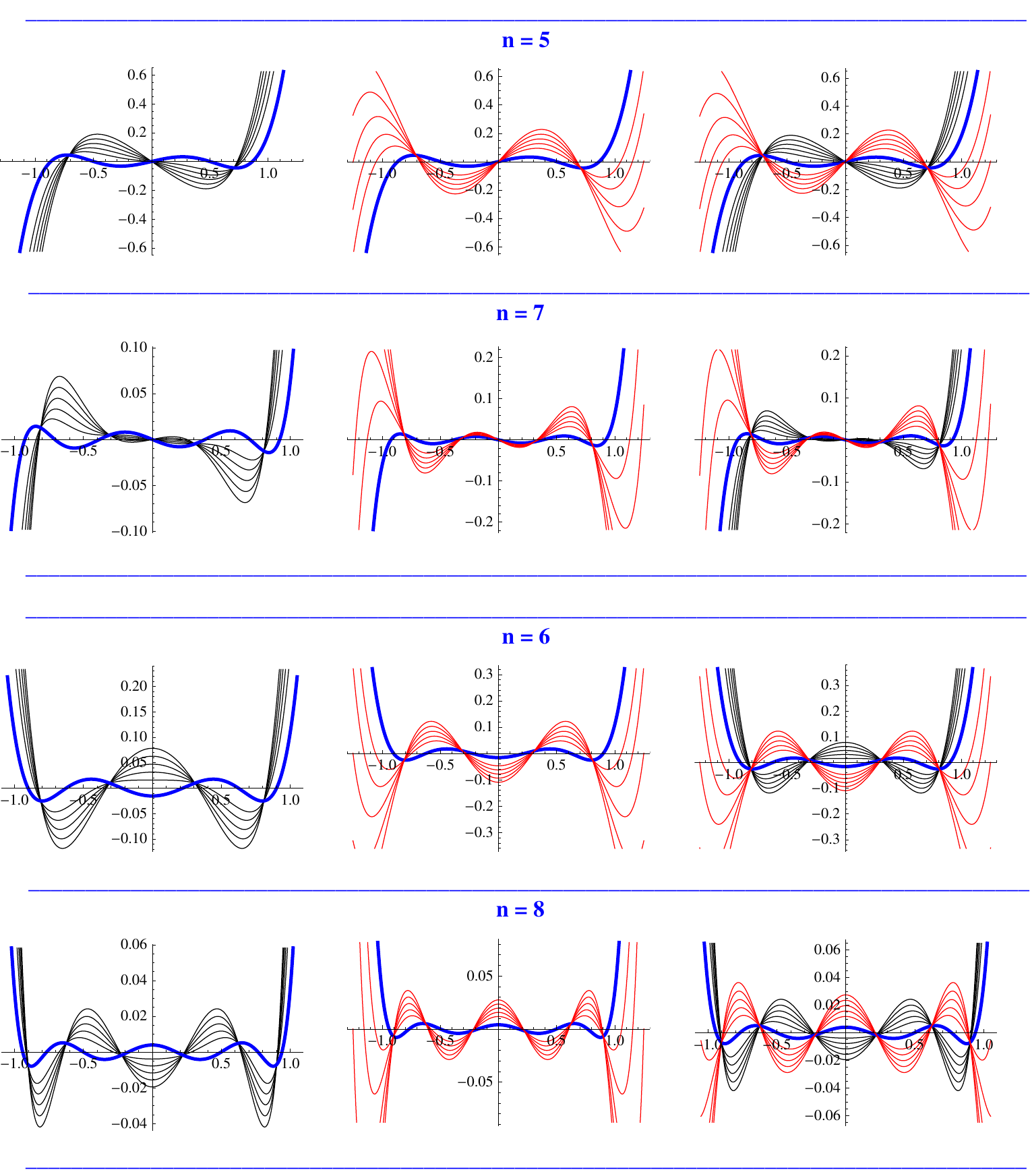}
\end{center}
\end{figure}
\begin{figure}[htbp]
\caption{Some perturbed of {\bf order 2 by dilatation} with parameters $\lambda_2=-5(1)-1\ {\bf < 1}$ (in black) and $\lambda_2=3(1)7 \  {\bf > 1}$  (in red) of Chebyshev polynomials of second kind (in bleu) of degrees $n=$ $5$, $7$, $6$, $8$.}
\begin{center}
\includegraphics[keepaspectratio]{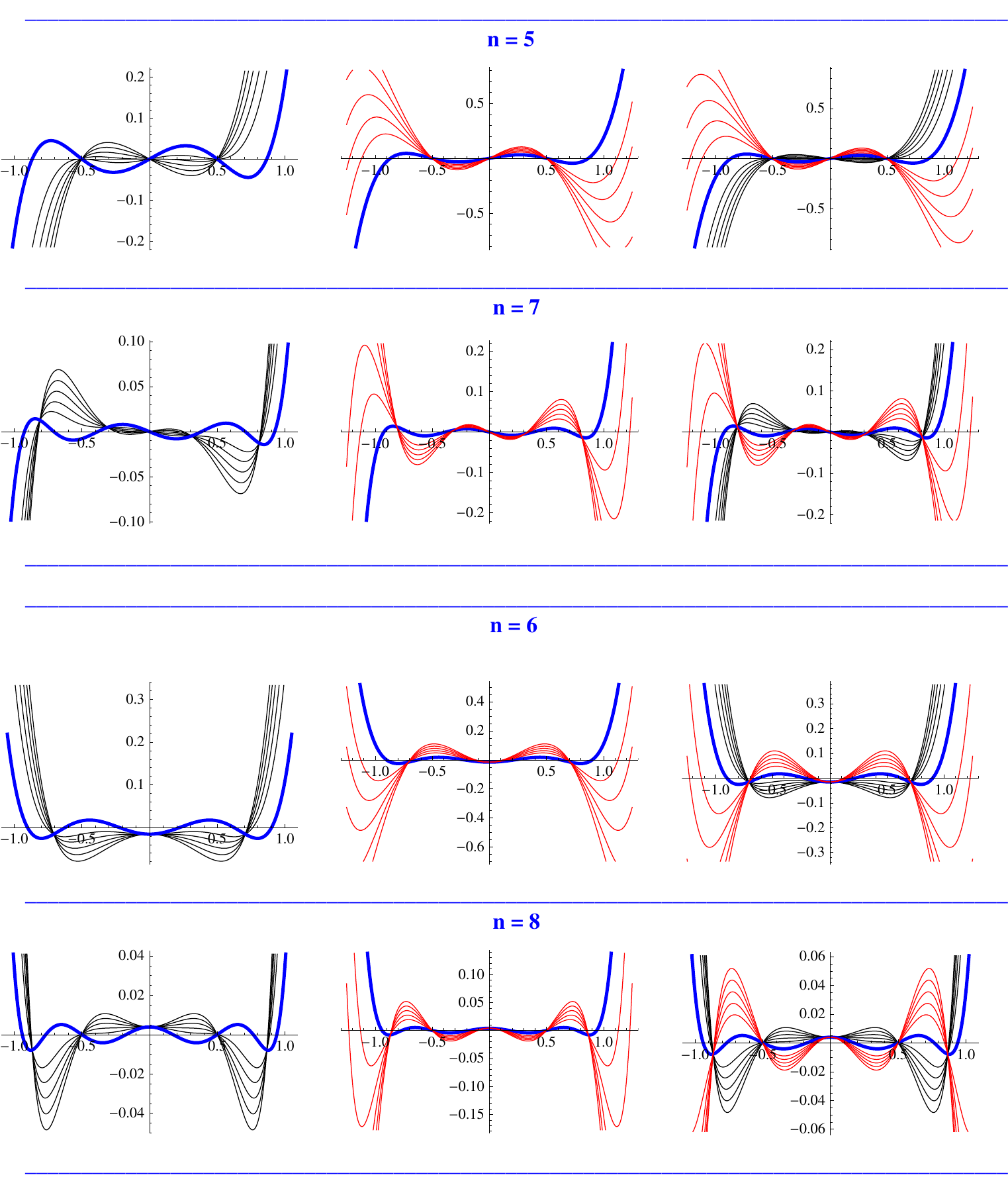}
\label{Fig_PerDil_2}
\end{center}
\end{figure}


\begin{figure}[htbp]
\caption{Some perturbed of {\bf order 5 by translation} with parameters $\mu'_5=-5(1)-1\ {\bf < 0}$ (in black) and $\mu_5=1(1)5 \  {\bf > 0}$  (in red) of Chebyshev polynomials of second kind (in bleu).}\label{Figure_TT2_Tra_r_5}
\begin{center}
\begin{tabular}{| c | c |}\hline
\includegraphics[width=6.75cm]{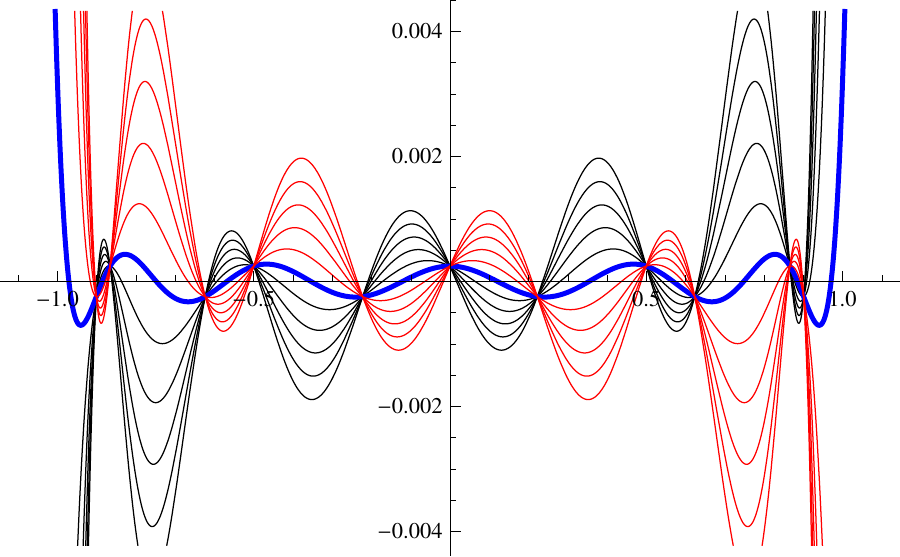} &
\includegraphics[width=6.75cm]{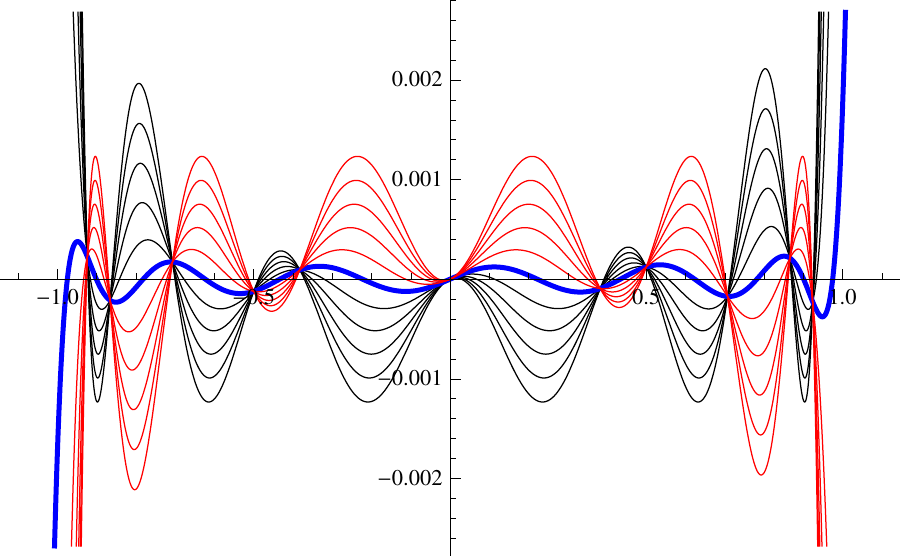}\\
 $n=12$ & $n=13$\\ 
there are no common zeros &  0 is a double common zero \\
all interception points are simple & other interception points are simple  \\ \hline
\includegraphics[width=6.75cm]{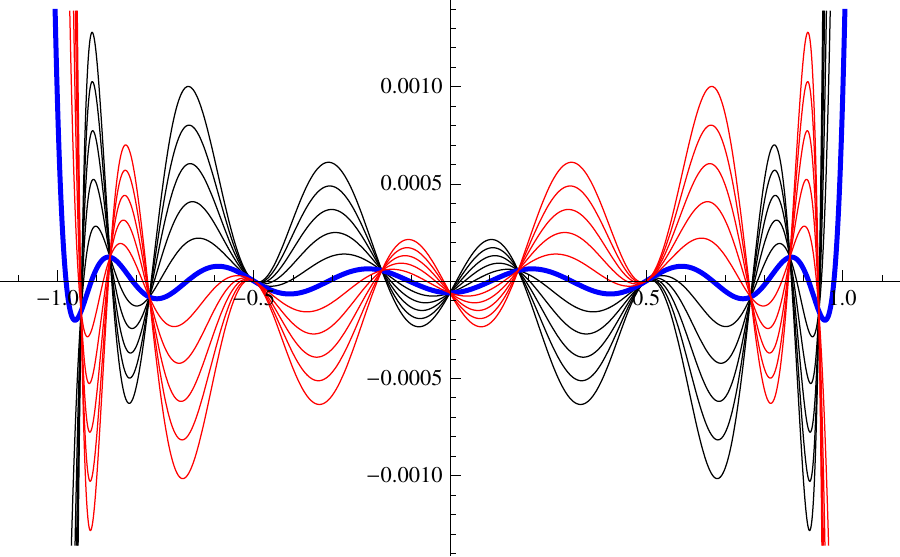} &
\includegraphics[width=6.75cm]{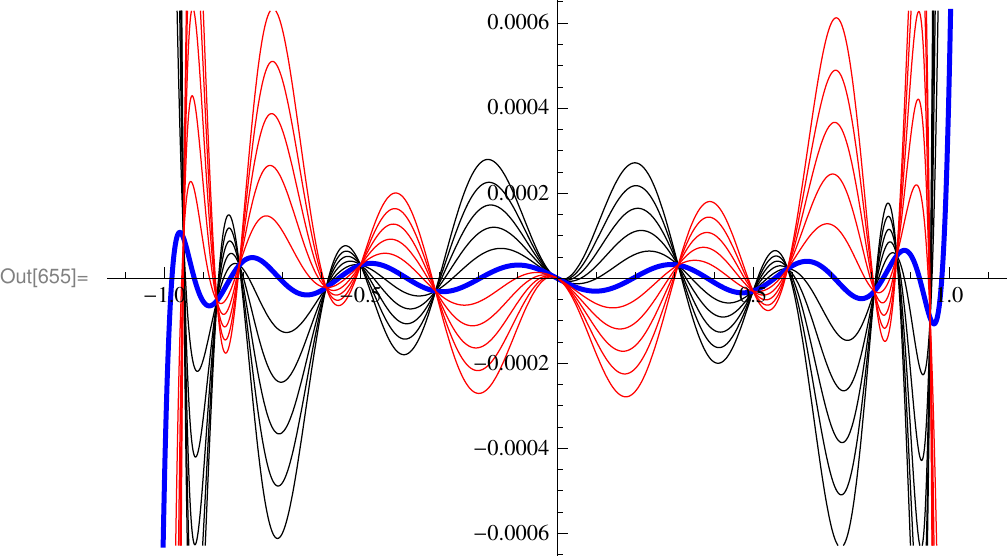}\\
 $n=14$ & $n=15$\\ 
-0.5 and 0.5 are double common zeros  & 0 is a double common zero \\
other interception points are simple & other interception points are simple \\ \hline
\includegraphics[width=6.75cm]{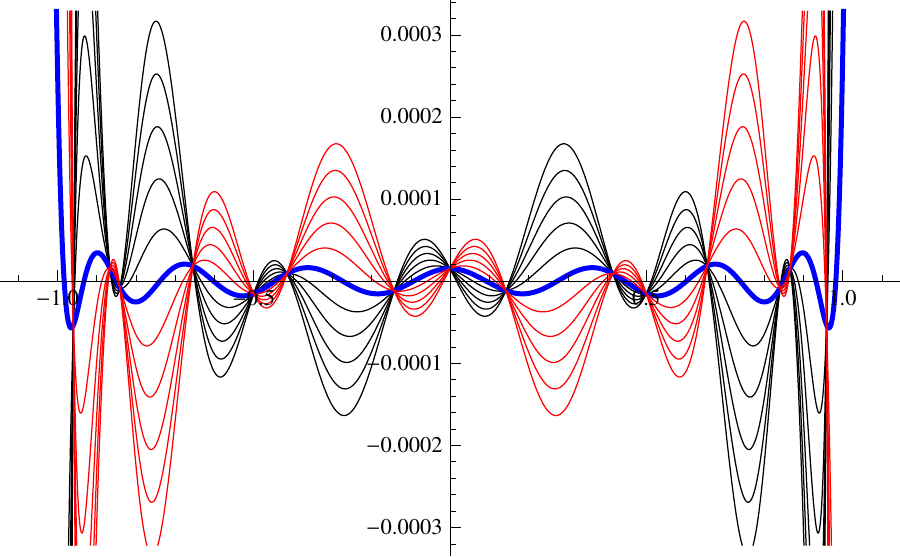} &
\includegraphics[width=6.75cm]{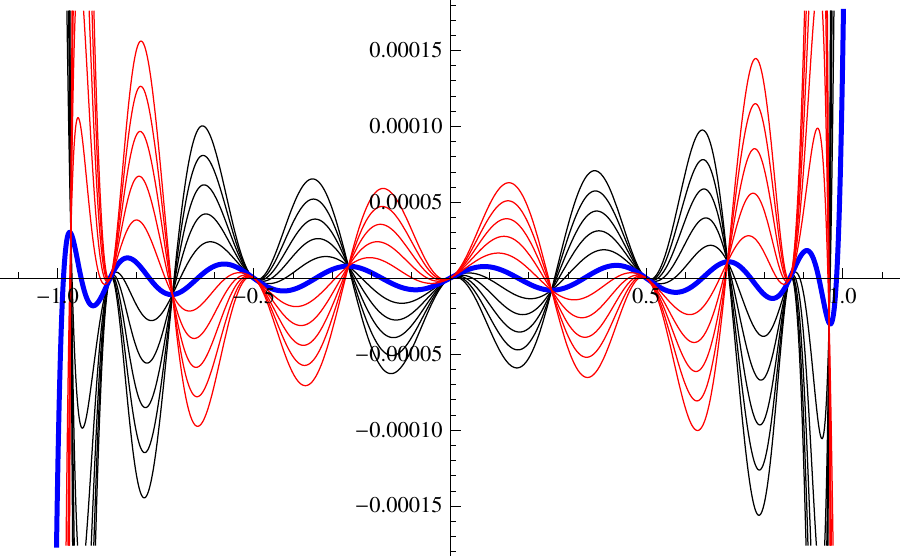}\\
 $n=16$ & $n=17$ \\
 there are no common zeros &  5 double common zeros\\
 all interception points are simple &  all zeros of $P_5(x)$ are d. common zeros \\\hline
\end{tabular}

\hspace{0.5cm}
The zeros of $P_5(x)$ are $-\frac{\sqrt{3}}{2}\approx-0.87,\ -\frac{1}{2},\ 0, \ \frac{1}{2} ,\ \frac{\sqrt{3}}{2}\approx 0.87.$
\end{center}
\end{figure}



\begin{figure}[htbp] 
\caption{Some perturbed of {\bf order 6 by dilatation} with parameters $\lambda'_6=-5(1)-1\ {\bf < 1}$ (in black) and $\lambda_6=3(1)7 \  {\bf > 1}$  (in red) of Chebyshev polynomials of second kind (in bleu).}\label{Figure_TT2_Dil_r_6}
\begin{center}
\begin{tabular}{| c | c |}\hline
\includegraphics[width=6.75cm]{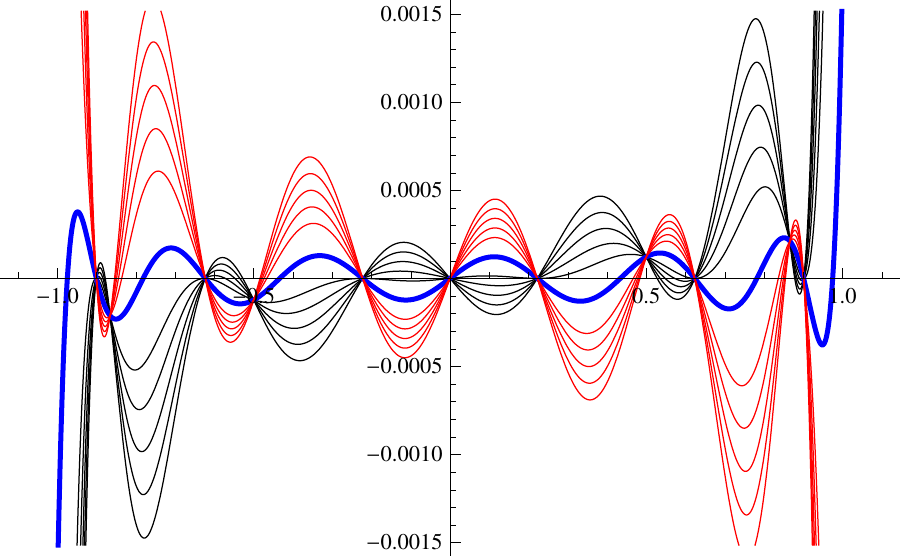} &
\includegraphics[width=6.75cm]{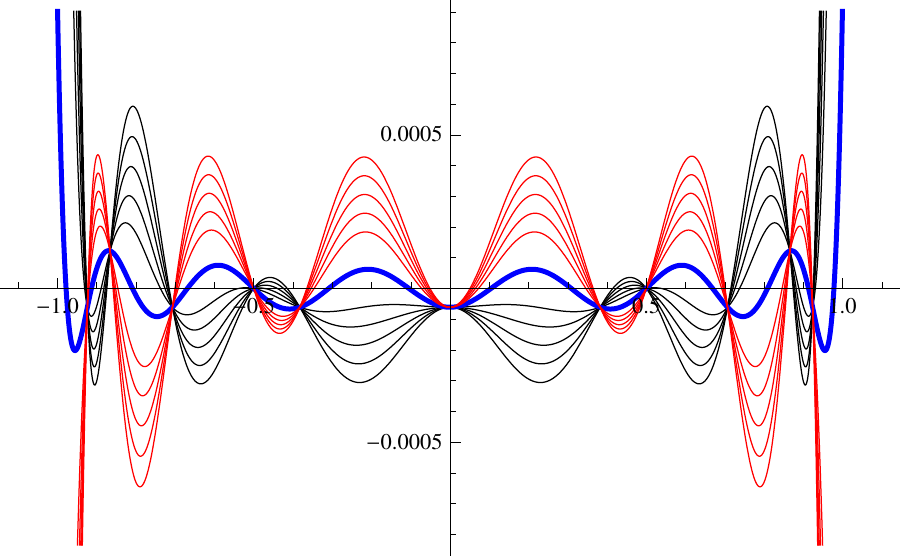}\\
 $n=13$ & $n=14$ ; 2 simple common zeros\\ 
5 simple common zeros &  0 is a double interception point \\
all interception points are simple & other interception points are simple  \\ \hline
\includegraphics[width=6.75cm]{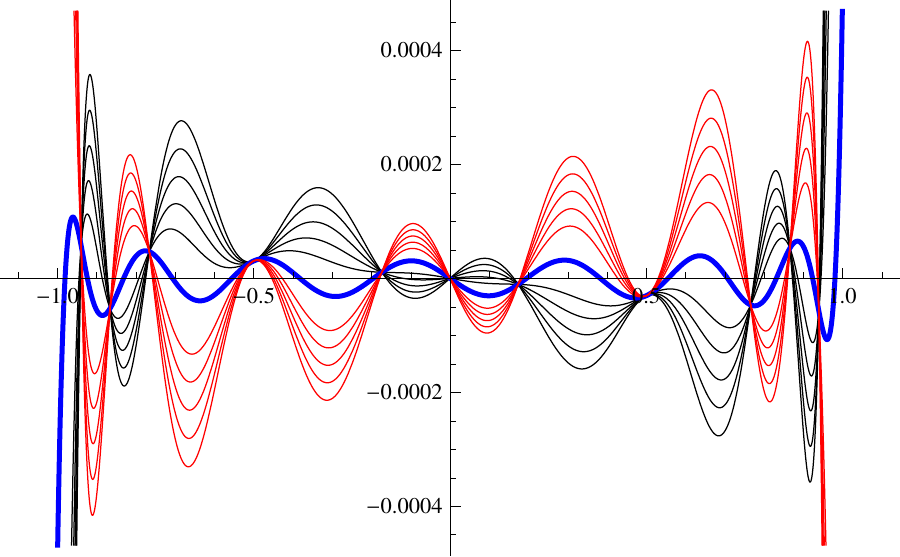} &
\includegraphics[width=6.75cm]{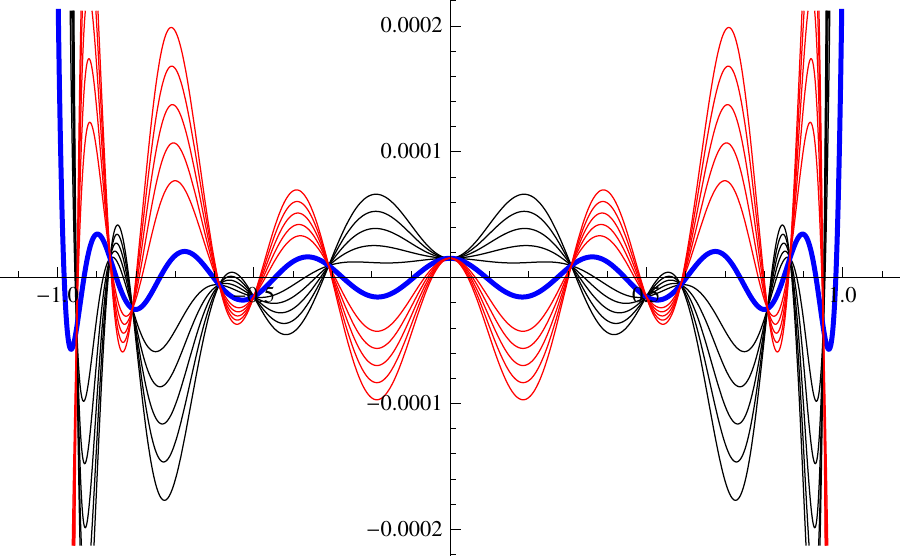}\\
 $n=15$; 0 is a simple common zero & $n=16$; no common zeros\\ 
-0.5 and 0.5 are double inter. p.  & 0 is a double interception point \\
other interception points are simple & other interception points are simple \\ \hline
\includegraphics[width=6.75cm]{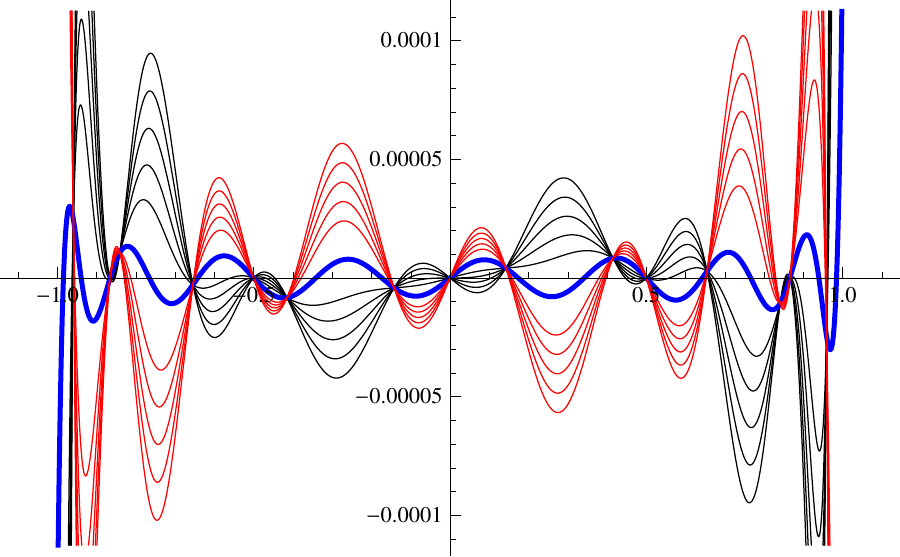} &
\includegraphics[width=6.75cm]{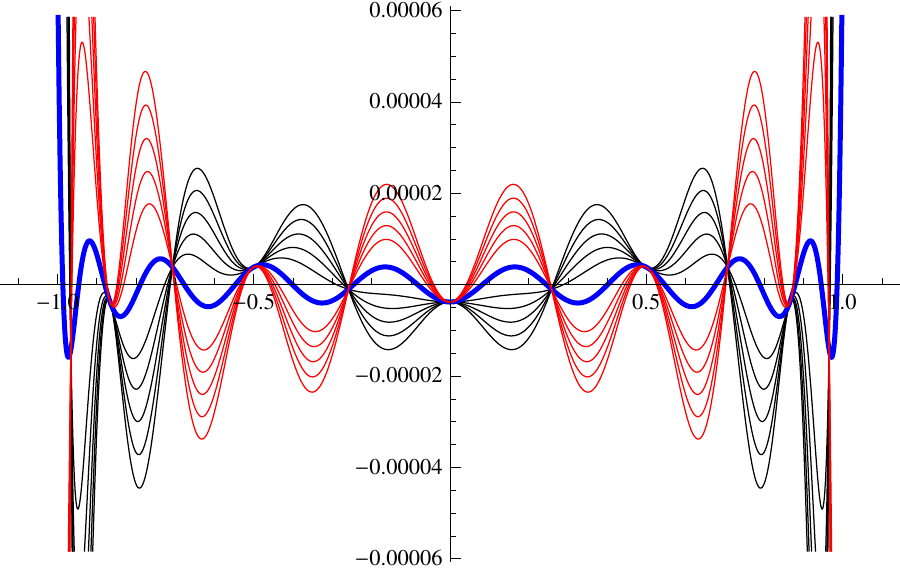}\\
 $n=17$; 5 simple common zeros & $n=18$; no common zeros \\
 all interception points are simple &  all zeros of $P_5(x)$ are double inter. p. \\\hline
\end{tabular}

\hspace{0.5cm}
The zeros of $P_5(x)$ are $-\frac{\sqrt{3}}{2}\approx-0.87,\ -\frac{1}{2},\ 0, \ \frac{1}{2} ,\ \frac{\sqrt{3}}{2}\approx 0.87.$
\end{center}
\end{figure}



\begin{figure}[htbp]
\caption{Some perturbed of {\bf order 5 by translation} with negative parameters 
$\mu'_5=-5(1)-1\ {\bf < 0}$ (in black) and positive parameters, $\mu_5=1(1)5{\bf > 0}$  (in red) of Chebyshev polynomials of second kind (in bleu) of degree $n=17$. 
 All zeros of $P_5(x)$ ($-\frac{\sqrt{3}}{2}\approx-0.87,\ -\frac{1}{2},\ 0, \ \frac{1}{2} ,\ \frac{\sqrt{3}}{2}\approx 0.87$) are double common zeros 
 of $P^{t}_{17}(\mu'_5;5)(x)$, 
$P^{t}_{17}(\mu_5;5)(x)$ and $P_{17}(x)$. There are 11 distinct interception points.
There is a zero of $P^{t}_{17}(\mu'_5;5)(x)$ on the left of -1. There is a zero of $P^{t}_{17}(\mu_5;5)(x)$ on the right of 1. Other zeros are in $[-1,1]$. All zeros are real and simple. Polynomials are not symmetric.}\label{Figure_TT2_Tra_r_5_n_17}
\begin{center}
\begin{tabular}{|c|c|}\hline
\multicolumn{2}{| c |}{\includegraphics[width=13.0cm]{TT2_Tra_r_5_n_17} }\\ \hline
\includegraphics[width=6.5cm]{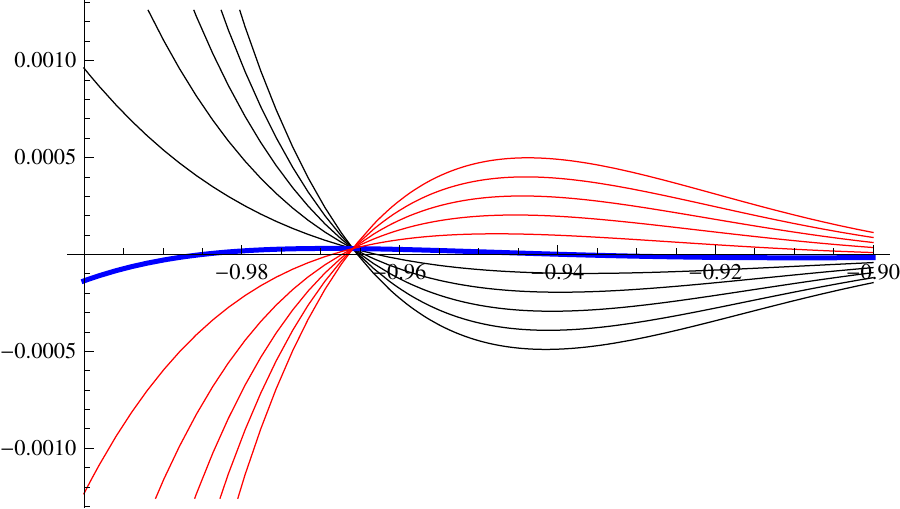} & 
\includegraphics[width=6.5cm]{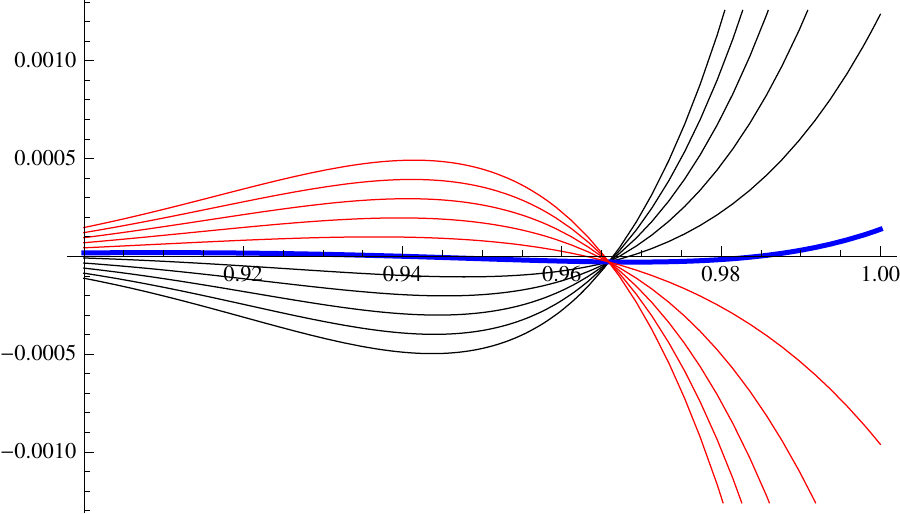} \\ \hline
\includegraphics[width=6.5cm]{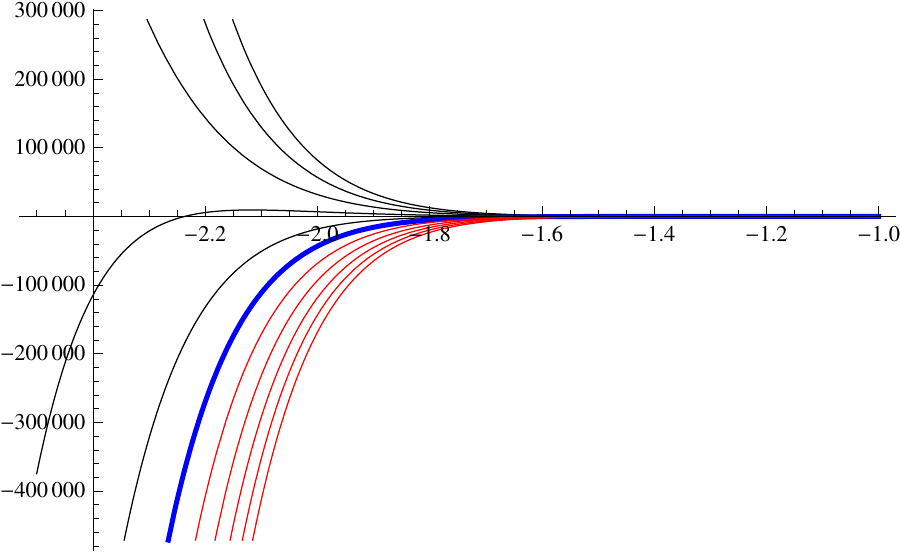} & 
\includegraphics[width=6.5cm]{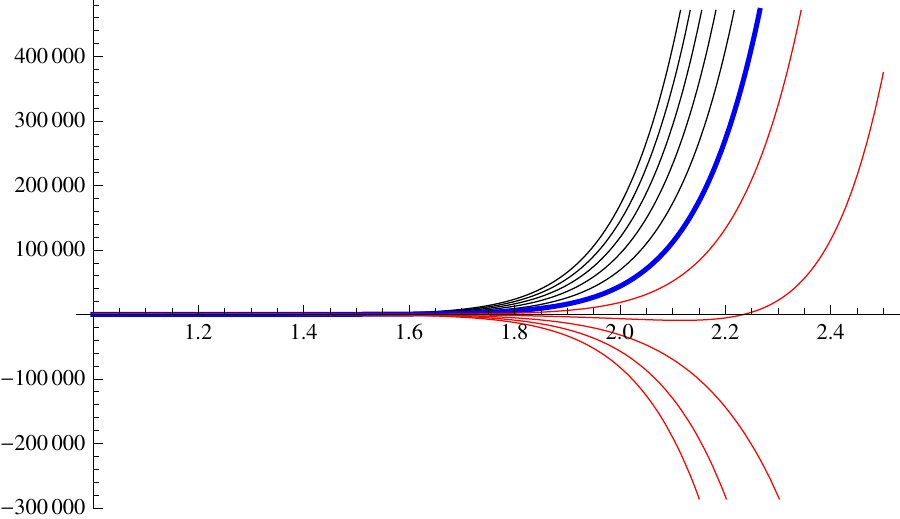} \\ \hline
\end{tabular}
\end{center}
\end{figure}



\begin{figure}[htbp]
\caption{Some perturbed of {\bf order 6 by dilatation} with parameters 
$\lambda'_6=-5(1)-1\ {\bf < 1}$ (in black) and $\lambda_6=3(1)7 \  {\bf > 1}$  (in red) of Chebyshev polynomials of second kind (in bleu) of degree $n=18$.  All zeros of $P_5(x)$ ($-\frac{\sqrt{3}}{2}\approx-0.87,\ -\frac{1}{2},\ 0, \ \frac{1}{2} ,\ \frac{\sqrt{3}}{2}\approx 0.87$) are double interception points 
of $P^{d}_{18}(\lambda'_6;6)(x)$, 
$P^{d}_{18}(\lambda_6;6)(x)$ and $P_{18}(x)$. There are no common zeros. There are 11 distinct interception points. $P^{d}_{18}(\lambda'_6;6)(x)$ has 6 real zeros in 
$[-1,1]$ and it has 6 pairs of complex conjugate zeros. $P^{d}_{18}(\lambda_6;6)(x)$ has 16 real zeros in  $[-1,1]$, it has a zero on the left of -1 and a zero on the right of 1. Polynomials and zeros are symmetric.}\label{Figure_TT2_Dil_r_6_n_18}
\begin{center}
\begin{tabular}{|c|c|}\hline
\multicolumn{2}{| c |}{\includegraphics[width=13.0cm]{TT2_Dil_r_6_n_18}}\\ \hline
\includegraphics[width=6.5cm]{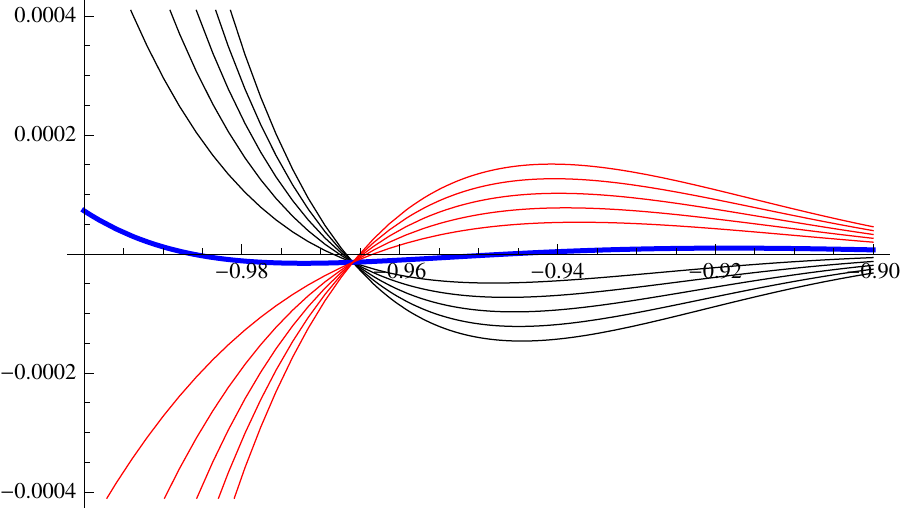} & 
\includegraphics[width=6.5cm]{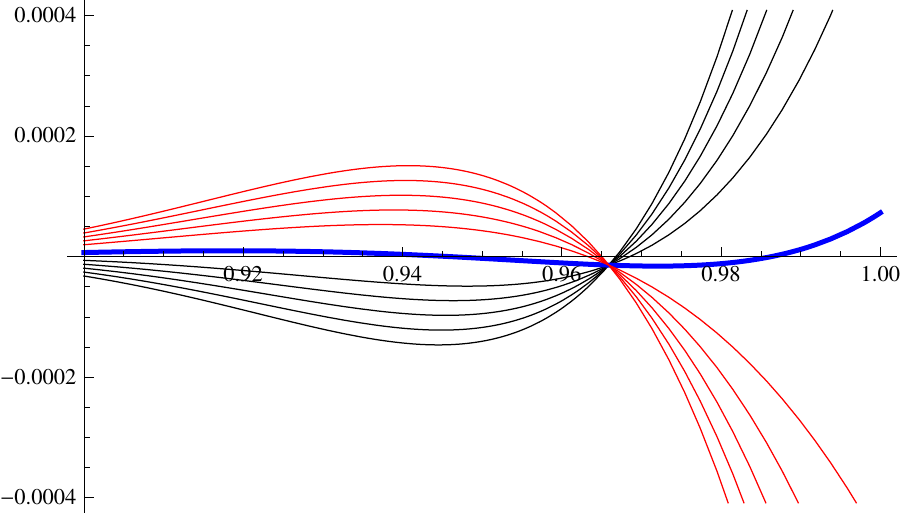} \\ \hline
\includegraphics[width=6.5cm]{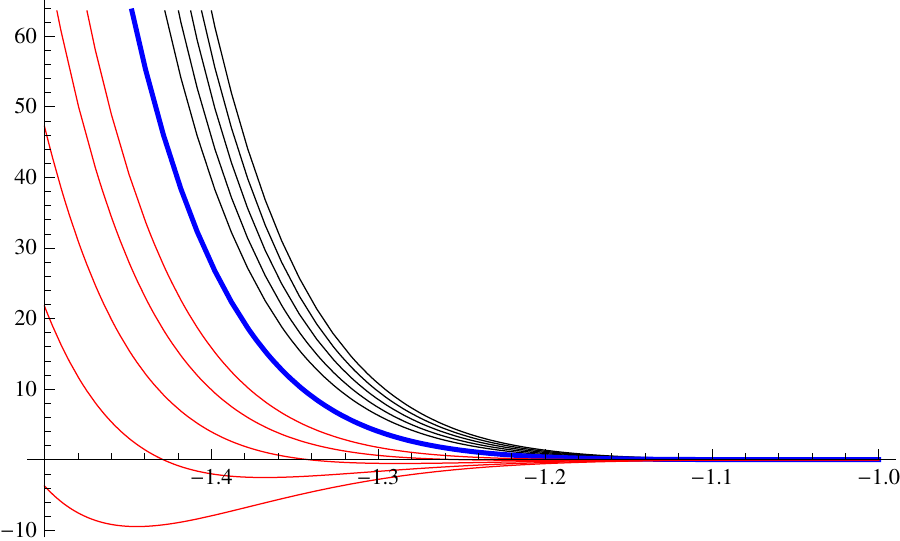} & 
\includegraphics[width=6.5cm]{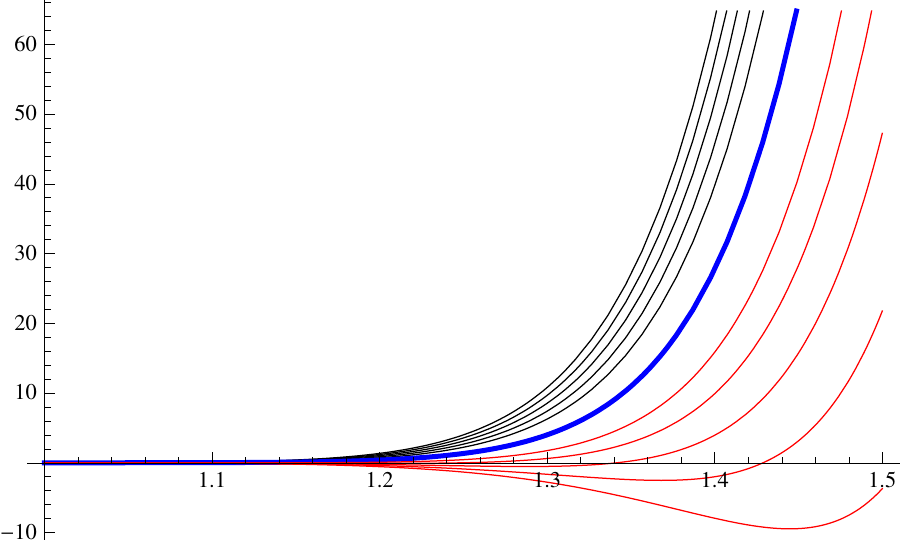} \\ \hline
\end{tabular}
\end{center}
\end{figure}


\end{document}